\documentclass[a4paper,10pt]{amsart}

\usepackage[english]{babel}
\usepackage{dsfont} % \mathds
\usepackage{amssymb} % \upharpoonright
\usepackage{amsthm} % \newtheorem

\theoremstyle{plain}
\newtheorem{thm}{Theorem} \newtheorem{dfn}{Definition} \newtheorem{lem}{Lemma} \newtheorem{exa}{Example} \newtheorem{prop}{Proposition} \newtheorem{cor}{Corollary} \newtheorem{rem}{Remark}

\newcommand{\C}{\mathds{C}} \newcommand{\R}{\mathds{R}}  \newcommand{\N}{\mathds{N}} 
\newcommand{\Disc}{\mathds{D}} \newcommand{\Circle}{\mathds{T}}\newcommand{\cO}{\mathcal{O}}
\newcommand{\Def}{\mathcal{D}} \newcommand{\Range}{\mathcal{R}} \newcommand{\Graph}{\mathcal{G}} \newcommand{\Null}{\mathcal{N}}
 \newcommand{\cF}{\mathcal{F}} \newcommand{\cS}{\mathcal{S}} \newcommand{\cR}{\mathcal{R}} \newcommand{\cT}{\mathcal{T}}   \newcommand{\cX}{\mathcal{X}} \newcommand{\cJ}{\mathcal{J}}
\newcommand{\Hil}{\mathcal{H}} \newcommand{\K}{\mathcal{K}} \newcommand{\ii}{\rm{i}} 
\newcommand{\Komp}{\mathcal{K}} \newcommand{\Abg}{\mathcal{C}} \newcommand{\Be}{\mathbf{B}} \newcommand{\Adj}{\mathcal{L}} \newcommand{\Reg}{\mathcal{R}} \newcommand{\oC}{\mathcal{C}_o} \newcommand{\oCl}{\mathcal{C}'_{o}} \newcommand{\grReg}{\mathcal{R}_{gr}} \newcommand{\Multiplier}{\mathtt{M}} \newcommand{\LeftMultiplier}{\mathtt{LM}} 
\newcommand{\CAlg}[1]{\mathcal{#1}} \newcommand{\sAlg}[1]{\mathcal{#1}} 
\newcommand{\Hom}{\mathtt{Hom}}  
\newcommand{\SP}[2]{\left<#1,#2\right>} \newcommand{\ol}[1]{\overline{#1}}
   
\newcommand{\reg}{\mathtt{reg}} \newcommand{\singsuppr}{\mathtt{sing\mbox{-}supp_r}} 

\title{Unbounded Operators on Hilbert $C^*$-Modules}

\author{Ren\'e Gebhardt}
\address{}
\curraddr{}
\email{}
\thanks{}

\author{Konrad Schm\"udgen}
\address{}
\curraddr{}
\email{}
\thanks{}

\subjclass[2010]{Primary 46 L 08; Secondary  47 D 40, 47 L 05.}

\begin{document}
%\===================/

\begin{abstract} Let $E$ and $F$ be Hilbert $C^*$-modules over a $C^*$-algebra $\CAlg{A}$. New classes of (possibly unbounded) operators $t:E\to F$ are  introduced and investigated. Instead of the density of the domain $\Def(t)$ we only assume  that $t$ is essentially defined, that is, $\Def(t)^\bot=\{0\}$.  Then $t$ has a well-defined adjoint. We call an essentially defined operator $t$  graph regular if its graph $\Graph(t)$ is orthogonally complemented in $E\oplus F$ and  orthogonally closed if $\Graph(t)^{\bot\bot}=\Graph(t)$. A theory  of these operators and related concepts is developed. Various characterizations of graph regular operators are given. A number of examples of graph regular operators are presented ($E=C_0(X)$, a fraction algebra related to the Weyl algebra, Toeplitz algebra, Heisenberg group). A new characterization of affiliated operators with a $C^*$-algebra in terms of resolvents is given.
\end{abstract}

\maketitle

\tableofcontents

  \section{Introduction}

Hilbert  $C^*$-modules are a well established tool in the theory of $C^*$-algebras and their applications. They have been invented by I. Kaplansky \cite{kap} for commutative $C^*$-algebras and by W. Paschke \cite{paschke} and M. Rieffel \cite{rieffel} in the general case. 
Standard textbooks are \cite{lance} and \cite{manuilov}. 
Unbounded operators on Hilbert $C^*$-modules play an important role  for the study of noncompact quantum groups \cite{wor91}, in KK-theory \cite{bj,kuc1} and in noncommutative geometry \cite{gracia}. 

Let $E$ and $F$  be  Hilbert $C^*$-modules over a $C^*$-algebra $\CAlg{A}$. By an operator from $E$ into $F$ we mean  an $\CAlg{A}$-linear and $\C$-linear mapping $t$ of an $\CAlg{A}$-submodule $\Def(t)$ of $E$ into $F$. Such an operator $t$ is called {\it regular} if $t$ is closed,  $\Def(t)$ is dense in $E$,  $\Def(t^*)$ is dense in $F$, and  $I+t^*t$ is a bijection of $E$. Regular operators form the most important class of unbounded operators on Hilbert $C^*$-modules. A nice presentation of their theory can be found in Chapters 9 and 10 of  Lance' book \cite{lance}. Regular operators were invented by S. Baaj \cite{baaj,bj} and extensively studied by S.L. Woronowicz in two seminal papers \cite{wor91,wor92}. Woronowicz considered  the case when $E$ is the $C^*$-algebra $\CAlg{A}$ itself and called the corresponding operators {\it affiliated} with $\CAlg{A}$. That is, the affiliated operators are precisely the  regular  operators on the Hilbert $C^*$-modules $E=\CAlg{A}$. Unbounded operators on Hilbert $C^*$-modules are studied in \cite{hilsum}, \cite{kuc1}, \cite{kuc2}, \cite{fs}, \cite{KaadLesch}, \cite{pal}, \cite{pierrot}. A generalization of regular  operators are the {\it semiregular} operators introduced by A. Pal \cite{pal}.  Note that semiregular operators are always densely defined. 

The aim of the present paper is to introduce several new classes of unbounded operators on $C^*$-modules and  to develop the basics of their theory. 
Let us  briefly explain  the main new concepts by avoiding technical subtleties. Precise definitions and further explanations will be given in the corresponding sections of the text. 
Suppose that  $t$ is an 
operator from $E$ into $F$ such that its domain $\Def(t)$ is essential, that is,  $\Def(t)^\bot=\{0\}$. Here $\Def(t)^\bot$ denotes the 
orthogonal complement of $\Def(t)$  with respect to the $\CAlg{A}$-valued scalar product of $E$.  Then the operator $t$ has a well-defined \emph{adjoint operator} $t^*$.  
An operator $t$ is called \emph{orthogonally closed} if its graph $\Graph (t)=\{(x,tx);t\in \Def(t)\}$ is orthogonally closed, that is, if $\Graph (t)^{\bot\bot}=\Graph(t)$ in $E\oplus F$. 
An orthogonally closed operator $t$ is called \emph{graph regular} if its graph $\Graph(t)$ is orthogonally complemented in $E\oplus F$.  Graph regularity is the most important new concept of operators appearing in this paper. 
This notion is also of interest in the case when the operator $t$ is bounded and only essentially defined. 

It should be  emphasized that all these operators are not necessarily densely defined! Instead we  assume only  that their domains are essential, that is, they have trivial orthogonal complements, and we replace the closure of the graph $\Graph(t)$ in\, $E\oplus F$\, by its double orthogonal complement $\Graph (t)^{\bot\bot}$.

The difference between graph regularity and regularity can be nicely illustrated by the Hilbert $C^*$-module $C_0(X)$ for the $C^*$-algebra $C_0(X)$, where $X$ is a locally compact Hausdorff space. Then both classes of operators  are given by multiplication operators: the regular operators by functions of $C(X)$, but the graph regular operators by functions which are only continuous up to a no-where dense set and for which the modulus goes to infinity in  neighbourhoods of discontinuities (see Theorem \ref{maintheoremcont}). This example shows another essential difference: While a regular operator can be transported  into a densely defined closed operator by each induced representation, for a graph regular operator this is only possible for certain representations. 
To include such phenomena  was one   motivation for studying graph regular operators. We will elaborate this  elsewhere more in detail.

This paper is organized as follows. 
Section \ref{orthognally} contains basic definitions and  facts  on essentially defined operators, their adjoints and  orthogonally closed operators. 

Graph regular operators are introduced and studied in Section \ref{graphregular}. Basic properties of these operators are derived (Theorems \ref{Characterization__orthognal_closable} and \ref{Basics_adjointable}) and  two characterizations of graph regular operators in terms of bounded operators are obtained. The first one,  the  $(a,a_*,b)$-transform, gives a purely algebraic description of graph regular operators  by a triple of adjointable operators (Theorem \ref{aabTransform}). 
The second characterization (Theorems \ref{boundedTransform__tz} and \ref{BoundedTranform_grReg}) concerns the bounded transform. This transform played already a crucial role in Woronowicz' approach to affiliated operators. 
To each graph regular operator $t:E\to F$ we associate a regular operator $t_0:E_0\to F_0$ acting between essential submodules $E_0$ and $F_0$ of $E$ and $F$, respectively. As a byproduct of these considerations we use the notion of essentially defined partial isometries to prove a general result  on the polar decomposition of adjointable operators (Theorem \ref{PolarDecomposition_AdjointableOperator}) and of graph regular operators (Theorem \ref{PolarDecomposition_grReg}). Further, we show that quotients of adjointable operators provide a large source of examples of graph regular operators (Theorems \ref{ClosedQuotient_is_GraphRegular} and \ref{Graph_regular__admissible_pair}). The functional calculus of normal regular operators is extended to normal graph regular operators (Theorem \ref{grReg_FunctionalCalculus}).

In Section \ref{associated} we specialize  to the case when $E$ is the $C^*$-algebra $\CAlg{A}$ itself and $\CAlg{A}$ is faithfully realized on a Hilbert space $\Hil$. 
Then  the regular operators on the Hilbert $C^*$-module $E=\CAlg{A}$  are precisely Woronowicz' affiliated  operators.  
Suppose that $t$ is a densely defined closed operator on $\Hil$. We shall say that $t$ is \emph{associated} with  $\CAlg{A}$ 
if the operators $a_t=(I+t^*t)^{-1}$, $a_{t^*}=(I+tt^*)^{-1}$, and $b_t=t(I+t^*t)^{-1}$ are in the multiplier algebra $\Multiplier(\CAlg{A})$.   A number of results as well as examples and counter-examples on associated and affiliated operators are derived. One of the main new results  (Theorem \ref{Characterisation__affiliation_resolvent}) 
provides a  characterization of affiliated operators in terms of resolvents:
If $\lambda\in  \rho(t)$, then $t$ is affiliated with $\CAlg{A}$\, if and only if $(t-\lambda I)^{-1}\in\Multiplier(\CAlg{A})$ and $(t-\lambda I)^{-1}\CAlg{A} $   and $(t^*-\overline{\lambda}\, I)^{-1}\CAlg{A}$ are dense in $\CAlg{A}$. This seems to be a useful criterion for proving that operators are affiliated, since the resolvent is better understood in operator theory than the bounded transform.

In Sections \ref{Theory_CX} and \ref{examples} we develop 
various classes of  examples. A particular emphasis is on   graph regular operators that are not regular on the corresponding $C^*$-modules.    Section \ref{Theory_CX} and Subsection \ref{graphC_0(X)}  contain a careful treatment of the commutative case $E=C_0(X)$,  where $X$ is a locally compact Hausdorff space. Among others, essentially defined orthogonally closed operators and graph regular operators are characterized in  this case. In Subsection \ref{matrices} we consider some simple examples for matrices over commutative $C^*$-algebras. In Subsection \ref{fractionWeyl} the position operator $Q=x$ and the momentum operator $P=-i\frac{d}{dx}$ become graph regular operators on a $C^*$-algebra obtained from a fraction algebra. In  Subsection \ref{unboundedToeplitz} some unbounded Toeplitz operators are described as graph regular operators on the Toeplitz $C^*$-algebra, while in Subsection \ref{Heisenberg} a graph regular operator on the $C^*$-algebra of  the Heisenberg group is constructed. 
 
Finally, let us fix some notation that will be used in this paper.
Throughout,  $\CAlg{A}$ denotes a $C^*$-algebra, $\sAlg{A}_h:=\{a\in \CAlg{A}: a^*=a\}$ is its hermitian part, and $\CAlg{A}_+:=\{a\in\CAlg{A}|a\geq 0\}$ is its cone of positive elements. If $\Hil$ is a Hilbert  space, we denote by ${\bf B}(\Hil)$ the bounded operators on $\Hil$, by $\mathcal{K}(\Hil)$ the compact operators on $\Hil$, and  by $\mathcal{C}(\Hil)$  the set of densely defined closed operators on $\Hil$. For an operator $t\in\mathcal{C}(\Hil)$, let $\Def(t) $ be its domain on $\Hil$, $\Null(t)$ its null space, $\Range(t)$ its range and $\rho(t)$ its resolvent set.

  \section{Orthogonally closed operators on Hilbert $C^*$-modules}\label{orthognally}

First we  recall a standard definition.
\begin{dfn}
A (right) \emph{pre-Hilbert $C^*$-module} $E$ over $\CAlg{A}$ is a right $\CAlg{A}$-module $E$ equipped with  a map $\SP{.}{.}:E\times E\to\CAlg{A}$ satisfying the following conditions:
\begin{align*}
\lambda(xa) &= (\lambda x)a = x(\lambda a), \quad \lambda\in\C,x\in E,a\in\CAlg{A},\\
\SP{\alpha x+\beta y}{z} &= \alpha\SP{x}{z}+\beta\SP{y}{z}, \quad \alpha,\beta\in\C,x,y,z\in E,\\
\SP{x}{ya} &= \SP{x}{y}a, \quad a\in\CAlg{A},x,y\in E,\\
\SP{x}{y} &= \SP{y}{x}^*, \quad x,y\in E,\\
\SP{x}{x} &\geq 0, \quad\quad x\in E,\\ \SP{x}{x} &= 0 \implies x=0, \quad x\in E.
\end{align*}
A pre-Hilbert $C^*$-module $E$ over $\CAlg{A}$ is called a \emph{Hilbert $C^*$-module} over $\CAlg{A}$, briefly a \emph{Hilbert $\CAlg{A}$-module}, if\, $(E,\|.\|_E)$ is complete, where  $\|.\|_E$ is the norm on $E$ given by $$\|x\|_E:=\|\SP{x}{x}\|_{\CAlg{A}}^{1/2}, \quad  x\in E.$$ 
\end{dfn}

\begin{exa}\label{cstaralgebras}
The $C^*$-algebra  $\CAlg{A}$ itself is a Hilbert $\CAlg{A}$-module over  $\CAlg{A}$  by taking the multiplication as right action and the $\CAlg{A}$-valued scalar product  $\SP{a}{b}:=a^*b$,  $a,b\in E$. In this case $\|a\|_E=\|\SP{a}{a}\|^{1/2}_{\CAlg{A}}=\|a\|_{\CAlg{A}}$ for $a\in E$.
\end{exa}
\smallskip
Suppose that $E$ is a Hilbert $\CAlg{A}$-module. If $x,y\in E$ and $\SP{x}{y}=0$, we write $x\bot y$. For a subset  $F$ of $ E$, the set
$$F^\bot:=\{x\in E|\forall y\in F:\SP{x}{y}=0\}$$
is called the \emph{orthogonal complement} of $F$. Clearly, $F^\bot$ is a submodule of $E$ and $x\bot y$ is equivalent to $y\bot x$, since $\SP{x}{y} = \SP{y}{x}^*$. If $F,G$ are subsets of $E$, then
\begin{align*}
F \subseteq F^{\bot\bot}, \quad F\subseteq G \implies G^\bot\subseteq F^\bot, \quad F^\bot = F^{\bot\bot\bot}.
\end{align*}
We only verify  the last equality $ F^\bot = F^{\bot\bot\bot}$. The first inclusion yields $F^\bot \subseteq (F^{\bot})^{\bot\bot}$. On the other hand, the relation  $F \subseteq F^{\bot\bot}$ implies that $(F^{\bot\bot})^\bot\subseteq F^\bot$. Therefore, $ F^\bot = F^{\bot\bot\bot}$.

Let $F$ and $G$ be subsets of $E$. We set $F+G:=\{f+g|f\in F,g\in G\}$ and write $F\oplus G:=F+G$ if $F\subseteq G^\bot$. Since $F\subseteq G^\bot$ implies $G\subseteq G^{\bot\bot}\subseteq F^\bot$, we always have  $F\oplus G=G\oplus F$. Since $F\subseteq(F^\bot)^\bot$, it is justified to write $F\oplus F^\bot$. 
\begin{dfn}
A subset $F$ of $ E$ is said to be \emph{essential} if $F^\bot=\{0\}$. A submodule $F$ of $E$ is called \emph{orthogonally closed} if $F=F^{\bot\bot}$ and \emph{orthogonally complemented} if $F\oplus F^\bot=E$. 
\end{dfn}
Since $F^\bot$ is always closed in $E$, each orthogonally closed submodule is closed.

\begin{exa}
For the $C^*$-algebra $C_0(X)$ of continuous functions on a locally compact Hausdorff $X$ space vanishing at infinity each closed ideal is of the form
\begin{align*}
\CAlg{I}_\cO &:= \{ f\in C_0(X) | \forall x\in X\setminus \cO:f(x)=0 \}
\end{align*}
for some open subset $\cO\subseteq X$. The mapping $\cO\to\CAlg{I}_O$ is a bijection of the  open subsets  of $X$ onto the closed ideals of\, $C_0(X)$. For $x\in X$ we have $x\in \cO$ if and only if there exists $f\in\CAlg{I}_O$ with $f(x)\neq 0$. The following facts  are easily  verified:
\begin{itemize}
\item $\sAlg{I}_\cO^\bot = \sAlg{I}_{(X\setminus \cO)^\circ}$.
\item $\sAlg{I}_\cO$ is essential in $C_0(X)$ if and only if\, $\cO$ is dense in $X$.
\item $\sAlg{I}_\cO$ is orthogonally closed if and only if\, $\cO$ coincides with the interior of its closure.
\item $\sAlg{I}_\cO$ is orthogonally complemented if\, $\cO$ is closed.
\end{itemize}
\end{exa}

The following simple facts will be often used: If $F,G$ are submodules of $E$, then $\left(F\cap G\right)^\bot \supseteq (F^\bot+G^\bot)^{\bot\bot}$ and $(F+G)^\bot = F^\bot\cap G^\bot$. Further, if $F,G$ are orthogonally closed, then $\left(F\cap G\right)^\bot = (F^\bot+G^\bot)^{\bot\bot}$.

\begin{lem}\label{Essential_Anisotropic}
For any subset $F$ of $E$, the set $F\oplus F^\bot$ is essential.
\end{lem}
\begin{proof}
If $x\in(F\oplus F^\bot)^\bot=F^\bot\cap F^{\bot\bot}$, then $\SP{x}{x}=0$ and hence $x=0$.
\end{proof}

\begin{lem}\label{technic_core}
If $F,G$ are submodules of $E$ with $F\subseteq G$ and $F^\bot\cap G=\{0\}$, then $F^\bot \cap (G\oplus G^\bot) = G^\bot$.
\end{lem}
\begin{proof}
Clearly, $F^\bot\supseteq G^\bot$, so $F^\bot\cap(G\oplus G^\bot)\supseteq F^\bot\cap G^\bot=G^\bot$. Now assume that
$x=g+g^\bot\in F^\bot$ with $g\in G$, $g^\bot\in G^\bot$. But then $x-g^\bot=g\in G\cap F^\bot$, so $g=0$ and $x=g^\bot\in G^\bot$.
\end{proof}

The direct sum $E\oplus F$ of two Hilbert $\CAlg{A}$-modules $E$ and $F$ is a right $\CAlg{A}$-module. If we define a mapping $\SP{.}{.}_{E\oplus F}:(E\oplus F)\times(E\oplus F)\to\sAlg{A}$ by
\begin{align*}
\SP{(e_1,f_1)}{(e_2,f_2)}_{E\oplus F} &:= \SP{e_1}{e_2}_E + \SP{f_1}{f_2}_F ,\quad e_1,e_2\in E,f_1,f_2\in F,
\end{align*}
then this module becomes also a Hilbert $\CAlg{A}$-module.

Now we turn to operators on Hilbert $\CAlg{A}$-modules. By an \emph{operator} $t$ from $E$ into $F$ we mean  a $\C$-linear and $\sAlg{A}$-linear mapping  defined on a right $\sAlg{A}$-submodule $\Def(t)$ of $E$, called the \emph{domain} of $t$. The symbol $t:E\to F$ always denotes an operator from $E$ into $F$. The $\C$-linearity and $\sAlg{A}$-linearity of $t$ mean that
$$t(\lambda x)=\lambda  t(x)\quad{\rm and}\quad t( xa)=t(x)a \quad {\rm for}\quad  \lambda\in\C, x\in\Def(t), a\in\sAlg{A}.$$
For an operator $t:E\to F$, its \emph{null space} $\Null(t):=\{x\in E | tx=0\}$  is a right $\sAlg{A}$-submodule of $E$, its  \emph{range} $\Range(t):=\{tx | x\in\Def(t)\}$ is a right $\sAlg{A}$-submodule of $F$ and its \emph{graph} $\Graph(t):=\{(x,tx)|x\in\Def(t)\}$  is a right $\sAlg{A}$-submodule of $E\oplus F$. As in the case of ordinary Hilbert space operators we say $t$ is \emph{closed} if $\Graph(t)$ is closed in $E\oplus F$ and $t$ is \emph{closable} if there exists an operator $s$ which is a closed extension of $t$. In this case there exists a unique closed operator, denoted by $\overline{t}$ and called the \emph{closure} of $t$, such that $\Graph(\ol{t})=\ol{\Graph(t)}$.

An operator $t:E\to E$ is called \emph{positive} if $\SP{tx}{x}\geq 0$ for all $x\in\Def(t)$.

\begin{dfn}
An operator $t:E\to F$ is called \emph{essentially defined} if\, $\Def(t)$ is an essential submodule of $E$.
\end{dfn}

Suppose that $t:E\to F$ is an essentially defined operator. Set
\begin{align*}
\Def(t^*) &:= \{ y\in F | \exists z\in E : \forall x\in\Def(t) : \SP{tx}{y}_F=\SP{x}{z}_E \}.
\end{align*}
Since $\Def(t)$ is an essential submodule, the element $z\in E$ is uniquely determined by $y$. We  define $t^*y:=z$. Then $t^*:F\to E$  is an operator, called the \emph{adjoint} of $t$, and
\begin{align*}
\SP{tx}{y} = \SP{x}{t^*y} \quad {\rm for}\quad x\in\Def(t), y\in\Def(t^*).
\end{align*}
The operators of the set
\begin{align*}
\Adj(E,F) &:= \{ t:E\to F | \Def(t)=E, \Def(t^*)=F \}
\end{align*}
are called \emph{adjointable}. Note that $\Adj(E):=\Adj(E,E)$ is a unital $C^*$-algebra.

\begin{dfn}
An essentially defined operator $t:E\to E$ is called \emph{symmetric} if $t\subseteq t^*$, and \emph{self-adjoint} if\, $t= t^*$.
\end{dfn}

Let $v:E\oplus F\to F\oplus E$ denote the unitary operator $(x,y)\mapsto(-y,x)$.

\begin{prop}\label{Graph_tadj}
Suppose that $t:E\to F$ is essentially defined. Then:
\begin{enumerate}
\item $\Graph(t^*)=v\Graph(t)^\bot$.
\item $\Null(t^*)=\Range(t)^\bot$.
\item If $t$ is injective and $\Range(t)^\bot=\{0\}$, then $t^*$ is injective and $(t^*)^{-1}=(t^{-1})^*$.
\item $\Graph(t)\oplus v\Graph(t^*)$ is essential.
\end{enumerate}
\end{prop}
The proof of these statements is similar to the Hilbert space case; we omit the details.

\begin{lem}\label{Image_restriction_Adj_essentially_dense}
Let $r:E\to F$ be essentially defined. Suppose that  $\Def(r^*)=F$ and $\Range(r)$ is essential. If\, $\Def\subseteq\Def(r)$ is essential, then $\Range(r\upharpoonright_\Def)$ is essential.
\end{lem}
\begin{proof}
Let $x\in F$ be such that $\SP{x}{ry}=0$ for all $y\in\Def$. Then $\SP{r^*x}{y}=0$ for all $y\in\Def$. Since $\Def^\bot=\{0\}$ by assumption, we conclude that $r^*x=0$. That is, $x\in\Null(r^*)=\Range(r)^\bot=\{0\}$.
\end{proof}

\begin{prop}\label{Adjoint__Addition_Compositon}
Let $t,t_1,t_2$ be essentially  defined operators from $E$ into $F$ and let $s$ be an essentially defined operator from $F$ into $G$. Then:
\begin{enumerate}
\item $t_1\subseteq t_2$ implies $t_1^*\supseteq t_2^*$.
\item If $t_1+t_2$ is essentially defined, then $(t_1+t_2)^*\supseteq t_1^*+t_2^*$. If $\Def(t_2)\subseteq\Def(t_1)$ and $\Def(t_1^*)=E$, then $t_1+t_2$ is essentially defined and $(t_1+t_2)^*=t_1^*+t_2^*$.
\item If $st$ is essentially defined, then $(st)^*\supseteq t^*s^*$. If $\Range(t)\subseteq\Def(s)$ and $\Def(s^*)=G$, then $st$ is essentially defined and $(st)^*=t^*s^*$.
\item If $t$ is injective, $\Def(s)\subseteq\Range(t)$ and $\Def((t^{-1})^*)=F$, then $st$ is essentially defined and $(st)^*=t^*s^*$.
\end{enumerate}
\end{prop}
\begin{proof}
Assertions (1)--(3) are shown by simple  computations. 

We prove (4). By Lemma  \ref{Image_restriction_Adj_essentially_dense}, the domain $\Def(st)=t^{-1}\Def(s)$ is essential. For $x\in\Def((st)^*)$ and $y\in\Def(s)$ we derive
\begin{align*}
\SP{sy}{x} &= \SP{(st)t^{-1}y}{x} = \SP{t^{-1}y}{(st)^*x} = \SP{y}{(t^{-1})^*(st)^*x}.
\end{align*}
Therefore, $x\in\Def(s^*)$ and $s^*x=(t^{-1})^*(st)^*x=(t^*)^{-1}(st)^*x\in\Def(t^*)$. That is, $\Def((st)^*)\subseteq\Def(t^*s^*)$. Now (3) completes the proof.
\end{proof}

\begin{dfn}
An essentially defined operator $p$ is a \emph{projection} if\, $p=p^2=p^*$.
\end{dfn}

The next proposition characterizes projections and shows that they are in one-to-one correspondence to orthogonally closed submodules.

\begin{prop}\label{Characterization_Projectors}
For a submodule $G$ of $E$ we define an operator $p_G:E\to E$ by
\begin{align*}
\Def(p_G) &:= G \oplus G^\bot, \quad p_G(x+y) := x,\quad x\in G,y\in G^\bot.
\end{align*}
Then $p_G$ is essentially defined and $p_G=p_G^2\subseteq p_G^*=p_{G^{\bot\bot}}$. In particular, $p_G=(p_G)^*$ is a projection if and only if\, $G$ is orthogonally closed.

Conversely, if $p$ is a projection, then $p=p_G$ for some orthogonally closed submodule $G$ of $E$.
\end{prop}
\begin{proof}
By Lemma \ref{Essential_Anisotropic}, $p:=p_G$ is essentially defined. Obviously,  $p^2=p\subseteq p^*$. Let $z\in\Def(p^*)$. Then there exists $w\in E$ such that
\begin{align*}
\SP{x}{z} = \SP{p(x+y)}{z} = \SP{x+y}{w}, \quad x\in G,y\in G^\bot.
\end{align*}
Equivalently, $\SP{x}{z-w}=\SP{y}{w}$ for all $x\in G$, $y\in G^\bot$, so both sides are zero. The latter is equivalent to $z-w\in G^\bot$ and $w\in G^{\bot\bot}$, that is, $z\in G^\bot\oplus G^{\bot\bot}$. Therefore, since $w=p^*z$, the operator $p^*$ is given by
\begin{align*}
\Def(p^*) &= G^{\bot\bot} \oplus G^\bot, \quad p^*(x+y) = x, \quad x\in G^{\bot\bot},y\in G^\bot.
\end{align*}
This completes the proof of the first half. 

Now assume that $p=p^2=p^*$ and let $G:=\Range(p)\subseteq\Def(p)$. For $x=py\in G$ we have $px=p^2y=py=x$. For $x\in G^\bot$, we obtain $\SP{x}{py}=0=\SP{0}{y}$ for $y\in\Def(p)$. Therefore, $x\in\Def(p^*)=\Def(p)$ and $px=p^*x=0$. Thus, we have shown that $p_G\subseteq p$. Further, the assumption $p=p^2$ implies that $\Range(p)=\Null(1-p)$. Hence
\begin{align*}
G &= \Range(p) = \Null(1-p) = \Null(1-p^*) = \Range(1-p)^\bot.
\end{align*}
Therefore, $G=G^{\bot\bot}$ and $p_G$ is self-adjoint. Since $p$ is also self-adjoint and $p_G\subseteq p$, we conclude  that  $p=p_G$.
\end{proof}

\begin{dfn}\label{defcore}
Let $t:E\to F$ be an operator. A submodule $\Def$ of $\Def(t)$ is called an \emph{essential core} for $t$ if $\Graph(t\upharpoonright_{\Def})^\bot=\Graph(t)^\bot$.
\end{dfn}

By Lemma \ref{Graph_tadj} this is equivalent to the relation $(t\upharpoonright_{\Def})^*=t^*$ if $\Def$ is essential. Clearly, $\Def$ is an essential core for\, $t$ if and only if\, $\Graph(t\upharpoonright_\Def)^\bot\subseteq\Graph(t)^\bot$, or equivalently, if  the sum $\Graph(t\upharpoonright_\Def)^\bot+\Graph(t)$ is an orthogonal sum\, $\Graph(t\upharpoonright_\Def)^\bot\oplus\Graph(t)$.

\begin{exa}\label{Core_Adjointable}
Let $b\in\Adj(E,F)$ and let $\Def$ be an essential submodule of $E$. Then $\Def$ is an essential core for $b$, since $(b\upharpoonright_\Def)^*\supseteq b^*$ and $b^*$ is everywhere defined.
\end{exa}

\begin{dfn}
A (not necessarily essentially defined) operator $t:E\to F$  is \emph{orthogonally closed} if $\Graph(t)$ is orthogonally closed, that is, if\, $\Graph(t)^{\bot\bot}=\Graph(t)$, and \emph{orthogonally closable} if\, $\Graph(t)^{\bot\bot}=\Graph(s)$ for some (orthogonally closed) operator $s$.
\end{dfn}

By Proposition \ref{Graph_tadj} the adjoint is always orthogonally closed.

\begin{thm}\label{Characterization__orthognal_closable}
Let $t:E\to F$ be essentially defined. Then $\Def(t^*)$ is an essential submodule of $F$ if and only if $t$ is orthogonally closable.
\end{thm}
\begin{proof}
Suppose  that $\Def(t^*)$ is an essential submodule. Then $t^{**}$ exists and it follows easily that $t\subseteq t^{**}$. Applying Lemma \ref{Graph_tadj} twice, first to $t$ and then to $t^*$, we obtain  $\Graph(t^{**})=\Graph(t)^{\bot\bot}$. Hence $t^{**}$ is orthogonally closed and $t$ is orthogonally closable.

Now suppose that $\Graph(t)^{\bot\bot}=\Graph(s)$ for some (essentially  defined) operator $s$.  Let $z\in\Def(t^*)^\bot$. Then $\SP{(-z,0)}{(y,t^*y)}=-\SP{z}{y}+\SP{0}{t^*y}=0$ for all $y\in\Def(t^*)$, so $(z,0)\in v\Graph(t^*)^\bot=\Graph(t)^{\bot\bot}=\Graph(s)$ and hence $z=0$, since $s$ is an operator.
\end{proof}

Let us define the sets
\begin{align*}
\oCl(E,F) &:= \{ t:E\to F | \Def(t)^\bot=\{0\}, \Def(t^*)^\bot=\{0\} \}, \quad \oCl(E) := \oCl(E,E),\\
\oC(E,F) &:= \{t\in\oCl(E,F) | t ~~ {\rm is~ orthogonally~ closed}\}, \quad  \oC(E) := \oC(E,E).
\end{align*}
Then, by Theorem \ref{Characterization__orthognal_closable},\, $\oCl(E,F)$ is the set of essentially defined operators $t:E\to F$ that are orthogonally closable. 

\begin{thm}\label{Basics_adjointable}
\begin{enumerate}
\item  Suppose that $t\in\oCl(E,F)$. Then we have:
\begin{enumerate}
\item $t\subseteq t^{**}$,  $t^* = t^{***}$, and\, $\Graph(t)^{\bot\bot}=\Graph(t^{**})$.
\item $t=t^{**}$ if and only if\, $t\in\oC(E,F)$.
\item $\Null(t^{**})$ is orthogonally closed.
\item If\, $\Range(t)$ and $\Range(t^*)$ are essential, then $t$ and $t^{**}$ are injective, $t^{-1}$ is essentially defined and orthogonally closable, and $(t^{**})^{-1}=(t^{-1})^{**}$.
\end{enumerate}
\item Suppose that $t\in\oC(E,F)$. Then:
\begin{enumerate}
\item $\Null(t)$ is orthogonally closed. 
\item If $\Range(t)$ and $\Range(t^*)$ are essential, then $t$ is injective and $t^{-1}$ is essentially defined and orthogonally closed.
\end{enumerate}
\end{enumerate}
\end{thm}
\begin{proof}
(1a): By Theorem \ref{Characterization__orthognal_closable}, $t^{**}$ exists. Applying Proposition \ref{Graph_tadj} several times we obtain $\Graph(t^{**})=v\Graph(t^*)^\bot=\Graph(t)^{\bot\bot}$, so that $t\subseteq t^{**}$. This implies that $t^{***}$ exists. Using  Proposition \ref{Graph_tadj} once again we get $\Graph(t^{***})=v\Graph(t)^{\bot\bot\bot}=v\Graph(t)^\bot=\Graph(t^*)$, hence $t^{***}=t^*$.

(1b) follows at once from the equality $\Graph(t)^{\bot\bot}=\Graph(t^{**})$. 

(1c): We derive $\Null(t^{**})^{\bot\bot}=\Range(t^*)^{\bot\bot\bot}=\Range(t^*)^\bot=\Null(t^{**})$. 

(1d): By Proposition \ref{Graph_tadj},  $t^*$ is injective and $(t^*)^{-1}=(t^{-1})^*$. By assumption, $\Range(t^*)$ is essential.  Therefore, using  Proposition \ref{Graph_tadj} again we conclude that $t^{**}$ is injective and $(t^{**})^{-1}=((t^*)^{-1})^*=(t^{-1})^{**}$. 

(2a) and (2b) follow from (1c) and (1d), respectively.
\end{proof}

The operator $t^{**}$ is called the \emph{orthogonal closure} of the orthogonally closable operator $t$.

\begin{lem}\label{Composition__orthoClosed_Adjoint}
Let $t\in\oC(F,G)$ and $b\in\Adj(E,F)$. Suppose   that the operator\, $tb$ is essentially defined. Then\, $tb\in\oC(E,G)$ and $(tb)^*=(b^*t^*)^{**}$.
\end{lem}
\begin{proof}
From Proposition \ref{Adjoint__Addition_Compositon} it follows that $(tb)^*\supseteq b^*t^*$ and the latter is essentially defined. Therefore,  $tb\subseteq(tb)^{**}\subseteq(b^*t^*)^*=t^{**}b^{**}=tb$. This implies that $(tb)^*=(b^*t^*)^{**}$.
\end{proof}

\begin{dfn}
An operator $t\in\oCl(E)$ is called \emph{essentially self-adjoint} if $t^*= t^{**}$. An operator $t\in\oC(E,F)$ is called normal if $t^*t=tt^*$.
\end{dfn}

\begin{prop}
Let $t\in\oCl(E,F)$ and let $\Def$ be a submodule of $\Def(t)$. If $\Def$ is an essential core for $t$, then $\Def$ is an essential submodule of $E$ and
\begin{align*}
\left(t\upharpoonright_\Def\right)^* &= t^*, \quad t \subseteq \left(t\upharpoonright_\Def\right)^{**}=t^{**}.
\end{align*}
In particular, $\Def$ is an essential core for $t^{**}$ and if\, $t$ is orthogonally closed, then $t=\left(t\upharpoonright_\Def\right)^{**}$.
\end{prop}
\begin{proof}
Define an operator $p_E:E\oplus F\to E$ by $p_E(x,y):=x$ for $x\in E,y\in F$. It is straightforward to check  that the adjoint operator $p_E^*:E\to E\oplus F$ acts  by $p_E^*(x)=(x,0)$ for $x\in E$. Thus,  $p_E\in\Adj(E\oplus F,E)$. Now, using Proposition \ref{Adjoint__Addition_Compositon},(3) for the second equality,  Proposition \ref{Characterization_Projectors} for the third equality, and the relations $\Graph(t\upharpoonright_\Def)^{\bot\bot}=\Graph(t)^{\bot\bot}=\Graph(t^{**})$ for the fourth equality, we  derive
\begin{align}\label{dbot}
\Def^\bot &= \Range(p_Ep_{\Graph(t\upharpoonright_\Def)})^\bot = \Null(p_{\Graph(t\upharpoonright_\Def)}^*p_E^*) = \Null(p_{\Graph(t\upharpoonright_\Def)^{\bot\bot}}p_E^*) = \Null(p_{\Graph(t^{**})}p_E^*).
\end{align}
Let $x\in \Null(p_{\Graph(t^{**})}p_E^*)$. Then  $x\in\Def(p_{\Graph(t^{**})}p_E^*)$ and $(x,0)\in\Graph(t^{**})\oplus v\Graph(t^*)$. Hence there exist
elements $y\in\Def(t^{**})$ and $z\in\Def(t^*)$ such that $x=y+t^*z$ and $0=t^{**}y-z$. Since $(0,0)=p_{\Graph(t^{**})}p_E^*x=p_{\Graph(t^{**})}(x,0)=(y,t^{**}y)$, we get $y=0$, so $z=0$ and $x=0$.  Thus, 
$\Null(p_{\Graph(t^{**})}p_E^*)=\{0\}$. Hence $\Def$ is essential by (\ref{dbot}). The other statements are easily verified.
\end{proof}

Some algebraic properties on orthogonally closable operators are collected in the next proposition. The proofs are straightforward and  we omit the details. 

\begin{prop}\label{OrthoClosable__Addition_Composition}
Let $t\in\oCl(E,F)$.
\begin{enumerate}
\item If $s\in\Adj(E,F)$, then $t+s\in\oCl(E,F)$.
\item If $s\in\Adj(F,G)$ is injective with $s^{-1}\in\Adj(G,F)$, then $st\in\oCl(E,G)$.
\item If $s\in\Adj(G,E)$ is injective with $s^{-1}\in\Adj(E,G)$, then $ts\in\oCl(G,F)$.
\end{enumerate}
All these statements remain valid if\, $\oCl$ is replaced by $\oC$.
\end{prop}

\begin{rem}\label{firstremark}
In this Remark we want to emphasize that the  theory developed so far is valid in a much more general setting. For this we introduce some definitions. 

Let  $A$ be a complex $*$-algebra. A \emph{ $*$-bimodule} for $A$ is a bimodule $X$ for $A$ equipped with an involution (that is, an antilinear mapping $x\to x^*$ of $X$ such that $(x^*)^*=x$ for $x\in X$) satisfying $(ax)^*=x^*a^* $ and $(xa)^*=a^*x^*$ for  $ a\in A, x\in X.$

\begin{dfn}\label{defquadraticspace}

A {\rm quadratic $*$-space} over $A$ is a triple $(E,X,\SP{.}{.})$ of a right $A$-module $E$, a $*$-bimodule $X$ for $A$ and a map $\SP{.}{.}:E\times E\to X$ such that for  $\alpha,\beta\in\C$, $x,y,z,x_1,\dots,x_n\in E$, $a\in A$:
\begin{align*}
\SP{\alpha x+\beta y}{z} &= \alpha\SP{x}{z}+\beta\SP{y}{z},\\
\SP{x}{ya} &= \SP{x}{y}a,\\
\SP{x}{y} &= \SP{y}{x}^*,\\
\sum\nolimits_{j=1}^n \SP{x_j}{x_j}&=0  \quad {\rm implies}\quad x_1=\dots=x_n=0. 
\end{align*}
\end{dfn} 

All preceding notions and results on  essentially defined operators and their adjoints, on orthogonally closed operators and on graph regular operators remain valid for  quadratic $*$-spaces rather than Hilbert $\CAlg{A}$-modules. Indeed, an inspection of the definitions and proofs shows that only the axioms from Definition \ref{defquadraticspace} are needed. The same is true for Theorem \ref{Basics__graph_regular}, Theorem \ref{aabTransform}  and Proposition \ref{grReg__Morphism}. However, from now on $C^*$-algebra properties are essentially used.
\end{rem}

\section{Operators on the commutative $C^*$-algebra $C_0(X)$}\label{Theory_CX}

Throughout this section we suppose that $X$ is a locally compact Hausdorff space. We consider the $C^*$-algebra $C_0(X)$ of continuous functions on $X$ vanishing at infinity as a Hilbert $C^*$-module $E$ (see Example \ref{cstaralgebras}) and study operators 
 $t:E\to E$. The main aim of this section is to describe  orthogonally closed operators on $E$ and to  show that $\oC(E)$ consists  of multiplication operators. 
 For this reason we  investigate  multiplication operators   in detail. To state the results we introduce some notation that is inspired by \cite[Section 6]{KaadLesch}.

For a function $m:X\to\C$ we set
\begin{align*}
\reg(m) := &\{ x\in X | m\text{ is continuous in a neighborhood of } x \},\\
\reg_b(m) := &\left\{ x\in \partial\reg(m) | \exists U\subseteq\ol{\reg(m)} \text{ open},\tilde{m}: U\to\C \text{ continuous}, \right.\\
&\left. \text{ with } x\in U, \tilde{m}\equiv m \text{ on } U\cap\reg(m)\right\},\\
\reg_\infty(m) := &\left\{ x\in \partial\reg(m) | \exists U\subseteq\ol{\reg(m)} \text{ open},\tilde{m}: U\to\ol{\C} \text{ continuous},\right.\\
&\left. \text{ with } x\in U, \tilde{m}\equiv m \text{ on } U\cap\reg(m), \tilde{m}(x)=\infty \right\},\\
\singsuppr(m) := &\partial \reg(m) \setminus (\reg_b(m)\cup\reg_\infty(m)).
\end{align*}

Clearly, $\reg(m)$ is the largest open set on which $m$ is continuous. Further, $\reg(m)\cup \reg_b(m)$ is the largest open set contained in $\ol{\reg(m)}$ on which $m$ restricted to $\reg(m)$ has a (indeed unique) continuous $\C$-valued extension. Finally, $\reg(m)\cup\reg_b(m)\cup\reg_\infty(m)$ is the largest open set contained in $\ol{\reg(m)}$ on which $m$ restricted to $\reg(m)$ has a (unique) continuous $\ol{\C}$-valued extension. In particular, $\reg(m)\cup\reg_b(m)\cup\reg_\infty(m)$ is contained in $(\ol{\reg(m)})^\circ$. The space $X$ is a disjoint union
\begin{align*}
X = \reg(m) \cup \underbrace{\reg_b(m) \cup \reg_\infty(m) \cup \singsuppr(m)}_{\partial\reg(m)} \cup \underbrace{(X\setminus\ol{\reg(m)})}_{(X\setminus\reg(m))^\circ}.
\end{align*}

By $F_0(X)$ we denote the set of functions on $X$ vanishing at infinity. We note $E=C_0(X)=\{f\in F_0(X)|\reg(f)=X\}$.
\\

Let us briefly summarize the results that will be obtained in this section. For an arbitrary function $m$ on $X$ a multiplication operator $t_m$ is defined and we introduce an equivalence relation to characterize those functions  giving the same operator (Lemma \ref{Domain_Equivalence__CommutativeCase}). Then we show that $t_m$ is essentially defined if and only if $\reg(m)$ is dense in $X$ and that in this case $t_m$ is already orthogonally closed and its adjoint is $t_{\ol{m}}$ (Theorem \ref{MultiplicationOperator__Essential_ssupp}). Further, we prove that $t\in\oC(E)$ if and only if $t=t_m$ for some function $m:X\to\C$ (Theorem \ref{maintheorem}). In Example \ref{Multiplier__different_Dm_Dm*} we give a normal operator $t$ for which  the domains $\Def(t)$ and $\Def(t^*)$ are different.

For $m,\tilde{m}:X\to\C$ we write
\begin{align*}
m\simeq\tilde{m} \quad \Leftrightarrow \quad & \ol{\reg(m)}=\ol{\reg(\tilde{m})} \text{ and } m \equiv \tilde{m} \text { on } \reg(m)\cap \reg(\tilde{m}).
\end{align*}
In Lemma \ref{Basics_EquivalenceRelation}(2) it will be shown that $\simeq$ is an equivalence relation. To simplify the notation we associate to each function $m:X\to\C$ a function $\hat{m}$ on $X$ defined by
\begin{align*}
\hat{m}(x) &= \begin{cases}
               m(x) &, x\in X\setminus \reg_b(m)\\
               \tilde{m}(x) &, x\in \reg_b(m)
              \end{cases},
\end{align*}
where $\tilde{m}$ is one of the functions appearing in the definition of $\reg_b(m)$. In fact, any two  of those functions  have the same values at $x$, since $x$ is in the boundary of $\reg(m)$
and the two continuous functions coincide on this set. 
Hence $\hat{m}$ is well-defined.  Lemma \ref{Basics_EquivalenceRelation}(3) shows that $\hat{m}\simeq m$.

\begin{lem}\label{Basics_EquivalenceRelation}
\begin{enumerate}
\item Let $m_1,m_2:X\to\C$. If $m_1\simeq m_2$, then $\reg_\infty(m_1) = \reg_\infty(m_2)$ and\, $\singsuppr(m_1) = \singsuppr(m_2)$.
\item $\simeq$ is an equivalence relation on the set of functions from $X$ to $\C$.
\item If $m:X\to\C$, then $\hat{m}\simeq m$, $\reg(\hat{m})=\reg(m)\cup\reg_b(m)$, $\reg_b(\hat{m})=\emptyset$.
\end{enumerate}
\end{lem}
\begin{proof}
Note that the intersection of two open and dense sets  is again open and dense.

(1):  From $R:=\ol{\reg(m_1)}=\ol{\reg(m_2)}$ we conclude that $\reg(m_1)\cap\reg(m_2)$ is dense in $R$. Since $m_1$ and $m_2$ are equal and continuous on this set, any $x$ in $R$ is in $\reg(m_1)\cup\reg_b(m_1)$ if and only if it is in $\reg(m_2)\cup\reg_b(m_2)$. By the same argument any $x$ in $R$ is in $\reg(m_1)\cup\reg_b(m_1)\cup\reg_\infty(m_1)$ if and only if it is in $\reg(m_2)\cup\reg_b(m_2)\cup\reg_\infty(m_2)$. This proves (1).

(2): Obviously $\simeq$ is reflexive and symmetric, so it remains to show transitivity. Let $m_1\simeq m_2$ and $m_2\simeq m_3$. Clearly, $\ol{\reg(m_1)}=\ol{\reg(m_2)}=\ol{\reg(m_3)}$. Arguing as above, $\reg(m_1)\cap\reg(m_2)\cap\reg(m_3)$ is open and dense in the latter set. Again by continuity of $m_1$ and $m_3$ on $\reg(m_1)\cap\reg(m_3)$, these functions coincides on this set, since they do on $\reg(m_1)\cap\reg(m_2)\cap\reg(m_3)$. That is, $m_1\simeq m_3$.

(3): Clearly, $\hat{m}$ is continuous on $\reg(m)\cup\reg_b(m)$, so the latter is contained in $\reg(\hat{m})$. Since $m$ and $\hat{m}$ coincide on the open set $X\setminus\ol{\reg(m)}$, we even have $\reg(\hat{m})\subseteq\ol{\reg(m)}$. In particular, $\ol{\reg(m)}=\ol{\reg(\hat{m})}$ and since $m$ and $\hat{m}$ are equal on $\reg(m)\cap\reg(\hat{m})=\reg(m)$, it follows that $m\simeq\hat{m}$. By (1) the proof is complete, since $\reg(m)\cup\reg_b(m)=\reg(\hat{m})\cup\reg_b(\hat{m})$ implies that $\reg_b(\hat{m})=\emptyset$.
\end{proof}

In particular, changing $m$ on $\reg_\infty(m)\cup\singsuppr(m)$ does not change any of the sets $\reg(m)$, $\reg_b(m)$, $\reg_\infty(m)$, $\singsuppr(m)$, and $X\setminus\ol{\reg(m)}$.

Now we define the multiplication operators $t_m$.

\begin{dfn}\label{definitiontm}
For a function $m:X\to\C$ let
\begin{align*}
\Def(t_m) &:= \{ f\in E | \widehat{mf} \in E \}, \quad t_mf := \widehat{mf}, \quad f\in\Def(t_m).
\end{align*}
\end{dfn}
It is straightforward to check that $t_m$ is indeed an operator on $E$.

\begin{lem}\label{RegularityTransfer__f_not_0}
Let $m,f:X\to\C$ and $x\in X$. If $f$ is continuous at $x$ and $f(x)\neq 0$, then $x\in \reg(m)$ if and only if $x\in \reg(mf)$.
\end{lem}
\begin{proof}
Since $f(x)\neq 0$, there is an open set $U_f$ containing $x$ such that $f(x')\neq 0$ for all $x'\in U_f$. By definition, $x\notin \reg(m)$ if for any neighbourhood $U$ of $x$, there is an $x'\in U$ such that $m$ is discontinuous at $x'$. 
This holds if and only if for any neighbourhood $U$ of $x$, there exists $x'\in U$ such that $mf$ is discontinuous at $x'$. This means that $x\notin \reg(mf)$.
\end{proof}

We now show that the operator $t_m$  depends only on the equivalence class of $m$.

\begin{lem}\label{Domain_Equivalence__CommutativeCase}
Let $m:X\to\C$ be given. Then
\begin{align*}
\Def(t_m) &= \{f\in E | f\equiv 0 \text{ on } X\setminus\reg(\hat{m}), \partial\reg(\hat{m})\subseteq\reg(\widehat{mf}),\widehat{m}f\in F_0(X)\},\\
(t_mf)(x) &= \hat{m}(x)f(x) ,\quad f\in\Def(t_m),x\in X\setminus\reg_\infty(m).
\end{align*}
For $x\in X$ there is an $f\in\Def(t_m)$ such that $f(x)\neq 0$ if and only if $x\in\reg(\hat{m})$. If $m,\tilde{m}:X\to\C$, then $t_m=t_{\tilde{m}}$ if and only if\, $m\simeq\tilde{m}$. In particular, $t_m=t_{\hat{m}}$.
\end{lem}
\begin{proof}
By Definition \ref{definitiontm} $\Def(t_m)$ consists of those $f\in E$ for which $\reg(\widehat{mf})=X$ and $\widehat{mf}\in F_0(X)$. Since $\reg(\hat{m})\subseteq\reg(\widehat{mf})$ for $f\in E$,  $\Def(t_m)$ is the set of all $f\in E$ such that $X\setminus\reg(\hat{m})\subseteq\reg(\widehat{mf})$ and $\widehat{mf}\in F_0(X)$.

Let $f\in E$ with $\widehat{m}f\in F_0(X)$. Suppose that  $\partial\reg(\hat{m})\subseteq\reg(\widehat{mf})$ and $f\equiv 0$ on $X\setminus\reg(\hat{m})$. Then $mf\equiv 0$ on $X\setminus\reg(\hat{m})$ and so on  $X\setminus\ol{\reg(\hat{m})}$. Hence  $X\setminus\reg(\hat{m})=X\setminus\ol{\reg(\hat{m})}\cup\partial\reg(\hat{m})$ is contained in $\reg(\widehat{mf})$. To show $\widehat{mf}\in F_0(X)$, let $\epsilon>0$. Since $\widehat{m}f\in F_0(X)$, there exists an compact set $K\subseteq X$ such that $|\widehat{m}f|\leq\epsilon$ on $X\setminus K$. By continuity of $\widehat{mf}$ on $X$ the same is true for this function, since $\widehat{m}f$ and $\widehat{mf}$ coincide on the dense set $X\setminus\partial\reg(m)$. That is $\widehat{mf}\in F_0(X)$, finally $f\in\Def(t_m)$.

Now we suppose  that $f\in\Def(t_m)$. In particular, $\reg(mf)$ is dense in $X$, hence $\partial\reg(\hat{m})\subseteq X=\reg(\widehat{mf})$. Assume that $x\in X\setminus\reg(\hat{m})$ and $f(x)\neq 0$. By the continuity of $f$ there exists an open set $U\subseteq X\setminus\reg(\hat{m})$ such that  $f(y)\neq 0$ for  $y\in U$. From Lemma \ref{RegularityTransfer__f_not_0} it follows that $U$ is even contained in $X\setminus\reg(\hat{m}f)$. But this contradicts the density of $\reg(mf)\subseteq\reg(\hat{m}f)$ in $X$, hence $f(x)=0$. In particular $\widehat{mf}$ coincides with $\widehat{m}f$ on $\reg(\hat{m})$ and $\widehat{m}f\equiv 0$ on $X\setminus\reg(\hat{m})$. So $\widehat{mf}\in F_0(X)$ implies $\widehat{m}f\in F_0(X)$ and the description of $\Def(t_m)$ is proven.

For $f\in\Def(t_m)$ it is obvious that $t_mf\equiv\hat{m}f$ on $\reg(\hat{m})$. Since $f\equiv 0$ on $X\setminus\reg(\hat{m})$, it remains to show $\widehat{mf}\equiv 0$ on $(X\setminus\ol{\reg(\hat{m})})\cup\singsuppr(m)$. For this it suffices to prove the following: If $x\in X$ with $f(x)=0$ and $\widehat{mf}(x)\neq$ 0, then $x\in\reg_\infty(m)$. Indeed, $\widehat{mf}$ does not vanish on a neighbourhood $U$ of $x$. Let $\tilde{m}:U\to\ol{\C}$ denote the function $\widehat{mf}/f$ on $U$, where $\alpha/0:=\infty$ for $\alpha\in\C$. Then  $\tilde{m}$ is continuous, since $\widehat{mf}$ does not vanish on $U$, and  $\tilde{m}$ coincides with $m$ on $U\cap\reg(m)$. Further,  since $f(x)=0$ and $\widehat{mf}(x)\neq 0$, we  have  $\tilde{m}(x)=\infty$. Hence $x\in\reg_\infty(m)$.

Let $x\in X$. By the preceding,  if $x\notin\reg(\hat{m})$, then $f(x)=0$ for all $f\in\Def(t_m)$. If $x\in\reg(\hat{m})$, there exists a function $f:X\to\C$ with compact support contained in $\reg(\hat{m})$ and $f(x)\neq 0$, since $X$ is locally compact. Then $f\in\Def(t_m)$, since $\hat{m}f$ is continuous everywhere and has compact support.

Next we prove that $t_m=t_{\tilde{m}}$ if and only if $m\simeq\tilde{m}$. Assume first that $t_m=t_{\tilde{m}}$. Then, by the preceding,
\begin{align*}
x\in \reg(\hat{m}) &\Leftrightarrow \exists f\in\Def(t_m): f(x)\neq 0 \Leftrightarrow \exists f\in\Def(t_{\tilde{m}}): f(x)\neq 0 \Leftrightarrow x\in \reg(\hat{\tilde{m}}).
\end{align*}
Thus $\reg(\hat{m})=\reg(\hat{\tilde{m}})$ which implies that $\ol{\reg(m)}=\ol{\reg(\hat{m})}=\ol{\reg(\hat{\tilde{m}})}=\ol{\reg(\tilde{m})}$. For $x\in\reg(m)\cap\reg(\tilde{m})$ we  choose $f\in\Def(t_m)$ such that  $f(x)\neq 0$. Then $m(x)f(x)=\widehat{mf}(x)=\widehat{\tilde{m}f}(x)=\tilde{m}(x)f(x)$, so that  $m(x)=\tilde{m}(x)$. Hence $m\simeq\tilde{m}$.

Conversely, assume  that $m\simeq\tilde{m}$. Let $f\in\Def(t_m)$. Then $\widehat{mf}\equiv mf\equiv \tilde{m}f$ on $\reg(m)\cap\reg(\tilde{m})$ and the latter set is dense in $\ol{\reg(\hat{m})}$, since $m\simeq\tilde{m}$. Further, $f\equiv 0$ on $X\setminus\ol{\reg(\hat{m})}$ by $f\in\Def(t_m)$ and hence $\widehat{mf}\equiv 0\equiv \tilde{m}f$ on $X\setminus\ol{\reg(\hat{m})}$. Thus, $\widehat{mf}\equiv\tilde{m}f$ on a dense set. Therefore, since $\widehat{mf}$ is continuous on $X$, we  have  $\widehat{\tilde{m}f}=\widehat{mf}$. Thus,  $f\in\Def(t_{\tilde{m}})$ and $t_{\tilde{m}}f=t_mf$. This  proves that  $t_m\subseteq t_{\tilde{m}}$. By symmetry, $t_m=t_{\tilde{m}}$.
\end{proof}

Note that  the sets $\reg(m)$, $\reg_b(m)$, $\reg_\infty(m)$ and $\singsuppr(m)$ remain unchanged if $m$ is replaced by its complex conjugate function $\ol{m}$. Moreover,  $\hat{\ol{m}}=\ol{\hat{m}}$.

\begin{thm}\label{MultiplicationOperator__Essential_ssupp}
Let $m:X\to\C$. The operator  $t_m$ is essentially defined if and only if $\reg(m)$ is dense in $X$. In this case we have $t_m^*=t_{\ol{m}}$ and $t_m\in\oC(E)$.
\end{thm}
\begin{proof}
By Lemma \ref{Domain_Equivalence__CommutativeCase} we can assume without loss of generality that $m=\hat{m}$. Assume that $\reg(m)$ is dense in $X$. Let $g\bot\Def(t_m)$. For $x\in\reg(m)$ there exists $f\in\Def(t_m)$ such that $f(x)\neq 0$ by Theorem \ref{Domain_Equivalence__CommutativeCase}. From $\ol{g(x)}f(x)=\SP{g}{f}(x)=0$ we conclude that $g(x)=0$. Hence, by density of $\reg(m)$,  $g=0$. That is, $\Def(t_m)^\bot=\{0\}$. 

Conversely, suppose that $\reg(m)$ is not dense. Since $X$ is a locally compact Hausdorff space, there is a  function $g\in C_0(X)$, $g\neq 0$, with support contained in $X\setminus\reg(m)$. Then $\SP{f}{g}=0$ for $f\in\Def(t_m)$ by Theorem \ref{Domain_Equivalence__CommutativeCase} and so $\Def(t_m)^\bot\neq\{0\}$.

Now assume that $\reg(m)$ is dense. We  show that $t_m^*=t_{\ol{m}}$. Let $f\in\Def(t_m)$ and $g\in\Def(t_{\ol{m}})$. For $x\in \reg(m)=\reg(\ol{m})$ we get $$\SP{t_mf}{g}(x)=\ol{m(x)f(x)}g(x)=\SP{f}{t_{\ol{m}}g}(x).$$ Since $\reg(m)$ is dense, we conclude that $\SP{t_mf}{g}=\SP{f}{t_{\ol{m}}g}$. Thus, $t_{\ol{m}}\subseteq t_m^*$. 

Let $f\in\Def(t_m^*)$. There exists $h\in E$ such that $$\ol{\widehat{mg}(x)}f(x)=\SP{t_mg}{f}(x)=\SP{g}{h}(x)=\ol{g(x)}h(x),\quad g\in\Def(t_m),~x\in X.$$
For $x\in \reg(m)$ we choose $g\in\Def(t_m)$ such that $g(x)\neq 0$. Hence $h\equiv\ol{m}f$ on $\reg(m)$. Since $\reg(m)$ is dense, we get $\widehat{\ol{m}f}=h\in E$. Thus $f\in\Def(t_{\ol{m}})$.

Finally, from $\reg(\ol{m})=\reg(m)$ we easily derive\, $t_m^{**}=(t_{\ol{m}})^*=t_m\in\oC(E)$.
\end{proof}
The domain $\Def(t_m)$ can be trivial  if $m$ is continuous on an dense set, since $\reg(m)$ is empty if there is no open  set of continuous points of $m$.

\begin{lem}\label{AdditionMultiplication_CommutativeCase}
Let $m,n:X\to\C$ be two functions. Then:
\begin{enumerate}
\item  $t_m+t_n\subseteq t_{m+n}$ and $t_mt_n\subseteq t_{mn}$.
\item If $\reg(m)$ and $\reg(n)$ are dense in $X$, then $t_m+t_n,t_mt_n\in\oCl(E)$ and $(t_m+t_n)^{**}=t_{m+n}$, $(t_mt_n)^{**} = t_{mn}$.
\end{enumerate}
\end{lem}
\begin{proof}
We  verify the statements about the product; the proof for the sum is similar.

(1): Let $f\in\Def(t_mt_n)$. Then $\reg(nf)$ and $\reg(m\widehat{nf})$ are dense and open in $X$, so is their intersection which is contained in $\reg(mnf)$. On this set, $mnf\equiv m\widehat{nf}$. Hence $\widehat{(mn)f}=[m\widehat{nf}]^\wedge\in E$, that is, $f\in\Def(t_{mn})$ and $t_mt_nf=t_{mn}f$.

(2): First we prove that $\Def(t_mt_n)^\bot=\{0\}$. Let $g\bot\Def(t_mt_n)$ and $x\in\reg(n)\cap\reg(m)$. Since the latter set is open and $X$ is a locally compact  Hausdorff space, there exists $f\in E$ such that its support is contained in this set and $f(x)\neq 0$. Clearly, $nf\in E$ and $mnf\in E$, hence $f\in\Def(t_mt_n)$. But $\ol{f(x)}g(x)=\SP{f}{g}(x)=0$, so $g(x)=0$. Since $\reg(n)\cap\reg(m)$ is dense, $g=0$. Thus, $t_mt_n$ is essentially defined. By (1), $(t_mt_n)^*\supseteq t_{mn}^*=t_{\ol{mn}}$, and $t_{\ol{mn}}$ is essentially defined by Theorem \ref{MultiplicationOperator__Essential_ssupp}, since $\reg(m)\cap\reg(n)\subseteq\reg(\ol{mn})$ is dense.

Now we show that $\Def((t_mt_n)^*)\subseteq\Def(t_{\ol{mn}})$ which in turn implies the last assertion $(t_mt_n)^{**}=t_{\ol{mn}}^*=t_{mn}$. Let $f\in\Def((t_mt_n)^*)$. Then there exists $g\in E$ such that $\SP{f}{t_mt_nh}=\SP{g}{h}$ for all $h\in\Def(t_mt_n)$. Arguing as above, for $x\in\reg(n)\cap\reg(m)$ there exists $h\in\Def(t_mt_n)$ with $h(x)\neq 0$. Thereore,  $\ol{f}mn\equiv \ol{g}$ on the dense set $\reg(m)\cap\reg(n)$. Hence $\widehat{\ol{mn}f}=g\in E$, that is, $f\in\Def(t_{\ol{mn}})$.
\end{proof}

\begin{rem}
Set $m(x):=e^{i/x}$ on $X=[0,1]$. Then  $t_m^*t_m\subsetneq t_{|m|^2}$.
\end{rem}

\begin{thm}\label{maintheorem}
Let $m:X\to\C$. Suppose  that $\reg(m)$ is dense in $X$. Then:
\begin{enumerate}
\item $t_m$ is normal and $t_m^*t_m$ is essentially self-adjoint.
\item $\Def(t_m^*t_m)$ is an essential core for $t_m$.
\item $\Range(1+t_m^*t_m)=\{g\in E | \forall x\in\singsuppr(m): g(x)=0\}$. In particular, $\Range(1+t_m^*t_m)$ is essential.
\end{enumerate}
\end{thm}
\begin{proof}
By Lemma \ref{Domain_Equivalence__CommutativeCase} we can assume  that $m=\hat{m}$. Then $t_m\in\oC(E)$ and $t_m^*=t_{\ol{m}}$ by Theorem \ref{MultiplicationOperator__Essential_ssupp}.

(1): Using Lemma \ref{AdditionMultiplication_CommutativeCase}(2) we get $(t_m^*t_m)^{**}=(t_{\ol{m}}t_m)^{**}=t_{|m|^2}$. Since  the latter is self-adjoint, $t_m^*t_m$ is essentially self-adjoint.

Now we prove that\, $t_{\ol{m}}t_m=t_mt_{\ol{m}}$\,. Since both operators are restrictions of $t_{|m|^2}$, it suffices to show that their domains coincide. By symmetry it is enough to prove that $\Def(t_{\ol{m}}t_m)\subseteq\Def(t_mt_{\ol{m}})$. Let $f\in\Def(t_{\ol{m}}t_m)$, that is, $\widehat{mf}\in E$ and $[\ol{m}\widehat{mf}]^\wedge\in E$. We have to show that $\widehat{\ol{m}f}\in E$ and $[m\widehat{\ol{m}f}]^\wedge\in E$. Clearly, $[m\widehat{\ol{m}f}]^\wedge=[|m|^2f]^\wedge=[\ol{m}\widehat{mf}]^\wedge\in E$, since  these functions coincide on $\reg(m)$. Further,
\begin{align*}
|\widehat{mf}|^2 &= \SP{t_mf}{t_mf} = \SP{t_{\ol{m}}t_mf}{f} = \ol{[|m|^2f]^\wedge}f.
\end{align*}
Hence $[|m|^2f]^\wedge(x)=0$ implies that $\widehat{mf}(x)=0$ for $x\in X$. Using this fact it follows that the function
\begin{align*}
h(x) &:= \begin{cases}
\ol{\widehat{mf}}(x)\left(\widehat{|m|^2f}(x)\right)^2/\left|\widehat{|m|^2f}(x)\right|^2 &,\text{if } \widehat{|m|^2f}(x) \neq 0,\\
0 &,\text{if } \widehat{|m|^2f}(x) = 0
\end{cases}
\end{align*}
belongs to  $E$ and coincides with $\ol{m}f$ on $\reg(m)$. Thus $[\ol{m}f]^\wedge\in E$.

(2): We  show that $\Graph(t_m\upharpoonright_{\Def(t_m^*t_m)})^\bot\subseteq\Graph(t_m)^\bot$. Assume that $\SP{(g,h)}{(f,t_mf)}=0$ for all $f\in\Def(t_{\ol{m}}t_m)$. Let $x\in\reg(m)$. Since $X$ is a locally compact Hausdorff space, there exists $f_x\in C_0(X)$ with support contained in $\reg(m)$ and $f_x(x)\neq 0$. Clearly, $f_x\in\Def(t_{\ol{m}}t_m)$. Then $\ol{g(x)}f_x(x)=-\ol{h(x)}t_mf_x(x)=-\ol{h(x)}m(x)f_x(x)$, hence $\ol{g}\equiv -\ol{h}m$ on $\reg(m)$. Now for all $f\in\Def(t_m)$ we have $\SP{(g,h)}{(f,t_mf)}\equiv \ol{g}f+\ol{h}mf\equiv 0$ on $\reg(m)$, that is, $(g,h)\bot\Graph(t_m)$, since $\reg(m)$ is dense.

(3): Let $x\in\singsuppr(m)=\singsuppr(\ol{m})$ and $f\in\Def(t_{\ol{m}}t_m)$. By Theorem \ref{Domain_Equivalence__CommutativeCase},  $f(x)=0$, $(t_mf)(x)=0$ and $(t_{\ol{m}}t_mf)(x)=0$. Thus  $((1+t_m^*t_m)f)(x)=0$ which proves one inclusion. Conversely, let $g\in E$ with $g(x)=0$ for all $x\in\singsuppr(m)$. The functions $[1/(1+|m|^2)]^\wedge$, $[m/(1+|m|^2)]^\wedge$ and $[|m|^2/(1+|m|^2)]^\wedge$ are bounded and continuous on $\reg(m)\cup\reg_\infty^\infty(m)$. Hence, setting $f:=[g/(1+|m|^2)]^\wedge$, we have
\begin{align*}
f \in E, \quad \widehat{mf} = \left[\frac{gm}{1+|m|^2}\right]^\wedge\in E, \quad [\ol{m}\widehat{mf}]^\wedge = \left[\frac{g|m|^2}{1+|m|^2}\right]^\wedge\in E,
\end{align*}
by using that $g\equiv 0$ on $\singsuppr(m)$. That is, $f\in\Def(t_m^*t_m)$ and $(1+t_m^*t_m)f\equiv g$ on $\reg(m)$, so that $g\in\Range(1+t_m^*t_m)$, since $\reg(m)$ is dense in $X$.
\end{proof}

The next theorem is one of our main results in this Section. It says that all essentially defined orthogonally closable operators  on $E$ are  multiplication operators.

\begin{thm}\label{operatorfunctionstm}
If\, $t\in\oCl(E)$, then there is a function $m:X\to\C$ such that $t=t_m$.
\end{thm}
\begin{proof}
Let $t\in\oCl(E)$. We abbreviate $\Def:=\Def(t)$ and $\Def_*:=\Def(t^*)$. We set
\begin{align*}
\mathcal{O} := \cup_{f\in\Def}\mathcal{O}_f, \quad \mathcal{O}_* := \cup_{f\in\Def_*}\mathcal{O}_f \quad\text{with}\quad \mathcal{O}_f := \{x\in X|f(x)\neq 0\}.
\end{align*}
Further, since $\Def$ and $\Def_*$ are essential,\, $\cO$ and $\cO_*$ are dense in $X$. Hence $\cO':=\cO\cap \cO_*$ is also dense. For $x\in X$, we have
\begin{align*}
\ol{g(x)}(tf)(x) &= \SP{g}{tf}(x) = \SP{t^*g}{f}(x) = \ol{(t^*g)(x)}f(x) \quad {\rm for}\quad f\in\Def,g\in\Def_*.
\end{align*}
If $x\in \cO'$, there are $f\in\Def$ and $g\in\Def_*$ such that $f(x)\neq 0$ and $g(x)\neq 0$. Then
\begin{align*}
m(x) &:= (tf)(x)/f(x) = \ol{(t^*g)(x)}/\ol{g(x)}.
\end{align*}
In particular, this shows that $m(x)$ is independent of the chosen functions $f$ and $g$  and that $m$ is continuous on $\cO'$. Now let $f\in\Def$ and $x\in \cO'$. Then there is a  $g\in\Def_*$ such that $g(x)\neq 0$, so $(tf)(x)=m(x)f(x)$. Similarly, $(t^*g)(x)=\ol{m(x)}g(x)$ for $g\in\Def_*$ and $x\in \cO'$.

We now extend $m$ arbitrarily to a function defined on the whole set $X$. It follows that $\widehat{mf}\in E$ for $f\in\Def$ and $\widehat{\ol{m}g}\in E$ for $g\in\Def_*$. Then we have $f\in\Def(t_m)$ and $t_mf=tf$\, for $f\in\Def$. Likewise,  $g\in\Def(t_{\ol{m}})$ and $t_{\ol{m}}g=t^*g$ for $g\in\Def_*$. Thus, $t\subseteq t_m$ and $t^*\subseteq t_{\ol{m}}$. Therefore, $t_m=(t_{\ol{m}})^*\subseteq t^{**}=t\subseteq t_m$, that is, $t_m=t$.
\end{proof}

The next example gives a normal operator $t_m$ such that $\Def(t_m)\neq\Def(t_{\ol{m}})$.

\begin{exa}\label{Multiplier__different_Dm_Dm*}
Set $X:=[0,1]$ and define $m,f:X\to\C$ by
\begin{align*}
m(x) &:= \begin{cases}
          e^{i/x}/x &, x\neq 0\\
          0 & ,x=0
         \end{cases}, \quad f(x) := \begin{cases}
          e^{-i/x}x &, x\neq 0\\
          0 & ,x=0
         \end{cases}.
\end{align*}
Then $\reg(m)=(0,1]$ and $f\in C_0(X)$. Further, $\widehat{mf}=1$, so  $f\in\Def(t_m)$. On the other side, $(\ol{m}f)(x)=e^{-2i/x}$ for $x\in(0,1]$, so $\reg(\widehat{\ol{m}f})=(0,1]$ and $f\notin\Def(t_{\ol{m}})$. This proves   that\, $\Def(t_m)\neq\Def(t_{\ol{m}})$. By Theorem \ref{maintheorem}, the operator $t_m$ is normal.
\end{exa}

  \section{Graph regular operators}\label{graphregular}

  \subsection{Definition and basics on graph regular operators}\label{DefBasics_GraphREgular}

Graph regular operators are the  most important new class of operators introduced in  this paper. 
\begin{dfn}\label{defgraphregular}
An operator $t:E\to F$ is called \\
$\bullet$\,  \emph{graph regular} if\, $t$ is essentially defined and orthogonally closed and its graph $\Graph(t)$ is orthogonally complemented in\, $E\oplus F$, \\
$\bullet$\,  \emph{regular} if\, $t$ is closed, $\Def(t)$ is dense in $E$, $\Def(t^*)$ is dense in $F$, and $\Range(1+t^*t)$ is dense in $E$. 
\end{dfn}
The preceding is the definition of a regular operator given  in \cite[p. 96]{lance}.
Each regular operator is graph regular by \cite[Theorem 9.3]{lance}.

By an equivalent definition, an operator $t:E\to F$ is graph regular if $t$ is closed, $\Def(t)$ and $\Def(t^*)$ are essential in $E$ and $F$, respectively, and  $\Range(1+t^*t)$ and $\Range(1+tt^*)$ are dense in $E$ and $F$,  respectively. (The equivalence to Definition \ref{defgraphregular} is easily verified by using some arguments from the proof of Theorem \ref{Basics__graph_regular}(1) below.)

We denote by $\grReg(E,F)$  the set of all  essentially defined graph regular operators and by $\Reg(E,F)$ the set of regular operators from $E$ into $F$. Let us abbreviate $\grReg(E):=\grReg(E,E)$ and $\Reg(E):=\Reg(E,E)$. 

A number of basic properties of graph regular operators are collected in the following theorem.

\begin{thm}\label{Basics__graph_regular}
\begin{enumerate}
\item  For $t\in\oCl(E,F)$ the following statements are equivalent:
\begin{enumerate}
\item $t\in\grReg(E,F)$.
\item $\Graph(t)\oplus v\Graph(t^*)=E\oplus F$.
\item $\Range(1+t^*t)=E$ and $\Range(1+tt^*)=F$.
\end{enumerate}
\item If $t\in\oC(E,F)$ and $t^*\in\grReg(F,E)$, then $t\in\grReg(E,F)$.
\item If $t\in\grReg(E,F)$, then
\begin{enumerate}
\item $\Def(t^*t)$ is an essential core for $t$.
\item $t^*\in\grReg(F,E)$.
\end{enumerate}
\end{enumerate}
\end{thm}
\begin{proof}
(1a) $\Rightarrow$ (1b) follows from the relation $v\Graph(t^*)=\Graph(t)^\bot$ by Proposition \ref{Graph_tadj}. 

(1b) $\Rightarrow$ (1a): Since $\Graph(t)\subseteq\Graph(t)^{\bot\bot}$, (1b) clearly implies that $\Graph(t)=\Graph(t)^{\bot\bot}$. Hence $t$ is orthogonally closed and graph regular. 

(1b) $\Rightarrow$ (1c): If $x\in E$, then $(x,0)\in\Graph(t)\oplus v\Graph(t^*)$, so there are $y\in\Def(t)$, $z\in\Def(t^*)$ such that $x=y+t^*z$, $0=ty-z$. Then $y\in\Def(t^*t)$ and $x=(1+t^*t)y\in\Range(1+t^*t)$. In the same way one shows that $F=\Range(1+tt^*)$.

(1c) $\Rightarrow$ (1b): Since $\Range(1+t^*t)=E$, we have $E\oplus 0=\Range(1+t^*t)\oplus 0\subseteq\Graph(t)\oplus v\Graph(t^*)$. For $x\in\Def(t^*t)$, we set $y:=tx\in\Def(t^*)$. Then $$((1+t^*t)x,0)=(x+t^*y,tx-y)\in\Graph(t)\oplus v\Graph(t^*).$$ Similarly,  $0\oplus F=0\oplus\Range(1+tt^*)\subseteq\Graph(t)\oplus v\Graph(t^*)$. Thus $E\oplus F\subseteq\Graph(t)\oplus v\Graph(t^*)$\, which yields \,
$\Graph(t)\oplus v\Graph(t^*)=E\oplus F$.

(2) is obtained from $t=t^{**}$ and (1).

(3a): We have $\SP{x}{y}+\SP{tx}{ty}=\SP{x}{(1+t^*t)y}$ for $x\in\Def(t)$ and $y\in\Def(t^*t)$, so that $\Graph(t\upharpoonright_{\Def(t^*t)})^\bot\cap\Graph(t)=\{0\}$. With Lemma \ref{technic_core} it follows
\begin{align*}
 \Graph(t)^\bot &= \Graph(t\upharpoonright_{\Def(t^*t)})^\bot\cap(\Graph(t)\oplus v\Graph(t^*)) = \Graph(t\upharpoonright_{\Def(t^*t)})^\bot\cap E = \Graph(t\upharpoonright_{\Def(t^*t)})^\bot,
\end{align*}
so $\Def(t^*t)$ is an essential core for $t$. 

(3b): This follows from (2), since $t$ and $t^*$ are orthogonally closed.
\end{proof}

A special situation is treated in the following example.

\begin{exa}
Let $\Hil$ be a separable Hilbert space,  $\CAlg{A}$ a $C^*$-algebra  of compact operators acting on $\Hil$, and $E$ a Hilbert $\CAlg{A}$-module. Then we have  $$\oC(E)=\grReg(E)=\Reg(E).$$ 
Indeed, since $\Reg(E)\subseteq \grReg(E)\subseteq\oC(E)$ by definition, it suffices to prove that $\oC(E)\subseteq \Reg(E)$.  Let $t\in \oC(E)$. Since  all closed submodules of any Hilbert $C^*$-module over a  $C^*$-algebra of compact  operators  are orthogonally complemented \cite{mag}, all essential submodules are dense. Hence $t$ and $t^*$ are densely defined on $E$. Further, since $t$ is closed,   $t$ is  semiregular in the sense of  \cite{pal}. As shown in \cite[Proposition 5.1]{pal}, for any Hilbert $\CAlg{A}$-module $E$ of a $C^*$-algebra of compact operators  semiregular operators are always regular. Thus,  $t\in \Reg(E)$.

In the very special case $E=\CAlg{A}=\Komp(\Hil)$  we have $\Reg(E)=\Abg(\Hil)$, since then the regular operators on $E$ are the  affiliated operators with $\CAlg{A}$  and  $\Abg(\Hil)$ is the set of affiliated operators  with $\mathcal{K}(\Hil)$ as noted in \cite{wor91}.
\end{exa}

\begin{prop}\label{GraphRegular_Addition_Composition}
Let\, $t\in\grReg(E,F)$ and $q\in\Adj(E,F)$. Suppose that $r\in\Adj(G,E)$ and  $s\in\Adj(F,G)$ are invertible with $r^{-1}\in\Adj(E,G)$ and  $s^{-1}\in\Adj(G,F)$. Then the operators $t+q$, $tr$ and $st$ are essentially defined and graph regular.
\end{prop}
\begin{proof}
Let $p_E$ and $p_F$ denote the projections from $E\oplus F$ onto $E$ and $F$, respectively. Clearly, $t+q$, $tr$, and $st$ are essentially defined and orthogonally closed by Proposition \ref{Adjoint__Addition_Compositon}. In particular, their graphs are closed. Since $t$ is graph regular, $\Graph(t)$ is orthogonally complemented, so there is a projection $p\in\Adj(E\oplus F)$ with $\Range(p)=\Graph(t)$. We now obtain
\begin{align*}
\Graph(t+q) &=\{(x,tx+qx)|x\in\Def(t)\} = \{(p_Epv,p_Fpv+qp_Epv)|v\in E\oplus F\}\\
&= \Range((p_E,p_F+qp_E)p),\\
\Graph(tr) &= \{(r^{-1}x,tx)|x\in\Def(t)\} = \{(r^{-1}p_Epv,p_Fpv)|v\in E\oplus F\}\\
&= \Range((r^{-1}p_E,p_F)p),\\
\Graph(st) &= \{(x,stx)|x\in\Def(t)\} = \{(p_Ev,sp_Fpv)|v\in E\oplus F\}\\
&= \Range((p_E,sp_F)p).
\end{align*}
Thus the closed subspaces $\Graph(t+q)$, $\Graph(tr)$, and $\Graph(st)$ are ranges of adjointable operators, hence they  are orthogonally complemented by \cite[Theorem 3.2]{lance}.
\end{proof}

The next lemma describes a cases where graph regularity implies regularity.

\begin{lem}
If $t\in\grReg(E,F)$,\, $\Range(t)\subseteq\ol{\Def(t^*)}$ and $\Range(t^*)\subseteq\ol{\Def(t)}$\,, then $t\in\Reg(E,F)$.
\end{lem}
\begin{proof}
We have to prove that $\Def(t)$ and $\Def(t^*)$ are dense in $E$ and $F$, respectively. For $\Def(t)$ this follows from the relations $$E=\Range(1+t^*t)\subseteq\Def(t)+\Range(t^*)\subseteq\ol{\Def(t)}.$$ 
Since $t^*\in\grReg(F,E)$ by Theorem \ref{Basics__graph_regular}, we can replace $t$ by $t^*$ in the preceding and obtain the density of $\Def(t^*)$.
\end{proof}

  \subsection{The $(a,a_*,b)$-transform}\label{Section_aabTransform}

In this section we establish a one-to-one correspondence between graph regular operators  and   certain triples  of adjointable operators. As noted in Remark \ref{firstremark} this works in a purely algebraic setting and it neither requires  the $C^*$-condition nor even a norm.

\begin{dfn}\label{defabtransform}
For Hilbert $\CAlg{A}$-modules $E$ and $F$, let $\mathcal{AB}(E,F)$ denote the set of all triples $(a,a_*,b)$ of operators $a\in\Adj(E)$, $a_*\in\Adj(F)$, $b\in\Adj(E,F)$ such that $a$ and $a_*$ are self-adjoint, $\Null(a)=\{0\}$, $\Null(a_*)=\{0\}$, and
\begin{align*}
b^*b = a-a^2, \quad bb^* = a_*-a_*^2, \quad ab^*=b^*a_*.
\end{align*}
\end{dfn}
In particular $0\leq a\leq I$ and $0\leq a_*\leq I$ in this case; further $\|b\|\leq 1$.

We call the map $t\to(a_t,a_{t^*},b_t)$ described in Theorem \ref{aabTransform}
the \emph{$(a,a_*,b)$-transform}.

\begin{thm}\label{aabTransform}
If $t\in\grReg(E,F)$, then $(a_t,a_{t^*},b_t)\in\mathcal{AB}(E,F)$, where
\begin{align*}
a_t &:= (1+t^*t)^{-1}, \quad a_{t^*} := (1+tt^*)^{-1}, \quad b_t := t(1+t^*t)^{-1}.
\end{align*}
Further, $\Null(b_t)=\Null(t)$, $b_{t^*}=b_t^*$, and the projection onto the graph of $t$ is given by
\begin{align*}
p &= \left(\begin{matrix}
      a_t & b_t^*\\
      b_t & 1-a_{t^*}
     \end{matrix}\right) \in \Adj(E\oplus F,E\oplus F).
\end{align*}
If $(a,a_*,b)\in\mathcal{AB}(E,F)$, then $t_{a,a_*,b}\in\grReg(E,F)$, where
\begin{align*}
t_{a,a_*,b} &:= (ba^{-1})^{**}=(b^*a_*^{-1})^*,
\end{align*}
and we have $t_{a,a_*,b}^*=t_{a_*,a,b^*}$. The map $t\mapsto(a_t,a_{t^*},b_t)$ is a bijection from $\grReg(E,F)$ onto $\mathcal{AB}(E,F)$ with inverse $(a,a_*,b)\mapsto t_{a,a_*,b}$.
\end{thm}
\begin{proof}
First we suppose  that  $t\in\grReg(E,F)$.
Then $\Range(1+t^*t)=E$, so $a_t$ is defined on the whole module $E$. It is straightforward to verify that $1+t^*t$ is positive and injective for each essentially defined operator $t$. Therefore $a_t$ is positive and has a trivial kernel. Analogous statements hold for $a_{t^*}$. Further, $b_t$ is defined on $E$, since $\Range(a_t)\subseteq\Def(t)$. Similarly, $b_{t^*}$ is defined on $F$. For $x:=(1+t^*t)x'\in E$ and $y:=(1+tt^*)y'\in F$, where $x'\in E$, $y'\in F$, we compute
\begin{align*}
\SP{b_tx}{y} &= \SP{tx'}{(1+tt^*)y'} = \SP{tx'}{y'}+\SP{tx'}{tt^*y'} = \SP{tx'}{y'}+\SP{t^*tx'}{t^*y'}\\
&= \SP{(1+t^*t)x'}{t^*y'} = \SP{x}{b_{t^*}y}.
\end{align*}
Hence $b_t=(b_{t^*})^*\in\Adj(E,F)$. From $b_{t^*}=(b_t)^*=(ta_t)^*$ we obtain $$b_{t^*}b_t\supseteq a_tt^*ta_t = a_t(1-a_t).$$ Since $a_t(1-a_t)$ is defined on the whole $E$, the latter yields $b_{t^*}b_t=a_t-a_t^2$. Further, $(1+t^*t)t^*=t^*(1+tt^*)$ and $\Range(a_{t^*}^2)=\Def(tt^*tt^*)\subseteq\Def(t^*tt^*)$ imply that
\begin{align*}
b_{t^*}a_{t^*} &= t^*a_{t^*}^2 = 1\upharpoonright_{\Def(t^*t)}t^*a_{t^*}^2 =  a_t(1+t^*t)t^*a_{t^*}^2 = a_tt^*(1+tt^*)a_{t^*}^2 = a_tb_{t^*}.
\end{align*}
The preceding proves that $(a_t,a_{t^*},b_t)\in\mathcal{AB}(E,F)$. 

Clearly, $\Range(b_{t^*})\subseteq\Range(t^*)$, so $\Null(t)\subseteq\Null(b_t)$. Suppose that $b_tx=0$ for some $x\in E$. Then $(a_t-a_t^2)x=b_t^*b_tx=0$, so $x=a_tx\in\Def(t^*t)\subseteq\Def(t)$. Further, $(1+t^*t)x=x$, so $t^*tx=0$ and from $\SP{tx}{tx}=\SP{t^*tx}{x}=0$ it follows that $x\in\Null(t)$.  Thus,  $\Null(t)=\Null(b_t)$. The statement concerning the projection is easily verified.

Conversely, we now assume that $(a,a_*,b)\in\mathcal{AB}(E,F)$. We define $t:=ba^{-1}$ and $s:=b^*a_*^{-1}$. Since $\Def(t)^\bot=\Range(a)^\bot=\Null(a)=\{0\}$, $t$ is essentially defined.  Similarly, $s$ is essentially defined. For $x\in E$, $y\in F$ we have 
\begin{align*}
\SP{t(ax)}{a_*y} &= \SP{bx}{a_*y} = \SP{a_*bx}{y} = \SP{bax}{y} = \SP{ax}{b^*y} = \SP{ax}{s(a_*y)},
\end{align*}
so $t\subseteq s^*$ and $s\subseteq t^*$. In particular, $t\in\oCl(E,F)$. 

Our next aim is to prove that $\Range(a)$ is an essential core for $s^*$. Since $s^*=a_*^{-1}b$ by Proposition \ref{Adjoint__Addition_Compositon}, it suffices  to show that
$\Graph(ba^{-1})^\bot \subseteq \Graph(a_*^{-1}b)^\bot$. Let $(r,s)\in\Graph(ba^{-1})^\bot$. Then $\SP{(r,s)}{(ax,bx)}=0$ for all $x\in E$, so $ar+b^*s = 0$.
Further, we have 
$$a_*(br+(1-a_*)s)=a_*br+(a_*-a_*^2)s=bar+bb^*s=b(ar+b^*s)=0.$$ Since $a^*$ is injective, this yields $s = a_*s - br$.

Let $x\in\Def(a_*^{-1}b)$. Then there exists a (unique) element $z\in F$ with $bx=a_*z$. Using  the assumption $b^*a_*=ab^*$ we obtain $$b^*z=a^{-1}b^*a_*z=a^{-1}b^*bx=a^{-1}(a-a^2)x=(1-a)x.$$ Now we compute 
\begin{align*}
\SP{(r,s)}{(x,a_*^{-1}bx)} &= \SP{r}{x} + \SP{s}{a_*^{-1}bx} = \SP{r}{x} + \SP{s}{z}\\
&= \SP{r}{x} + \SP{a_*s-br}{z} = \SP{r}{x} + \SP{s}{a_*z} - \SP{r}{b^*z}\\
&= \SP{r}{x} + \SP{s}{bx} - \SP{r}{(1-a)x} = \SP{b^*s}{x} + \SP{ar}{x} = 0.
\end{align*}
Therefore, $(r,s)\bot\Graph(a_*^{-1}b)$. This  proves that $\Range(a)$ is an essential core for $s^*$. 

Since $t\subseteq s^*$ and $\Def(t)=\Range(a)$ is an essential core for $s^*$, we have $\Graph(t)^{\bot\bot}=\Graph(s^*)$, so that $t^{**}=s^*$, that is, $(ba^{-1})^{**}=(b^*a_*^{-1})^*.$ Finally, we derive
\begin{align*}
1+t^*t^{**} &\supseteq 1+t^*t = 1+a^{-1}b^*ba^{-1} = 1+a^{-1}(a-a^2)a^{-1} = a^{-1},\\
1+t^{**}t^* & \supseteq 1+s^*s = 1+a_*^{-1}bb^*a_*^{-1} = 1+a_*^{-1}(a_*-a_*^2)a_*^{-1} = a_*^{-1}.
\end{align*}
Hence $a_{t^{**}}=a\in\Adj(E)$ and $a_{t^*}=a_*\in\Adj(F)$, so that $t\in\Reg_{gr}(E,F)$. Further, we have  $b_{t^{**}} = t^{**}a_{t^{**}} = t^{**}a \supseteq ta = b \in \Adj(E,F)$ and so $b_{t^{**}}=b$.
\end{proof}

From $a_t\in\Adj(E)$ it follows at once that the operator $1+t^*t$ is self-adjoint for any $t\in\grReg(E,F)$ by Proposition \ref{Graph_tadj}.

\begin{lem}\label{Function_atbt}
If $t\in\grReg(E,F)$, then $f(a_{t^*})b_t=b_tf(a_t)$ for all $f\in C([0,1])$.
\end{lem}
\begin{proof}
The operators $a_t,a_{t^*},b_t$ are adjointable and $a_{t^*}b_t=b_ta_t$ by Theorem \ref{aabTransform}. Hence $a_{t^*}^nb_t=b_ta_t^n$ for all $n\in\N$, so $f(a_{t^*})b_t=b_tf(a_t)$ for all polynomials $f$. Since the polynomials are uniformly dense in $C([0,1])$,  the assertion follows.
\end{proof}

\begin{lem}
Let $t\in\oC(E,F)$ and suppose that $t$ and $t^*$ are bounded. Then we have $t\in\grReg(E,F)$ if and only if $t\in\Adj(E,F)$.
\end{lem}
\begin{proof}
The if direction is trivial. To prove the only if part assume that $t$ is graph regular. 
Since $t$ is orthogonally closed, $t$  is   closed. Because $t$ is closed and bounded, the domain $\Def(t)$  is closed in $E$. By Theorem \ref{Basics__graph_regular}(3b), $t^*$ is also graph regular. Therefore,  replacing $t$ by $t^*$,  it follows that  $\Def(t^*)$ is also closed. Hence  $\Def(t^*t)$  is closed. Since $t$ is graph regular, we have $a_t\in\Adj(E)$. Therefore, by \cite[Theorem 3.2]{lance},  $\Range(a_t)=\Def(t^*t)$ is orthogonally complemented. But $\Range(a_t)^\bot=\Null(a_t)=\{0\}$, so $\Range(a_t)=E$. In particular, $E=\Def(t^*t)\subseteq\Def(t)$. Hence $\Def(t)=E$. By a similar reasoning we obtain $\Def(t^*)=F$. Therefore, $t\in \Adj(E,F)$.
\end{proof}

\begin{cor}\label{Characterization_grReg_normal_ab}
Let $t\in\grReg(E)$. Then $t$ is normal if and only if\, $a_t=a_{t^*}$. In this case $b_t$ is normal and the operators $a_t$ and $b_t$ commute.
\end{cor}
\begin{proof}
Since $t\in\grReg(E)$, we have $(a_t,a_{t^*},b_t)\in\mathcal{AB}(E)$ by Theorem \ref{aabTransform}, so
\begin{align*}
b_t^*b_t = a_t-a_t^2, \quad b_tb_t^* = a_{t^*}-a_{t^*}^2, \quad b_ta_t=a_{t^*}b_t.
\end{align*}
The first statement is clear and in this case is $b_t^*b_t=b_tb_t^*$ and $b_ta_t=a_tb_t$.
\end{proof}
In the next proposition $E$ and $F$ are  Hilbert $C^*$-modules of (possibly different !) $C^*$-algebras.
\begin{prop}\label{grReg__Morphism}
Suppose that $t\in\grReg(E)$ and $\phi\in\Hom(\Adj(E),\Adj(F))$. Then there exists an orthogonally closed operator $\phi(t):F\to F$ such that $\phi(a_t)F$ is an essential core for $\phi(t)$ and
\begin{align*}
\phi(t)(\phi(a_t)x) &= \phi(b_t)x ,\quad x\in F.
\end{align*}
Moreover,  $\SP{\phi(t)x}{y}=\SP{x}{\phi(t^*)y}$ for $x\in\Def(\phi(t)),y\in\Def(\phi(t^*))$. If $\Null(\phi(a_t))$ and $\Null(\phi(a_{t^*}))$ are trivial, then $\phi(t)\in\grReg(F)$, $\phi(t)^*=\phi(t^*)$, and
\begin{align*}
a_{\phi(t)} &= \phi(a_t), \quad a_{\phi(t)^*} = \phi(a_{t^*}), \quad b_{\phi(t)} = \phi(b_t).
\end{align*}
\end{prop}
\begin{proof}
Clearly, $0\leq\phi(a_t)\leq I$, $0\leq\phi(a_{t^*})\leq I$, and
\begin{align*}
\phi(b_t)^*\phi(b_t) &= \phi(a_t)-\phi(a_t)^2, \quad \phi(b_t)\phi(b_t)^* = \phi(a_{t^*})-\phi(a_{t^*})^2,\\
\phi(a_t)\phi(b_t)^* &= \phi(b_t)^*\phi(a_{t^*}),
\end{align*}
since $\phi$ is a $*$-homomorphism. If $x\in\Null(\phi(a_t))$, then
\begin{align*}
\SP{\phi(b_t)x}{\phi(b_t)x} &= \SP{x}{\phi(b_t^*b_t)x} = \SP{x}{\phi(a_t)x-\phi(a_t)^2x} = 0,
\end{align*}
so $x\in\Null(\phi(b_t))$. Therefore, the map $\phi(t)_0:\phi(a_t)x\mapsto\phi(b_t)x$ $(x\in F)$ is well-defined. Similarly, the kernel of $\phi(a_{t^*})$ is contained in the kernel of $\phi(b_t^*)$. Further, it is easy to see that
\begin{align*}
\Graph(\phi(t^*)_0) &\subseteq v\Graph(\phi(t)_0)^\bot = \{ (x,y)\in F\oplus F | \phi(b_t^*)x=\phi(a_t)y, y\in\Null(\phi(a_t))^\bot \}
\end{align*}
and the latter is the graph of an operator. Hence $\phi(t)_0$ and $\phi(t^*)_0$ are orthogonally closable. If we denote the corresponding orthogonal closures by $\phi(t)$ and $\phi(t^*)$, the first half of the proposition is shown. If the kernels of $\phi(a_t)$ and $\phi(a_{t^*})$ are trivial, then $(\phi(a_t),\phi(a_{t^*}),\phi(b_t))\in\mathcal{AB}(F)$ and all statements follow from Theorem \ref{aabTransform}.
\end{proof}

If the kernel of $\phi(a_t)$ is not trivial, it can happen  that the domain of the operator $\phi(t)$ is only $\{0\}$, see Example \ref{BoundedTranform_notAdjointbale} below.
\begin{cor}\label{grReg_association}
Let $\CAlg{A}$ be a (non-degenerated) concrete $C^*$-algebra on $\Hil$. Let $\phi$ be the embedding of $\Adj(\CAlg{A})=\Multiplier(\CAlg{A})$ into $\Be(\Hil)=\Adj(\Hil)$.
\begin{enumerate}
\item For any $T\in\Abg(\Hil)$ with $a_T,a_{T^*},b_T\in\Multiplier(\CAlg{A})$ there exists a unique $t\in\grReg(\CAlg{A})$ such that $\phi(t)=T$.
\item If $\CAlg{A}$ contains the compact operators, then we have $T:=\phi(t)\in\Abg(\Hil)$ and $a_T,a_{T^*},b_T\in\Multiplier(\CAlg{A})$ for $t\in\grReg(\CAlg{A})$.
In particular,  $\grReg(\CAlg{A})$ can be identified with those $T\in\Abg(\Hil)$ for which $a_T,a_{T^*},b_T\in\Multiplier(\CAlg{A})$.
\end{enumerate}
\end{cor}
\begin{proof}
(1): Since $\Abg(\Hil)=\grReg(\Hil)$, we have $(a_T,a_{T^*},b_T)\in\mathcal{AB}(\Hil)$ by Theorem \ref{aabTransform}. By assumption $a_T,a_{T^*}$, and $b_T$ are  elements of $\Multiplier(\CAlg{A})$. To show that $(a_T,a_{T^*},b_T)\in\mathcal{AB}(\CAlg{A})$ it suffices  to prove that $a_T$ and $a_{T^*}$ are injective as operators on $\CAlg{A}$. Clearly, they are injective as operators on $\Hil$. Assume that $a_Ta=0$ for some $a\in\CAlg{A}$. Then,  $a_Ta\xi=0$ for $\xi\in\Hil$. Hence $a\xi=0$ for all $\xi\in\Hil$, so that $a=0$. Thus, $a_T$  is injective on $\CAlg{A}$. Similarly, $a_{T^*}$ is injective on $\CAlg{A}$. Using once more Theorem \ref{aabTransform}  it follows that there exists an operator $t\in\grReg(\CAlg{A})$ such that $a_t=a_T$, $a_{t^*}=a_{T^*}$, $b_t=b_T$. Further,  $\phi(t)=T$, since
\begin{align*}
T(\phi(a_t)\xi) &= T(a_t\xi) = Ta_T\xi = b_T\xi = b_t\xi = \phi(b_t)\xi, \quad \xi\in\Hil,
\end{align*}
and $\Range(a_T)=\Def(T^*T)$ is a core for $T$.

(2): If $t\in\grReg(\CAlg{A})$, then $(a_t,a_{t^*},b_t)\in\mathcal{AB}(\CAlg{A})$. We  show that the kernels of $\phi(a_t)$ and $\phi(a_{t^*})$ are trivial. Assume that $\phi(a_t)\xi=0$ for some nonzero vector $\xi\in\Hil$. Since $\CAlg{A}$ contains all compact operators, the rank
one projection $p_\xi$ onto $\C{\cdot} \xi$ is in $\CAlg{A}$. Therefore,  $a_tp_\xi=\phi(a_t)p_\xi=0$ which  contradicts the injectivity of $a_t$ as operator on $\CAlg{A}$. Hence $a_t$,  similarly $a_{t^*}$, is injective on $\Hil$.  Therefore, $T\in\grReg(\Hil)=\Abg(\Hil)$ by Proposition \ref{grReg__Morphism}. \end{proof}

\begin{exa}
From Corollary \ref{grReg_association} it follows immediately that $\grReg(\Komp(\Hil))=\Abg(\Hil)$ and $\grReg(\Be(\Hil))=\Abg(\Hil)$, since $\Multiplier(\Komp(\Hil))=\Multiplier(\Be(\Hil))=\Be(\Hil)$.
\end{exa}

  \subsection{Quotients of adjointable operators}

A large class of examples of unbounded graph regular operators can be obtained as quotients of adjointable operators.

\begin{thm}\label{ClosedQuotient_is_GraphRegular}
Let $a\in\Adj(G,E)$ and $b\in\Adj(G,F)$. Suppose that $\Null(a)\subseteq\Null(b)$ and $\Null(a^*)=\{0\}$. If the operator $t:E\to F$ defined by
\begin{align*}
\Def(t) &= \Range(a), \quad t(ax) := bx, \quad x\in G,
\end{align*}
is closed, then $t\in\grReg(E,F)$ and $t^*=(a^*)^{-1}b^*$.
\end{thm}
\begin{proof}
Since $\Null(a)\subseteq\Null(b)$ and $\Range(a)^\bot=\Null(a^*)=\{0\}$, $t$ is well-defined and essentially defined. Since the graph of $t$ is the set $\{(ax,bx)|x\in G\}$,  it is the range of the adjointable operator $q:G\to E\oplus F$ defined by  $q(x):=(ax,bx)$. Since this range is closed by assumption this range,  \cite[Theorem 3.2]{lance} applies and shows that the range is orthogonally complemented. Hence $t$ is graph regular. The adjoint of $t$ is then easily computed; we omit the details.
\end{proof}

\begin{cor}\label{Inverse_Multiplier__GraphRegular}
Let $x\in\Adj(F,E)$ and assume that $\Null(x)=\Null(x^*)=\{0\}$. Then $x^{-1}\in\grReg(F,E)$ and $(x^{-1})^*=(x^*)^{-1}$.
\end{cor}
\begin{proof}
Since $x^{-1}$ is closed, the assertion follows from Theorem \ref{ClosedQuotient_is_GraphRegular} by letting  $b$ the identity on $F$.
\end{proof}

\begin{cor}\label{ClosedQuotient_Sufficient_Polynomial}
Let $a\in\Adj(E)$ and let $p,q\in\C[X]$ be relatively prime. Assume that $\Range(q(a))$ is essential and $\Null(q(a))\subseteq\Null(p(a))$. Let $t:E\to E$ be the operator defined by
\begin{align*}
\Def(t) &:= \Range(q(a)), \quad t(q(a)x) := p(a)x, \quad x\in E.
\end{align*}
Then $t$ is graph regular.
\end{cor}
\begin{proof}
In order to apply Theorem \ref{ClosedQuotient_is_GraphRegular} we only have to prove that $t$ is closed. Let $(x_n)_{n\in\N}$ be a sequence in $E$ such that $p(a)x_n\to x_p\in E$ and $q(a)x_n\to x_q\in E$. Since $p$ and $q$ are relative prime, there are polynomials $\tilde{p},\tilde{q}\in\C[X]$ such that $\tilde{p}p+\tilde{q}q=1$. Then 
$$x_n=(\tilde{p}(a)p(a)+\tilde{q}(a)q(a))x_n\to\tilde{p}(a)x_p+\tilde{q}(a)x_q=:x_r\in E,$$ so $x_p=p(a)x_r$ and $x_q=q(a)x_r$. That is, $t$ is closed.
\end{proof}

\begin{thm}\label{Graph_regular__admissible_pair}
Let $a\in\Adj(G,E)$, $b\in\Adj(G,F)$, $a_*\in\Adj(H,F)$, and $b_*\in\Adj(H,E)$ be such that $b^*a_*=a^*b_*$. Assume that $\Null(a^*)=\Null(a_*^*)=\{0\}$.  Then $\Null(a)\subseteq\Null(b)$ and $\Null(a_*)\subseteq\Null(b_*)$. The operators $t$ and $t'$ defined by
\begin{align*}
\Def(t) &:= \Range(a), \quad t(ax) := bx ,\quad x\in G,\\
\Def(t') &:= \Range(a_*), \quad t'(a_*y) := b_*y, \quad y\in H,
\end{align*}
are essentially defined, orthogonally closable and they satisfy $(t')^{**}\subseteq t^*$, $t^{**}\subseteq(t')^*$. 
If in addition $t$ and $t'$ are closed and $ab^*=b_*a_*^*$, then $t^*=t'$ and $t\in\grReg(E,F)$.
\end{thm}
\begin{proof}
Suppose that $a_*x=0$  for some $x\in H$. Then $0=b^*a_*a=a^*b_*x$ and hence $b_*x=0$, since  $\Null(a^*)=\{0\}$. This shows that $\Null(a_*)\subseteq\Null(b_*)$. In a similar manner, the assumption $ \Null(a_*^*)=\{0\}$ implies that $\Null(a)\subseteq\Null(b)$. Hence the operators $t$ and  $t'$ are well-defined. It is obvious that  $t$ and $t'$ are  essentially defined. 

From the relations $(a^*)^{-1}b^*a_*=b_*$ and $t^*y=(a^*)^{-1}b^*y, y\in G$, we get  $t'\subseteq t^*$. 
Since $t'\subseteq t^*$ and\, $t'$ is essentially defined, so is\, $t^*$. Since $t$ is also essentially defined, $t$ is orthogonally closable.
Interchanging the role of $t$ and $t'$ we conclude that $t\subseteq (t')^*$ and $t'$ is orthogonally closable. Applying the involution to the relations $t\subseteq (t')^*$ and $t'\subseteq t^*$ we obtain $(t')^{**}\subseteq t^*$ and $t^{**}\subseteq(t')^*$ which proves the first half of the proposition.

Now suppose  that $t$ and $t'$ are closed and $ab^*=b_*a_*^*$. Since $t$ and $t'$ are closed,
\begin{align*}
\mathcal{G} &:= \Graph(t)\oplus v\Graph(t') = \{ (ax+b_*y,bx-a_*y) | x\in E,y\in F \}
\end{align*}
is a closed submodule of\, $E\oplus F$. We define $q(x,y):=(ax+b_*y,bx-a_*y)$ for $(a,b)\in E\oplus F$. Then $q\in\Adj(E\oplus F)$ and $\Range(q)=\mathcal{G}$. By \cite[Theorem 3.2]{lance}, $\mathcal{G}$ is orthogonally complemented. It is easily calculated that $$q^*(x',y')=(a^*x'+b^*y',b_*^*x'-a_*^*y')\quad {\rm for}\quad (x',y')\in E\oplus F.$$ 
We show that $\Null(q^*)=\{0\}$. Suppose that $q^*(x',y')=0$. Then $a^*x'+b^*y'=0$ and $a_*^*y'-b_*^*x'=0$, so we obtain $$0=aa^*x'+ab^*y'=aa^*x'+b_*a_*^*y'=aa^*x'+b_*b_*^*x'.$$ 
The latter implies that $\SP{a^*x'}{a^*x'}+\SP{b_*^*x'}{b_*^*x'}=0$. Thus,  $a^*x'=0$ and $b_*^*x'=a_*^*y'=0 $. Therefore, $x'=0$ and $y'=0$ by the assumption $\Null(a^*)=\Null(a_*^*)=\{0\}$. That is, $\Null(q^*)=\{0\}$. Hence, $\mathcal{G}^\bot=\Range(q)^\bot=\Null(q^*)=\{0\}$. Therefore, since $\mathcal{G}$ is orthogonally complemented, we have $\mathcal{G}=E\oplus F$. This proves that $t'=t^*$.
\end{proof}

  \subsection{Absolute value}

The next theorem is concerned with the absolute value of graph regular operators.

\begin{thm}\label{AbsoluteValue_GrReg}
Suppose that $t\in\grReg(E)$ and define
\begin{align*}
\Def(|t|) &:= \Range(a_t^{1/2}), \quad |t|(a_t^{1/2}x) := (1-a_t)^{1/2}x, \quad x\in E.
\end{align*}
Then $|t|\in\grReg(E)$ is self-adjoint and positive. Further, we have\,   $|t|^2=t^*t$,  $a_{|t|}=a_t$, and $b_{|t|}=|b_t|$.
\end{thm}
\begin{proof}
Clearly, $|t|$ is essentially defined. We prove that $|t|$ is closed. Let $(x_n)$ be a sequence in $E$ such that $a_t^{1/2}x_n\to x\in E$ and $(1-a_t)^{1/2}x_n\to y\in E$. Then $x_n=(1-a_t)x_n+a_tx_n\to a_t^{1/2}x+(1-a_t)^{1/2}y=:x'$, so that $x=a_t^{1/2}x'$. This shows that $|t|$ is closed. Therefore, $|t|$ is graph regular by Proposition \ref{ClosedQuotient_is_GraphRegular}.

Applying Theorem \ref{Graph_regular__admissible_pair} with $a=a_*=a_t^{1/2}$ and $b=b_*=(1-a_t)^{1/2}$ we conclude that $|t|=|t|^*$, since the relations $b^*a_*=a^*b_*$ and $ab^*=b_*^*a_*$ are fulfilled. Further,  $\Def(|t|^2)\supseteq\Range(a_t)=\Def(t^*t)$. It is easily checked that $(1+|t|^2)a_t=1=(1+t^*t)a_t$, so $|t|^2\supseteq t^*t$. Since $t^*t$ is self-adjoint and $|t|^2=|t|^*|t|$ is symmetric, we obtain  $|t|^2=t^*t$. We derive
\begin{align*}
\SP{|t|(a_t^{1/2}x)}{a_t^{1/2}x} = \SP{(1-a_t)^{1/2}}{a_t^{1/2}x} = \SP{(a_t-a_t^2)^{1/2}x}{x} \geq 0\quad {\rm for}\quad x\in E,
\end{align*}
so $|t|$ is positive. Clearly, we have $a_t=a_{|t|}$. Finally, we compute
\begin{align*}
b_{|t|} &= |t|a_t = |t|a_t^{1/2}a_t^{1/2} = (1-a_t)^{1/2}a_t^{1/2} = (a_t-a_t^2)^{1/2} = (b_t^*b_t)^{1/2} = |b_t|.
\end{align*}
\end{proof}

\begin{rem}
In contrast to the Hilbert  space case  the domains $\Def(t)$ and $\Def(|t|)$ do not coincide in general, even more, neither $\Def(t)\subseteq\Def(|t|)$ nor $\Def(t)\supseteq\Def(|t|)$ holds. Indeed, let $E=\CAlg{A}:=C([0,1])$ and set $m(x):=x^{-1}e^{i/x}$ for $x\in(0,1]$. Then,  by Theorem \ref{maintheoremcont} below,  the operator $t_m$ is graph regular, since $\reg(m)=(0,1]$ and $\singsuppr(m)=\emptyset$. It is easily verified that $|t_m|=t_{|m|}$. Define $f(x):=xe^{-i/x}$ for $x\in(0,1]$ and $f(0)=0$. Then $f\in\Def(t_m)$, but $f\notin\Def(|t_m|)$. The function $g(x)=x$ is in $\Def(|t_m|)$, but not in  $\Def(t_m)$.
\end{rem}

  \subsection{Bounded transform}

The main results of this section give a  characterization of graph regular operators in terms of the bounded transform.

Let us begin with some important notation.
Let $\mathcal{Z}(E,F)$ denote the set of all $z\in\Adj(E,F)$ such that $\|z\|\leq 1$ and $\Null(I-z^*z)=\{0\}$ and let $\mathcal{Z}^d(E,F)$ be the set of those $z\in\mathcal{Z}(E,F)$ for which $\Range(I-z^*z)$ is dense in $E$. 

In \cite[Lemma 10.3]{lance} it was shown that  $z\in\mathcal{Z}^d(E,F)$ implies that $z^*\in\mathcal{Z}^d(F,E)$. The following lemma contains the analogues statement  for $\mathcal{Z}(E,F)$.

\begin{lem}
If $z\in\Adj(E,F)$, then $\Null(I-z^*z)=\{0\}$ if and only if $\Null(I-zz^*)=\{0\}$. In particular, $z\in\mathcal{Z}(E,F)$ if and only if $z^*\in\mathcal{Z}(F,E)$.
\end{lem}
\begin{proof}
It suffices to show one direction, since $z$ can be replaced by $z^*$. Assume that $x\in\Null(I-zz^*)\setminus\{0\}$. Then $\|z^*x\|^2=\SP{x}{zz^*x}=\SP{x}{x}=\|x\|^2$. Hence $z^*x\neq 0$. But $(I-z^*z)z^*x=z^*(I-zz^*)x=0$, so $z^*x\in\Null(I-z^*z)\neq\{0\}$.
\end{proof}

For $z\in\mathcal{Z}(E,F)$ we define an operator $t_z:E\to F$ by
\begin{align*}
\Def(t_z) &:= (I-z^*z)^{1/2}E, \quad t_z(I-z^*z)^{1/2}x := zx, \quad x\in E.
\end{align*}
Since $\Range((I-z^*z)^{1/2})^\bot=\Null((I-z^*z)^{1/2})=\Null(I-z^*z)=\{0\}$, $t_z$ is essentially defined. Clearly, $\Range(z)=\Range(t_z)$.

\begin{thm}\label{boundedTransform__tz}
If $z\in\mathcal{Z}(E,F)$, then $t_z\in\grReg(E,F)$. Further, the mapping $z\to t_z$ is injective from $\mathcal{Z}(E,F)$ into $\grReg(E,F)$. We have $t_z^*=t_{z^*}$ and
\begin{align*}
a_{t_z} &= I-z^*z, \quad z=t_za_{t_z}^{1/2}.
\end{align*}
\end{thm}
\begin{proof}
First we prove that the operator $t_z$  is closed. Let $(x_n)$ be a sequence of $\Def(z)$ such that $((I-z^*z)^*)^{1/2}x_n\to x$ and $zx_n\to y$. Then $$x_n=(I-z^*z)^*x_n+z^*zx_n\to((I-z^*z)^*)^{1/2}x+z^*y,$$  so  $x\in\Def(t_z)$ and $t_zx=y$. Thus, $t_z$ is closed and
$t_z\in\grReg(E,F)$ by Proposition \ref{ClosedQuotient_is_GraphRegular}.

Set\, $a=(1-z^*z)^{1/2}$, $b=z$, $a_*=(1-zz^*)^{1/2}$, $b_*=z^*$. Since $ z^*(1-zz^*)^{1/2}=(1-z^*z)^{1/2}z^*$, we then have $b^*a_*=a^*b_*$ and $ab^*=b_*a_*^*$. Hence Theorem \ref{Graph_regular__admissible_pair} applies and yields  $z_t^*=z_{t^*}$.  Finally, by Proposition \ref{Adjoint__Addition_Compositon}(4), $t_z^*=((I-z^*z)^{1/2})^{-1}z^*$, so
\begin{align*}
t_z^*t_z &= ((I-z^*z)^{1/2})^{-1}z^*z((I-z^*z)^{1/2})^{-1}\\
&= ((I-z^*z)^{1/2})^{-1}(I-z^*z)((I-z^*z)^{1/2})^{-1} - (((I-z^*z)^{1/2})^{-1})^2\\
&= I - ((I-z^*z)^{1/2})^{-1}((I-z^*z)^{1/2})^{-1}.
\end{align*}
Therefore, $(I+t_z^*t_z)^{-1}=I-z^*z$. In particular, $\Def(t_z^*t_z)^\bot=\Range(I-z^*z)^\bot=\{0\}$ and $((I+t_z^*t_z)^{-1})^*=(I+t_z^*t_z)^{-1}$. This also implies that $t_z^*t_z$ is self-adjoint and $t_z((I+t_z^*t_z)^{-1})^{1/2}=z$.
\end{proof}

According to  \cite[Theorem 10.4]{lance},  the mapping $z\mapsto t_z$ is a bijection from the set $\mathcal{Z}^d(E,F)$ onto the set $\Reg(E,F)$ of regular operators. For the extended mapping acting on $\mathcal{Z}(E,F)$ the situation is more subtle. It is still an injective mapping into the set $\grReg(E,F)$ of graph regular operators, but it is not sujective as shown by Example \ref{BoundedTranform_notAdjointbale} below.

\begin{lem}\label{AbsoluteValue_z_tz}
If $z\in\mathcal{Z}(E,F)$, then $|z|\in\mathcal{Z}(E)$ and $|t_z|=t_{|z|}$. Further, we have $\Null(z)=\Null(t_z)=\Null(|z|)=\Null(|t_z|)$.
\end{lem}
\begin{proof}
Since $1-z^*z=1-|z|^2$ and $\||z|\|=\|z\|\leq 1$, the first statement is clear. Further, $a_{t_z}=(1+t_z^*t_z)^{-1}=1-z^*z$, so
\begin{align*}
\Def(|t_z|) &= \Range(a_{t_z}^{1/2}) = \Range((1-z^*z)^{1/2}) = \Range(1-|z|^2) = \Def(t_{|z|}),\\
|t_z|(1-z^*z)x &= (1-a_{t_z})^{1/2}x = (z^*z)^{1/2}x = |z|x = t_{|z|}(1-|z|^2)x, \quad x\in E,
\end{align*}
that is, $|t_z|=t_{|z|}$. Since  kernels of orthogonally closed operators are orthogonally closed, we obtain $\Null(z)=\Range(z^*)^\bot=\Range(t_{z^*})^\bot=\Range(t_z^*)^\bot=\Null(t_z)$. Because $\Null(z)=\Null(|z|)$, this completes the proof.
\end{proof}
The following lemma restates \cite[Proposition 3.7]{lance}. It will be used several times in the proof of Theorem  \ref{BoundedTranform_grReg} below.
\begin{lem}\label{rangelemmalance}
If $a\in \Adj(E)_+$, then\, $\overline{\Range(a)}=\overline{\Range(a^\varepsilon)}$\, for any $\varepsilon>0$.
\end{lem}
Now suppose that $t\in\grReg(E,F)$. Recall that $a_t\in\Adj(E)_+$, $a_{t^*}\in\Adj(F)_+$, and $b_t=b_{t^*}^*\in\Adj(E,F)$ by Theorem \ref{aabTransform}. The operator $ta_t^{1/2}$ is essentially defined, since $\Range(a_t^{1/2})$ is contained in its domain. Its adjoint is an extension of $a_t^{1/2}t^*$, hence it is also essentially defined. Moreover, by Proposition \ref{Adjoint__Addition_Compositon}(3),  $(a_t^{1/2}t^*)^*=ta_t^{1/2}$. Hence  $ta_t^{1/2}$ is orthogonally closed; in particular, $ta_t^{1/2}$ is closed. 
\begin{dfn} The \emph{bounded transform} $z_t$ of\,  $t\in\grReg(E,F)$ is defined by
\begin{align*}
z_t &:= ta_t^{1/2}\upharpoonright_{\ol{\Def(t^*t)}}.
\end{align*}
\end{dfn}
Since $t\in \grReg(E,F)$, $\Def(t^*t)$ is essential in $E$ and $\Def(tt^*)$ is essential in $F$.

\begin{thm}\label{BoundedTranform_grReg}
Suppose $t\in\grReg(E,F)$. Set $E_0:=\ol{\Def(t^*t)}$ and $F_0:=\ol{\Def(tt^*)}$. Then $z_t\in\mathcal{Z}^d(E_0,F_0)$ and $z_t^*=z_{t^*}$, where the adjoint $z_t^*$ is taken in $\Adj(E_0,F_0)$. Further,
\begin{align*}
z_t &= \ol{a_{t^*}^{1/2}t}\upharpoonright_{\ol{\Def(t^*t)}}, \quad \Null(z_t)=\Null(t), \quad \Range(z_t)\subseteq\Range(t),\\
a_t &= (1-z_{t^*}z_t)^*, \quad a_{t^*} = (1-z_tz_{t^*})^*, \quad b_t=z_ta_t^{1/2},
\end{align*}
and
$t_{\rm reg}:=t_{z_t}\upharpoonright_{E_0}$ is a regular operator from $E_0$ to $F_0$ satisfying\, $t_{\rm reg}\subseteq t =(t_{z_t})^{**}$, 
where $t_{z_t}^{**}$ is the biadjoint of the operator $t_{z_t}:E\to F$.
\end{thm}
\begin{proof}
From Theorem \ref{aabTransform} we already know that $t=a_{t^*}^{-1}b_t$. Using this   we derive
\begin{align}\label{atstarsqrt}
a_{t^*}^{1/2}t &= a_{t^*}^{1/2}a_{t^*}^{-1}b_t \subseteq a_{t^*}^{-1}a_{t^*}^{1/2}b_t = a_{t^*}^{-1}b_ta_t^{1/2} = ta_t^{1/2}.
\end{align}
Here the second equality follows from Lemma \ref{Function_atbt} applied with $f(x)=\sqrt{x}$.

 Now we prove that\, $a_{t^*}^{1/2}t$\, is bounded with norm not exceeding $1$. Let $x\in\Def(t)$. Using that\,
  $t^*a_{t^*}t=b_{t^*}t=b_t^*t\subseteq(t^*b_t)^*=(t^*ta_t)^*=(1-a_t)^*=1-a_t$, we derive
\begin{align*}
\|a_{t^*}^{1/2}tx\|^2 &= \|\SP{a_{t^*}^{1/2}tx}{a_{t^*}^{1/2}tx}\| = \|\SP{t^*a_{t^*}tx}{x}\| \\&= \|\SP{(1-a_t)x}{tx}\|
\leq \|1-a_t\|\|x\|^2 \leq \|x\|^2.
\end{align*}

By (\ref{atstarsqrt}) the operators  $a_{t^*}^{1/2}t$  and\, $ta_t^{1/2}$ concides on $\Def(t)$ and so on its subspace $\Def(t^*t).$  Since both operators are  bounded on  $\Def(t^*t)$ and $ta_t^{1/2}$ is closed, we conclude that 
$z_t \equiv ta_t^{1/2}\upharpoonright_{\ol{\Def(t^*t)}}~ =   \ol{a_{t^*}^{1/2}t}\upharpoonright_{\ol{\Def(t^*t)}}$\,.

The latter equality   implies that $\Range(z_t)$ is contained in $\ol{\Range(a_{t^*}^{1/2})}$.  Lemma \ref{rangelemmalance}\, yields $\ol{\Range(a_{t^*}^{1/2})}=\ol{\Range(a_{t^*})}=\ol{\Def(tt^*)}=F_0$, so that $\Range(z_t)\subseteq F_0.$ Hence $z_t$ becomes a bounded operator from $E_0$  into $F_0$. Analogously, $z_{t^*}$ is a  bounded operator from $F_0$ into  $E_0$.

We show that $z_t^*=z_{t^*}$. Let $z_t^{*_E}$ denote the adjoint of $z_t$ considered as an operator from $E$ into $F$. Clearly,
$z_t^{*_E} \supseteq (\ol{a_{t^*}^{1/2}t})^* \supseteq t^*a_{t^*}^{1/2} \supseteq z_{t^*.}$
This implies that $z_t^*=z_t^{*_E}\upharpoonright_{F_0}=z_{t^*}$. Therefore, $z_t\in\Adj(E_0,F_0)$.

Next we verify the formulas for $a_t,a_{t^*}$, and $b_t$. First we note that
\begin{align}\label{ztztszarat}
z_{t^*}z_t\upharpoonright_{\Def(t^*t)} &= t^*a_{t^*}^{1/2}\upharpoonright_{F_0}a_{t^*}^{1/2}t\upharpoonright_{\Def(t^*t)} = t^*a_{t^*}t\upharpoonright_{\Def(t^*t)} = 1-a_t\upharpoonright_{\Def(t^*t)}.
\end{align}
Hence\, $a_t=(1-z_{t^*}z_t)^{**}$ by Example \ref{Core_Adjointable},   so\, $a_t=a_t^*=(1-z_{t^*}z_t)^{***}=(1-z_{t^*}z_t)^*.$ The formula for $a_{t^*}$ is proven in a similar manner. Since $\Range(a_t^{1/2})\subseteq E_0$, we  obtain $z_ta_t^{1/2}=ta_t^{1/2}\upharpoonright_{E_0}a_t^{1/2}=ta_t=b_t$.

Using relation (\ref{ztztszarat}) and again Lemma \ref{rangelemmalance} we get
\begin{align*}
\ol{\Range(1-z_{t^*}z_t)} &= \ol{\Range(a_t\upharpoonright_{\ol{\Def(t^*t)}})} = \ol{\Range(a_t\upharpoonright_{\ol{\Range(a_t)}})} \supseteq \ol{\Range(a_t\upharpoonright_{\Range(a_t)})}\\&= \ol{\Range(a_t^2)} = \ol{\Range(a_t)} = \ol{\Def(t^*t)} = E_0.
\end{align*}
Hence $z_t\in \mathcal{Z}^d(E_0,F_0)$. Therefore, $t_{\rm reg}=t_{z_t}$ is a regular operator from $E_0$ to $F_0$. Clearly, $t_{\rm reg}\subseteq t.$

Now we prove that\, $t=(t_{z_t})^{**}$. Since $t_{z_t}\subseteq t$  and
 \begin{align*}
\Def(t_{z_t})=\Range((1-z_t^*z_t)^{1/2})=\Range(a_t\upharpoonright_{\ol{\Def(t^*t)}})= \Range(a_t\upharpoonright_{\ol{\Range(a_t)}})\supseteq\Range(a_t^{2}),
\end{align*}
it suffices to show that  $\Range(a_t^{2})$ is an essential core for $t$. Assume $(x,tx)\bot\Graph(t\upharpoonright_{\Range(a_t^2)})$ for some $x\in\Def(t)$. Then, for all $y\in E$,
\begin{align*}
0=\SP{(x,tx)}{(a_t^{2}y,ta_t^{2}y)} &= \SP{x}{a_t^{2}y}+\SP{tx}{ta_t^{2}y} = \SP{a_tx}{a_t y}+\SP{tx}{b_ta_t y}\\
&= \SP{a_tx}{a_t y}+\SP{b_t^*tx}{a_t y}=\SP{a_tx+b_t^*tx}{a_t y}
\end{align*}
Therefore,  since $\Range(a_t)$ is essential, $0=(a_t+b_t^*t)x=(a_t+1-a_t)x=x$. That is, $\Graph(t)\cap\Graph(t\upharpoonright_{\Range(a_t^{2})})^\bot=\{0\}$. Since $t$ is graph regular and hence $\Graph(t)\oplus\Graph(t)^\bot=E$, it follows from  Lemma \ref{technic_core}  that
\begin{align*}
\Graph(t\upharpoonright_{\Range(a_t^{2})})^\bot &= \Graph(t\upharpoonright_{\Range(a_t^{2})})^\bot \cap (\Graph(t)\oplus\Graph(t)^\bot) = \Graph(t)^\bot.
\end{align*}
Thus $\Range(a_t^{2})$ is an essential core for $t$ which completes the proof of the equality $t=(t_{z_t})^{**}$.

Clearly, $\Range(z_t)\subseteq\Range(t)$. Finally,  we show that $\Null(z_t)=\Null(t)$.  Let $x\in\Null(t)\subseteq\Def(t^*t)$. Since $z_t\supseteq a_{t^*}^{1/2}t\upharpoonright_{\Def(t^*t)}$, we have  $x\in\Def(z_t)$ and $z_tx=0$, so $\Null(t)\subseteq\Null(z_t)$. Conversely, let $x\in\Null(z_t)$. Then $(1-a_t)x=z_{t^*}z_tx=0$ and $x=a_tx\in\Def(t^*t)\subseteq\Def(t)$. We obtain $(1+t^*t)x=x$ and $t^*tx=0$. From $\SP{tx}{tx}=\SP{t^*tx}{x}=0$ we get $tx=0$. Hence  $\Null(z_t)\subseteq\Null(t)$. 
\end{proof}
\begin{dfn}
The operator $t_{\rm reg}:=t_{z_t}\in \Reg(E_0,F_0)$ from Theorem \ref{BoundedTranform_grReg} is called the \emph{regular operator} associated with the graph regular operator $t\in \grReg(E,F)$. 
\end{dfn}

\begin{rem}
There are two other possibilities to define the bounded transform  for a graph regular operator $t$ and both of them  are natural in some sense. Define 
$$
z^\prime_t:=ta_t^{1/2}\quad {\rm and}\quad z^{\prime\prime}_t:=(b_t(a_t^{1/2})^{-1})^{**}.$$
Note that $z^{\prime}_t\in\oC(E,F)$ and $z_t$ is the restriction of $z^\prime_t$ to\, $\overline{\Def(t^*t)}$. It is easily seen  that  $b_t(a_t^{1/2})^{-1}$ is the restriction of\, $z_t^\prime=ta_t^{1/2}$ to $\Range(a_t^{1/2})$. Hence $b_t(a_t^{1/2})^{-1}$ is essentially defined and orthogonally closable and its orthogonal closure  may also be taken as a bounded transform $z_t^{\prime\prime}$ of\, $t$. Then\, $z^\prime_t=z^{\prime\prime}_t$ if and only if $\Range(a_t^{1/2})$ is an essential core for  $z_t^{\prime}$. We were not able to prove or disprove this. But one can show that 
\begin{align*}
z_t &\subseteq z^{\prime\prime}_t \subseteq z^{\prime}_t, \quad (z^{\prime\prime}_t)^* = z^\prime_{t^*}\equiv t^*a_{t^*}^{1/2}, \quad (z^{\prime}_t)^* = z^{\prime\prime}_{t^*}\equiv (b_{t^*}(a_{t^*}^{1/2})^{-1})^{**}.
\end{align*}
Note that the operator $t$ can be recovered from both transforms $z^{\prime}_t$ and $z^{\prime\prime}_t$ as well.
We dont know wether or not the equality\, $(ta_t^{1/2})^*=t^*a_{t^*}^{1/2}$ holds.
\end{rem}
\begin{rem}
The relationship between  $t\in \grReg(E,F)$,  $t_{\rm reg}\in \Reg(E_0,F_0)$, and their bounded tranforms should be studied be further in detail. Despite of this it seems that graph regular operators form an important notion in its own, because  they act  on the given Hilbert $C^*$-module $E$.  Though each symmetric operator on a Hilbert space is a restriction of a self-adjoint operator in possibly larger space,  symmetric operators are a basic concept. 
\end{rem}

  \subsection{Polar decomposition}

\begin{dfn}
Let $E'\subseteq E$ and $F'\subseteq F$ be orthogonally closed. An operator $v\in\oC(E,F)$ is called a \emph{partial isometry} with \emph{initial space} $E'$ and \emph{final space} $F'$ if $v^*v$ is the projection $p_{E'}$ and $vv^*$ is the projection $p_{F'}$.
\end{dfn}

In this case $v^*$ is also a partial isometry with initial space $F'$ and final space $E'$. Using this 
general notion  of essentially defined partial isometries there is the following theorem on a polar decomposition of adjointable operators.

\begin{thm}\label{PolarDecomposition_AdjointableOperator}
Let $t\in\Adj(E,F)$. Then there is a partial isometry $v\in\oC(E,F)$ with initial space $\ol{\Range(t^*)}$ and final  space $\ol{\Range(t)}$ such that $t=v|t|$, $|t|=v^*t$ and $\Range(v)=\ol{\Range(t)}$, $\Range(v^*)=\ol{\Range(t^*)}$, $\Null(v)=\Null(t)$, and $\Null(v^*)=\Null(t^*)$ if and only if $\ol{\Range(t)}$ and $\ol{\Range(t^*)}$ are orthogonally closed.
\end{thm}
\begin{proof}
The only if direction follows easily from the definition of a partial isometry.

To prove the if part we assume that $\ol{\Range(t)}$ and $\ol{\Range(t^*)}$ are orthogonally closed. 
Define a map $w:\Range(|t|)\to\Range(t)$ by $w(|t|x):=tx$ for $x\in E$. Then $w$ is well-defined and isometric, since $\SP{tx}{tx}=\SP{|t|x}{|t|x}$ for $x\in E$. The continuous extension of $w$ to a map from $\ol{\Range(|t|)}$ onto $\ol{\Range(t)}$ is also an isometry which is denoted again by $w$. Now we define $v:E\to F$ by
\begin{align*}
v(x+y) &:= wx, \quad x\in\ol{\Range(|t|)},\, y\in\Null(t).
\end{align*}
Clearly, $\Def(v)^\bot=(\Null(t)\oplus\ol{\Range(|t|)})^\bot=\Null(t)^\bot\cap\Range(|t|)^\bot=\Null(t)^\bot\cap\Null(|t|)=\{0\}$, so $v$ is essentially defined. Further, $t=v|t|$, $\Null(v)=\Null(t)$ and $\Range(v)=\ol{\Range(|t|)}$. Let
\begin{align*}
v'(x+y) &:= w^{-1}x ,\quad x\in\ol{\Range(t)},\, y\in\Null(t^*).
\end{align*}
As above, $v'$ is essentially defined, $|t|=v't$, $\Null(v')=\Null(t^*)$ and $\Range(v')=\ol{\Range(t)}$. It is easily seen that $v'\subseteq v^*$ and $v\subseteq (v')^*$. Therefore, since $v'$ is essentially defined, so is $v^*$. Hence  $v$ is orthogonally closable by Theorem \ref{Characterization__orthognal_closable}.

We show that $v^*=v'$. Let $y\in\Def(v^*)$. Then there is an element $z\in E$ such that $\SP{v(x+x^\bot)}{y}=\SP{x+x^\bot}{z}$ for all $x\in\ol{\Range(|t|)}$ and $x^\bot\in\Null(t)$. Choosing $x=0$, we conclude that $z\in\Null(t)^\bot=\Range(t^*)^{\bot\bot}=\Range(|t|)^{\bot\bot}=\ol{\Range(|t|)}$. Thus, $\Range(v^*)\subseteq\Range(v')$. Putting now $x^\bot=0$, we get $\SP{tx'}{y}=\SP{v|t|x'}{y}=\SP{|t|x'}{z}$ for all $x'\in E$. Hence $t^*y=|t|z=|t|v^*y$ and $\Null(v^*)\subseteq\Null(t^*)=\Null(v)$. Now $v'\subseteq v^*$, $\Range(v^*)\subseteq\Range(v')$ and $\Null(v^*)\subseteq\Null(v')$ clearly imply that $v'=v^*$. In a similar manner it is shown that  $v'^*=v$. Obviously, $v^*v$ is the projection onto the orthogonally closed submodule $\ol{\Range(t^*)}$ and $vv^*$ the projection onto $\ol{\Range(t)}$.
\end{proof}

\begin{thm}\label{PolarDecomposition_grReg}
Let $z\in\mathcal{Z}(E,F)$. There exists a partial isometry $v\in\oC(E,F)$ with initial space $\ol{\Range(|t_z|)}$ and final space $\ol{\Range(t_z)}$ such that
\begin{align*}
t_z &= v|t_z|, \quad |t_z|=v^*t_z,
\end{align*}
$\Range(v)=\ol{\Range(t_z)}$, $\Range(v^*)=\ol{\Range(t_z^*)}$, $\Null(v)=\Null(t_z)$, and $\Null(v^*)=\Null(t_z^*)$ if and only if $\ol{\Range(z)}$ and $\ol{\Range(z^*)}$ are orthogonally closed. In this case, $z=v|z|$.
\end{thm}
\begin{proof}
Since $\Range(z)=\Range(t_z)$, $\Range(z^*)=\Range(t_z^*)$, $\Null(z)=\Null(t_z)$, and $\Null(z^*)=\Null(t_z^*)$,  one direction follows at once from the definition of a partial isometry.

Conversely, assume that $\ol{\Range(z^*)}=\ol{\Range(|z|)}$ and $\ol{\Range(z)}$ are orthogonally closed. By Proposition \ref{PolarDecomposition_AdjointableOperator}, there is a partial isometry $v\in\oC(E,F)$ with $z=v|z|$, $|z|=v^*z$ and $\Range(v)=\ol{\Range(z)}$, $\Range(v^*)=\ol{\Range(z^*)}$, $\Null(v)=\Null(z)$ and $\Null(v^*)=\Null(z^*)$. Finally, $v|t_z|(1-|z|^2)^{1/2}x=v|z|x=zx=t_z(1-|z|^2)^{1/2}x$ for $x\in E$, so that $v|t_z|=t_z$. Similarly,  $v^*t_z=|t_z|$.
\end{proof}

Note that the partial isometry belongs to $\Adj(E,F)$ if and only if\, $\ol{\Range(z)}$ and $\ol{\Range(z^*)}$ are orthogonal complements: $\Def(v)=\ol{\Range(z^*)}\oplus\Null(z)$ and $\Def(v^*)=\ol{\Range(z)}\oplus\Null(z^*)$.

  \subsection{Graph regular operators on $C_0(X)$}\label{graphC_0(X)}

Before we turn to the functional calculus of graph regular normals we reconsider the commutative case from Section  \ref{Theory_CX}. By Theorem \ref{operatorfunctionstm},  $\oC(C_0(X))$ consists  of multiplication operators. Theorem \ref{maintheoremcont}  characterizes the graph regular operators  among them as those for which $\singsuppr(m)$ is empty.

\begin{lem}\label{Inverse_CommutativeCase}
For any function $m:X\to\C$ the following are equivalent:
\begin{enumerate}
\item $t_m$ is injective if and only if\, $\{x\in\reg(m)|m(x)\neq 0\}$\, is dense in $\reg(m)$.
\item If $t_m$ is injective and $m$ does not vanish on $X$, then $t_m^{-1}=t_{1/m}$.
\end{enumerate}
\end{lem}
\begin{proof}
(1): Set $N:=\{x\in\reg(m)| m(x)=0\}$. Assume that $N$ contains a nonempty open set $U$. Since $X$ is  locally compact and Hausdorff, there is a non-zero function $f\in E$ with  support contained in $U$. Then $mf=0$, $f\in\Def(t_m)$ and  $t_mf=0$. Hence $t_m$ is not injective. On the other hand, assume that $f\in\Def(t_m)$ and $t_mf=0$. Then  $m(x)f(x)=(t_mf)(x)=0$ for $x\in\reg(m)$ . Thus $f\equiv 0$ on $\reg(m)\setminus N$. If the latter is dense in $\reg(m)$, then $f\equiv 0$\, on $\reg(m)$ by the continuity of $f$. Futher, we have $f\equiv 0$ on $X\setminus\ol{\reg(m)}$ by Lemma \ref{Domain_Equivalence__CommutativeCase}, since $f\in\Def(t_m)$. That is, $f=0$.

(2): We have\, $\reg(1/m)=\reg(m)$. In particular, $t_{1/m}$ is injective. From Lemma \ref{AdditionMultiplication_CommutativeCase}(1) it follows that $t_mt_{1/m}$ and $t_{1/m}t_m$ are restrictions of the identity. Therefore, $t_{1/m}\subseteq t_m^{-1}$ and $t_m\subseteq t_{1/m}^{-1}$. The last inclusion  gives $t_m^{-1}\subseteq t_{1/m}$, so equality is proven.
\end{proof}

\begin{thm}\label{maintheoremcont}
If $m:X\to\C$, then $t_m\in\grReg(C_0(X))$ if and only if\, $\reg(m)$ is dense in $X$ and\, $\singsuppr(m)$ is empty. In this case we have  
$t_m^*t_m = t_{|m|^2}$ and
\begin{align*}
a_{t_m} &= t_{\frac{1}{1+|m|^2}}, \quad b_{t_m} = t_{\frac{m}{1+|m|^2}}.
\end{align*}
\end{thm}
\begin{proof}
The first assertion is clearly a corollary of Theorem \ref{maintheorem}. 

Suppose that $t_m\in\grReg(E).$ Then $a_{t_m}$ is self-adjoint, hence $t_m^*t_m$ is self-adjoint. Further, by Lemma \ref{AdditionMultiplication_CommutativeCase}, $t_m^*t_m$ is contained in the self-adjoint operator $t_{|m|^2}$. Hence $t_m^*t_m = t_{|m|^2}$. Finally, using Lemma \ref{AdditionMultiplication_CommutativeCase}(2) and Lemma \ref{Inverse_CommutativeCase}(2), we compute
\begin{align*}
a_{t_m} &= (1+t_m^*t_m)^{-1} = (1+t_{|m|^2})^{-1} = (t_{1+|m|^2})^{-1} = t_{\frac{1}{1+|m|^2}},\\
b_{t_m} &= b_{t_m}^{**} = (t_ma_{t_m})^{**} = (t_mt_{\frac{1}{1+|m|^2}})^{**} = t_{\frac{m}{1+|m|^2}}.
\end{align*}
\end{proof}

\begin{cor}
Suppose that $m:X\to{\C}$ is bounded and $\reg(m)$ dense in $X$. Then $t_m$ is graph regular if and only if\, $\reg(\hat{m})=X$. In particular, we have  $t_m\in\Adj(C_0(X))=C_b(X)$ in this case.
\end{cor}
\begin{proof}
Since $m$ is bounded,  $\reg_\infty(m)$ is empty. Hence the statement follows directly from Theorem \ref{maintheoremcont}.
\end{proof}

The next example shows that not all operators $t\in\grReg(C_0(\R))$ are of the form $t_z$ for some $z\in\mathcal{Z}(C_0(\R))$. Moreover there is a representation $\pi$ of $\Adj(C_0(\R))$ such that the domain of $\pi(t)$ consists only of the zero element.

\begin{exa}\label{BoundedTranform_notAdjointbale}
Consider the operator $t:=t_m$ on $C_0(\R)$, where  $m(x)=1/x$ for $x\neq 0$. Then $t$ is self-adjoint and graph regular by Theorem \ref{maintheoremcont}. 

We show that there is no $z\in\mathcal{Z}(C_0(\R))$ such that $t=t_z$. Assume to the contrary that $t=t_z$ for $z\in\mathcal{Z}(C_0(\R))$. Then $z=ta_t^{1/2}$ by Theorem \ref{boundedTransform__tz}. 
Further, we have $a_t=t_{1/(1+|m|^2)}$ and therefore $a_t^{1/2}=t_{1/(\sqrt{1+|m|^2}\,)}$. Choose $g\in C_0(\R)$ with $g(0)\neq 0$. Then $a_t^{1/2}g(x)=|x|g(x)/(\sqrt{1+x^2})$. Since the function $\frac{1}{x}\frac{|x|g(x)}{\sqrt{1+x^2}}$ is continuous on $\R\setminus\{0\}$ and has no continuous extension to $\R$, we have $a_t^{1/2}g\notin\Def(t)$. Hence   $ta_t^{1/2}$ is not everwhere defined which contradicts the equality $z=ta_t^{1/2}.$ 

Let $\pi$ be the $*$-representation of $\Adj(C_0(\R))=\{t_n | n\in C_b(\R)\}$ given by $\pi(t_n)= n(0)$. Then $\pi(a_t)=0$, so the domain of\, $\pi(t)$ is  $\{0\}$ (compare Proposition \ref{grReg__Morphism}).
\end{exa}

\begin{exa}[Continuing Example \ref{Multiplier__different_Dm_Dm*}]\label{StrangeNormalOperator}
Recall that the operator $t_m$ from  Example \ref{Multiplier__different_Dm_Dm*} is normal, but $\Def(t_m)\neq\Def(t_m^*)$. Since $\reg(m)=(0,1]$ and $0\in\reg_\infty(m)$,  $t_m$ is  graph regular by Theorem \ref{maintheoremcont}. That is,  even for graph regular operators $t$ the  statements 
\begin{enumerate}
\item $t^*t=tt^*$ (that is, $t$ is normal),
\item $\Def(t)=\Def(t^*)$ and $\|tf\|=\|t^*f\|$ for all $f\in\Def(t)$,
\end{enumerate}
are \emph{not} equivalent! 
\end{exa}

  \subsection{Functional calculus of graph regular normal operators}

Let $\CAlg{A}$ and $\CAlg{B}$ be $C^*$-algebras, where $\CAlg{A}$ is non-unital and $\CAlg{B}$ is unital. Clearly, each $*$-homomorphism $\phi:\CAlg{A}\to\CAlg{B}$ extends uniquely to a $*$-homomorphism of the unitization $\CAlg{A}^\sim:=\CAlg{A}\oplus\C$ by $\phi(a+\alpha):=\phi(a)+\alpha 1$ for $a\in\CAlg{A}$, $\alpha\in\C$.

Let $\zeta$ denote the identity map of $\C$. Considered  as an operator on $C_0(\C)$, $\zeta$  is a regular operator, but on the   unitization  $C_0(\C)^\sim$ the operator $\zeta$ is no longer regular. On the other hand, since $a_\zeta=a_{\ol{\zeta}}=(1+|\zeta|^2)^{-1}\in C_0(\C)$  the operator $\zeta$ is graph regular according to Theorem \ref{aabTransform}. Further, $b_\zeta=\zeta (1+|\zeta|^2)^{-1}\in C_0(\C)$.

\begin{thm}\label{grReg_FunctionalCalculus}
Let $E$ be a Hilbert $\CAlg{A}$-module and let $t\in\grReg(E)$ be normal. Then there exists a unique $\phi_t\in\Hom(C_0(\C)^\sim,\Adj(E))$ with $\Null(\phi_t(a_\zeta))=\{0\}$ and $\phi_t(\zeta)=t$.
\end{thm}
\begin{proof}
Set
\begin{align*}
D &:= \{ z\in\C : |z|\leq 1/2 \}, ~ F := \{ (z_1,z_2)\in[0,1]\times D : |z_2|^2=z_1-z_1^2 \} \subseteq [0,1]\times D.
\end{align*}
Then $\partial F=\{(0,0)\}$ and $\mathring{F}=F\setminus\partial F$. By Corollary \ref{Characterization_grReg_normal_ab},  $a_\zeta$ is self-adjoint, $b_\zeta$ is normal, and $a_\zeta$ and $b_\zeta$  commute. Their joint spectrum $\sigma(a_\zeta,b_\zeta)$ is contained in $F$. Similar statements hold for $a_t$ and $b_t$.

Uniqueness: Let $\phi\in\Hom(C_0(\C)^\sim,\Adj(E))$ be such that  $\Null(\phi_t(a_\zeta))=\{0\}$ and $\phi(\zeta)=t$. Then,  by Proposition \ref{grReg__Morphism}, $\phi(a_\zeta) = a_{\phi(\zeta)} = a_t$ and $\phi(b_\zeta) = b_{\phi(\zeta)} = b_t.$
For $f\in C(F)$  the functional calculus of commuting bounded normal  operators yields
\begin{align}\label{Assocition__functional_calculus_hoomorphism}
\phi(f(a_\zeta,b_\zeta)) &= f(\phi(a_\zeta),\phi(b_\zeta)) = f(a_t,b_t).
\end{align}
Each function $g+\beta\in C_0(\C)^\sim$ is of the form $f(a_\zeta,b_\zeta)$ for some function $f\in C(F)$ with $f\upharpoonright_{\partial F}\equiv\beta$. Indeed, we have
\begin{align*}
g(z) &= g(a_\zeta(z)^{-1}b_\zeta(z)) = f(a_\zeta,b_\zeta)(z),
\end{align*}
where
\begin{align*}
f(z_1,z_2) := \begin{cases}
g(z_2/z_1)+\beta & , (z_1,z_2)\in\mathring{F}\\
\beta             & , (z_1,z_2)\in\partial F
\end{cases}.
\end{align*}
To show that $f$ is continuous at $\partial F$, assume $(z_1,z_2)\to(0,0)$. From $|z_2|^2=z_1-z_1^2$ it follows $|z_2/z_1)|=\sqrt{1/z_1-1}\to\infty$ since $z_1\to 0$. Therefore $g(z_2/z_1)\to 0$, since $g$ vanishes at infinity. This proves the  uniqueness assertion.

Existence: Equation (\ref{Assocition__functional_calculus_hoomorphism}) defines a $*$-homomorphism from $C_0(\C)^\sim$ into $\Adj(E)$. Inserting  $f(z_1,z_2):=z_1$  into (\ref{Assocition__functional_calculus_hoomorphism}) it follows that $\Null(\phi(a_\zeta))=\Null(a_t)$. Note the latter is trivial. Similarly, $\phi(b_\zeta)=b_t$. Frm Proposition \ref{grReg__Morphism} we get $a_{\phi(\zeta)}=\phi(a_\zeta)=a_t$ and $b_{\phi(\zeta)}=\phi(b_\zeta)=b_t$. From Theorem \ref{aabTransform} we finally conclude that $\phi(\zeta)=t$.
\end{proof}

  \section{Associated operators and affiliated  operators}\label{associated}

Throughout this section  we assume that the Hilbert $\CAlg{A}$-module $E$ is the $C^*$-algebra $\CAlg{A}$ itself equipped with the $\CAlg{A}$-valued scalar product $\SP{a}{b}:=a^*b$, $a,b \in \CAlg{A}$, and  that $\CAlg{A}$ is realized as a nondegenerate $C^*$-algebra on a Hilbert space $\Hil$. 

Then $\Adj(E)$ is the multiplier algebra $\Multiplier(\CAlg{A})=\{ x\in \Be(\Hil): x\CAlg{A}\subseteq \CAlg{A}, \CAlg{A}x\subseteq \CAlg{A}\}$.
and
$\Reg(E)$ is the set $\CAlg{A}^\eta$ of affiliated operators in the sense of Woronowicz \cite{wor91}. Recall that $\CAlg{A}^\eta$ is the set of operators $t\in \Abg(\Hil)$ for which  $ a_t = (I+t^*t)^{-1}\in\Multiplier(\CAlg{A})$, $b_t = t(I+t^*t)^{-1}\in\Multiplier(\CAlg{A})$, and $a_t \CAlg{A}$ is  dense in $\CAlg{A}$. We write $t\eta \CAlg{A}$ if $t\in \CAlg{A}^\eta$. Note that\,  $t=\ol{b_ta_t^{-1}}$\, for $t\eta \CAlg{A}$.

\begin{dfn}
We say that an operator $t\in \Abg(\Hil)$ is   \emph{associated with $\CAlg{A}$} and write $t\mu \CAlg{A}$\, if\, $t\in \grReg(E)$.
The set of associated operators with $\CAlg{A}$ is denoted by $\CAlg{A}^\mu$.
\end{dfn}
That is, by Theorem \ref{aabTransform},  $t\in \Abg(\Hil)$ is in $\CAlg{A}^\mu$ if and only if $a_t, a_{t^*}, b_t\in \Multiplier(\CAlg{A})$. Note that the density of the set $a_t\CAlg{A}$ in $\CAlg{A}$ is not required for $t\mu \CAlg{A}$. Obviously, $\CAlg{A}^\mu\subseteq \CAlg{A}^\eta$. Further,  $t\in \CAlg{A}^\mu$ is in $\CAlg{A}^\eta$ if and only if $a_t\CAlg{A}$ is dense in $\CAlg{A}$. 

\begin{lem}
$\Multiplier(\CAlg{A}) = \{ t\mu \CAlg{A} | t\in\Be(\Hil) \}.$
\end{lem}
\begin{proof}
If $t\in\Multiplier(\CAlg{A})$, then $I+t^*t\in\Multiplier(\CAlg{A})$ and so $a_t\in\Multiplier(\CAlg{A})$ and $b_t=ta_t\in\Multiplier(\CAlg{A})$, hence $t\mu \CAlg{A}$. Conversely, suppose that $t$ is bounded.  Then $I+t^*t$ is bounded and $a_t\in\Multiplier(\CAlg{A})$, so  that $a_t^{-1}=I+t^*t\in\Multiplier(\CAlg{A})$. Therefore, $t=b_ta_t^{-1}\in\Multiplier(\CAlg{A})$.
\end{proof}

For  $t\mu \CAlg{A}$ all three operators $a_t, a_{t^*}, b_t$ have to be in the multiplier algebra $\Multiplier(\CAlg{A})$, while for $t\eta \CAlg{A}$ it is only  required that $a_t,b_t\in\Multiplier(\CAlg{A})$ (and the density of $a_t\CAlg{A}$). From $t\eta \CAlg{A}$ it follows  that $a_{t^*}\in\Multiplier(\CAlg{A})$. Therefore, it is natural to ask whether or not $a_t\in\Multiplier(\CAlg{A})$ and $b_t\in\Multiplier(\CAlg{A})$ already  imply that  $t\mu \CAlg{A}$, that is, $a_{t^*}\in\Multiplier(\CAlg{A})$. This is  true if $t\in \Abg(\Hil)$ is normal, since then $a_t=a_{t^*}$. Proposition \ref{Association_Relation_Adjoint} below contains an number of other sufficient conditions. In Example \ref{atstar_not_adjointable} we will show that this is not true in general. The following simple relations appeared already in Definition \ref{defabtransform}.

\begin{lem}\label{Algebraic_Relations_AT_BT}
Let $t\in\Abg(\Hil)$. Then:
\begin{enumerate}
%\item $A_T$ is self-adjoint.
\item $a_{t^*}-a_{t^*}^2=b_tb_{t^*}$ and $a_{t^*}b_t=b_ta_t$.
\item $b_t^*=b_{t^*}$.
\item $a_{t^*}^n-a_{t^*}^{n+1}=b_ta_t^{n-1}b_{t^*}$ for $n\in \N$.
\end{enumerate}
\end{lem}
\begin{proof}
(1): We have $a_{t^*}-a_{t^*}^2=tt^*a_{t^*}^2=\ol{ta_t^2t^*}=b_tb_{t^*}$. The second equality follows by a similar reasoning starting with the operator $a_tt^*a_t$. 

(2): Let $x,y\in\Hil$. Then $a_{t^*}y\in\Def(tt^*)\subseteq\Def(t^*)$ and using (1) we obtain $$\SP{b_tx}{a_{t^*}y}=\SP{a_tx}{t^*a_{t^*}y}=\SP{x}{a_tb_{t^*}y}=\SP{x}{b_{t^*}a_{t^*}y}.$$ Therefore $b_t^*=b_{t^*}$, since $\Range(a_{t^*})$ is dense and $b_t$ and $b_{t^*}$ are bounded. 

(3) is easily derived from (1). 
\end{proof}

\begin{prop}\label{Association_Relation_Adjoint}
Suppose that $a_t,b_t\in\Multiplier(\CAlg{A})$. Each of the following conditions imply that $a_{t^*}\in\Multiplier(\CAlg{A})$ and so $t\mu \CAlg{A}$.
\begin{enumerate}
\item $0\in\rho(t)$.
\item $\|a_{t^*}\|<1$, or equivalently, $tt^*\geq\varepsilon$ for some $\epsilon>0$.
\item $\Multiplier(\CAlg{A})_{sa}$ is closed under strong convergence of monotone sequences. 
\item $tt^*=qt^*t$ for some $q>0$.
\end{enumerate}
\end{prop}
\begin{proof}
Clearly, from (1) it follows that  $0\in \rho(t^*)$ which in turn implies (2). 

(2), (3): By Lemma \ref{Algebraic_Relations_AT_BT} and the assumptions $a_t,b_t\in\Multiplier(\CAlg{A})$ we have
$$a_{t^*}-a_{t^*}^{n+1}=b_t(I+\ldots+a_t^{n-1})b_t^*\in \Multiplier(\CAlg{A})_{sa}.$$
If (2) is fulfilled, then $a_{t^*}^{n+1}\to 0$ in $\Multiplier(\CAlg{A})$, hence $a_{t^*}\in\Multiplier(\CAlg{A})$. On the other side, $a_{t^*}^{n+1}\in\Multiplier(\CAlg{A})_{sa}$ is  monotone decreasing and strongly  converging. Hence again by assumption (3)  it follows $a_{t^*}\in\Multiplier(\CAlg{A})$. (4) finally follows from the relations 
\begin{align*}
a_{t^*}=(I+tt^*)^{-1}=(I+qt^*t)^{-1}=q^{-1}(I+(q^{-1}-1)a_t)^{-1}a_t\in\Multiplier(\CAlg{A}).
\end{align*}
\end{proof}

\begin{prop}\label{Characterisation_Assoziation_Resolvent}
Suppose that $t\in\Abg(\Hil)$ and $0\in\rho(t)$. Then $t\mu\CAlg{A}$ if and only if $t^{-1}\in\Multiplier(\CAlg{A})$.
\end{prop}
\begin{proof}
Since $0\in \rho(t)$,  $(t^*)^{-1}=(t^{-1})^* \in\Be(\Hil)$. Simple computations show that
\begin{align*}
I-a_t &= (I+t^{-1}(t^{-1})^*)^{-1}, \quad b_t = (t^{-1})^*(I-a_t),\quad
I-a_{t^*} &= (I+(t^{-1})^*t^{-1})^{-1}.
\end{align*}
From these identities we conclude  that $t^{-1}\in\Multiplier(\CAlg{A})$, so $(t^{-1})^* \in\Multiplier(\CAlg{A})$, implies that $a_t,b_t,a_{t^*}\in\Multiplier(\CAlg{A})$, that is, $t\mu \CAlg{A}$. 

Conversely, suppose that $t\mu\CAlg{A}$. Then, by Lemma \ref{Algebraic_Relations_AT_BT},(1), we have $b_{t^*}=(b_t)^*\in \Multiplier(\CAlg{A})$ and $a_{t^*}\in \Multiplier(\CAlg{A})$. Therefore, $t^{-1}=b_{t^*}(I-a_{t^*})^{-1}\in\Multiplier(\CAlg{A})$.
\end{proof}

\begin{cor}
If $t\in\Abg(\Hil)$ and\, $t\mu\CAlg{A}$, then\, $(I+t^*t)\mu \CAlg{A}$.
\end{cor}
\begin{proof}
Since $t\mu\CAlg{A}$, we have $(I+t^*t)^{-1}=a_t\in\Multiplier(\CAlg{A})$. Since $0\in\rho(1+t^*t)$ and $(1+t^*t)^{-1}=a_t\in\Multiplier(\CAlg{A})$, we obtain $(1+t^*t)\mu\CAlg{A}$ by Proposition \ref{Characterisation_Assoziation_Resolvent}.
\end{proof}

\begin{cor}
Suppose that $t\mu\CAlg{A}$ and $s\mu\CAlg{A}$.
\begin{enumerate}
\item If $0\in\rho(t)$ and $\lambda\in\rho(t)$ with $0<|\lambda|<1/\|t^{-1}\|$, then $(t-\lambda)\mu\CAlg{A}$.
\item If $0\in\rho(t)\cap\rho(s)$, then $ts\mu\CAlg{A}$.
\end{enumerate}
\end{cor}
\begin{proof}
Both assertions follow immediately from Proposition \ref{Characterisation_Assoziation_Resolvent}. For
(1) we use the equality  $(t-\lambda I)^{-1}=\lambda^{-1}t^{-1} (\lambda^{-1}-t^{-1})^{-1}\in \Multiplier(\CAlg{A})$, while  for (2) we note that $0\in\rho(ts)$ and $(ts)^{-1}=s^{-1}t^{-1}\in \Multiplier(\CAlg{A})$.
\end{proof}

Next we investigate affiliated operators and their resolvents. Before we turn to the main result we prove two simple lemmas.

\begin{lem}\label{auxlem1}
Let $\CAlg{A}$ be a $C^*$-algebra acting on a Hilbert space $\Hil$. Let $s\in\Be(\Hil)$ and $x,y\in\Multiplier(\CAlg{A})$. Suppose that $x\CAlg{A}$ and $y \CAlg{A}$ are dense in $\CAlg{A}$. If $sx\in\Multiplier(\CAlg{A})$ and $s^*y\in\Multiplier(\CAlg{A})$, then $s\in\Multiplier(\CAlg{A})$.
\end{lem}
\begin{proof}
Let $a\in \CAlg{A}$.  Since  $x \CAlg{A}$ is dense  in $\CAlg{A}$, there are elements $a_n\in \CAlg{A}$, $n\in N$,  such that $xa_n\to a$ in  $\CAlg{A}$.  Hence $sxa_n\to sa$ in $\CAlg{A}$. Since $sx\in \Multiplier(\CAlg{A})$ by assumption, $sxa_n\in\CAlg{A}$ and so $sa\in \CAlg{A}$. Replacing $x$ by $y$ and $s$ by $s^*$ it follows that $s^*a\in \CAlg{A}$. Therefore, $s\in\Multiplier(\CAlg{A})$.
\end{proof}

\begin{lem}\label{auxlem2}
Let $\CAlg{A}$ be a $C^*$-algebra and $x,y\in\Multiplier(\CAlg{A})$. Suppose that $\lambda y\geq xx^*$ for some $\lambda >0$.
If \, $x\CAlg{A}$ is dense in $\CAlg{A}$, so is $y\CAlg{A}$. In particular, $x\CAlg{A}$ is dense in $\CAlg{A}$ if and only if $xx^*\CAlg{A}$ is.
\end{lem}
\begin{proof}
Assume to the contrary that\, $\overline{y\CAlg{A}}\neq \overline{x\CAlg{A}}=\CAlg{A}.$ Then the  closure of $(y\CAlg{A})^*$ is a proper left ideal. Hence there exists a state $\omega$ of $\CAlg{A}$   that annihilates $(y\CAlg{A})^*$ (see e.g. \cite[Lemma 2.9.4]{dixmier}). Let $\pi_\omega$ be the GNS representation of $\CAlg{A}$ associated with the state $\omega$ and let $\varphi_\omega$ be the corresponding cyclic vector $\varphi_\omega$. We denote the extension of $\pi_\omega$ to the multiplier algebra $\Multiplier(\CAlg{A})$ also by the symbol $\pi_\omega$. Then 
\begin{align}\label{yvarphi}
0=\omega((ya)^*)=\langle \pi_\omega (a^*y)\varphi_\omega,\varphi_\omega\rangle=
 \langle \pi_\omega (y)\varphi_\omega,\pi_\omega(a)\varphi_\omega\rangle
\end{align}
for all $a\in\CAlg{A}$, so that   $\pi_\omega (y)\varphi_\omega=0$. Therefore,
\begin{align*}
|\omega(xa)|^2&=|\langle \pi_\omega (a)\varphi_\omega, \pi_\omega (x^*)\varphi_\omega\rangle|^2=
 \| \pi_\omega (a)\varphi_\omega\|^2 \|\pi_\omega(x^*)\varphi_\omega\|^2\\&=
 \| \pi_\omega (a)\varphi_\omega\|^2 \langle \pi_\omega(xx^*)\varphi_\omega,\varphi_\omega\rangle \leq  \| \pi_\omega (a)\varphi_\omega\|^2\lambda \langle \pi_\omega( y)\varphi_\omega,\varphi_\omega\rangle =0\
\end{align*}
for $a\in \CAlg{A}$, that is, $\omega$ annihilates $x\CAlg{A}$. Hence $x\CAlg{A}$ is not dense  in $\CAlg{A}$ which is a contradiction, since we assumed that $\overline{x\CAlg{A}}=\CAlg{A}.$ 

Applying this to the case $y=xx^*$ we conclude that $xx^*\CAlg{A}$ is dense provided that $x\CAlg{A} $ is dense. Since the converse implication is trivial, it follows that $xx^*\CAlg{A}$ is dense if and only if $x\CAlg{A} $ is dense.
\end{proof} 

The following theorem appeared in \cite{schm2005}.
\begin{thm}\label{Characterisation__affiliation_resolvent}
Suppose  $\CAlg{A}$ is a $C^*$-algebra acting on a Hilbert space $\Hil$ and $t$ is a densely defined closed operator on $\Hil$ with non-empty resolvent set. Let $\lambda\in\rho(t)$. Then  $t\eta\CAlg{A}$ if and only if $(t-\lambda I)^{-1}\in\Multiplier(\CAlg{A})$ and $(t-\lambda I)^{-1}\CAlg{A}$ and $(t^*-\overline{\lambda} I)^{-1}\CAlg{A}$ are dense in $\CAlg{A}$.
\end{thm}
\begin{proof}
Since $t\eta\CAlg{A}$ is equivalent to $(t-\lambda I)\eta\CAlg{A}$ (see \cite{wor91}, p. 412, Example 1), we can assume without restriction of generality  that $\lambda=0$. Then $t^{-1}$ and $(t^*)^{-1}$ are in $\Be(\Hil)$.

First we suppose that $t\eta\CAlg{A}$. Set $x:=(I+(tt^*)^{-1})^{-1}$ and $s:=t^{-1}$. Since $t\eta\CAlg{A}$ implies $t^*\eta\CAlg{A}$, it follows that $z_{t^*}=t^*(I+tt^*)^{-1/2} =(z_t)^*\in\Multiplier(\CAlg{A})$. Therefore, we obtain $(I+tt^*)^{-1}=I-z_tz_t^*\in\Multiplier(\CAlg{A})$ and hence $(I+tt^*)^{-1/2}\in\Multiplier(\CAlg{A})$. 
These relations imply that
\begin{align}\label{sx} 
sx&\equiv t^{-1}(I+(tt^*)^{-1})^{-1}=t^*(tt^*)^{-1}(I+(tt^*)^{-1})^{-1}\nonumber \\&=t^*(I+tt^*)^{-1}=z_{t^*}(I+tt^*)^{-1/2}\in\Multiplier(\CAlg{A}).
\end{align} 
Since $x:=(I+(tt^*)^{-1})^{-1}=I-(I+tt^*)^{-1}\in\Multiplier(\CAlg{A})$ and $x^{-1}$ is also bounded, we have $x^{-1}\in\Multiplier(\CAlg{A})$ and hence $x\CAlg{A}=\CAlg{A}$. Recall that $sx\in \Multiplier(\CAlg{A})$ by (\ref{sx}). Now we interchange the roles of $t$ and $t^*$ and set $y:=(I+(t^*t)^{-1})^{-1}$. By a similar reasoning as in (\ref{sx}) we derive $s^*y\in\Multiplier(\CAlg{A})$. Further, $y\in\Multiplier(\CAlg{A})$ and $y\CAlg{A}=\CAlg{A}$. Hence the assumptions of Lemma \ref{auxlem1} are satisfied, so we obtain $t^{-1}=s\in\Multiplier(\CAlg{A})$. 

Recall that $(I+t^*t)^{-1}\CAlg{A}$ is dense in $\CAlg{A}$, because $t\eta\CAlg{A}$. Therefore, since
$$
(I+t^*t)^{-1}\CAlg{A}=(t^*t)^{-1}(I+(t^*t)^{-1})^{-1}\CAlg{A}\subseteq (t^*t)^{-1}\CAlg{A}= t^{-1}(t^*)^{-1}\CAlg{A}\subseteq t^{-1}\CAlg{A},
$$
$t^{-1}\CAlg{A}$ is dense in $\CAlg{A}$. Replacing $t$ by $t^*$, it follows that $(t^*)^{-1}\CAlg{A}$ is dense in $\CAlg{A}$. This completes the proof of the only if part.

Conversely, let us assume that  $t^{-1}\in\Multiplier(\CAlg{A})$ and that $t^{-1}\CAlg{A}$ and $(t^*)^{-1}\CAlg{A}$ are dense in $\CAlg{A}$. Then $I-z_t^*z_t=(I+t^*t)^{-1}=t^{-1}(t^{-1})^*(I+t^{-1}(t^{-1})^*)^{-1}\in\Multiplier(\CAlg{A})$ and $z_t(I-z_t^*z_t)^{1/2}=t(I+t^*t)^{-1}=(t^{-1})^*(I+t^{-1}(t^{-1})^*)^{-1}\in\Multiplier(\CAlg{A})$. Therefore, setting $x:=(I-z_t^*z_t)^{1/2}$ and $s:=z_t$, we have $x\in\Multiplier(\CAlg{A})$ and $sx\in\Multiplier(\CAlg{A})$. Since $t$ has a bounded inverse, there exists $\epsilon\in(0,1/4)$ such that $t^*t\geq 2\epsilon I$. Then  $I+t^*t\leq\frac{1}{2\epsilon}t^*t+t^*t\leq\frac{1}{\epsilon}t^*t$ and hence $(I+t^*t)^{-1}\geq\epsilon t^{-1}(t^{-1})^*$. Therefore, since $t^{-1}\CAlg{A}$ is dense in $\CAlg{A}$ by assumption, $(I+t^*t)^{-1}\CAlg{A}=(I-z_t^*z_t)\CAlg{A}=x^2\CAlg{A}$ is dense in $\CAlg{A}$ by Lemma \ref{auxlem2}. Since $x\geq 0$, $x\CAlg{A}$ dense in $\CAlg{A}$ again by Lemma \ref{auxlem2}. By the assumptions we can interchange the roles of $t$ and $t^*$. Then we obtain $y:=(I-z_tz_t^*)^{1/2}\in\Multiplier(\CAlg{A})$ and $s^*y=z_t^*y\in\Multiplier(\CAlg{A})$. Further, $(I+tt^*)^{-1}\CAlg{A}=(I-z_tz_t^*)\CAlg{A}=y^2\CAlg{A}$ in $\CAlg{A}$ and hence $y\CAlg{A}$ are dense in $\CAlg{A}$. Thus, $z_t\in\Multiplier(\CAlg{A})$ by Lemma \ref{auxlem1} and hence $t\eta\CAlg{A}$.
\end{proof}

The preceding theorem has a  number of interesting corollaries. For these results we assume that $\CAlg{A}$ is a $C^*$-algebra acting on a Hilbert space $\Hil$ and $t$ and $s$ are densely defined closed operators on $\Hil$.

\begin{cor}\label{Affiliation_of_product__general_resolvent}
Suppose that $t,s\eta\CAlg{A}$, $\lambda\in\rho(t)$ and $\mu\in\rho(s)$. Then we have $-\lambda\mu\in\rho(ts-\lambda s-\mu t)$ and $(ts-\lambda s-\mu t)\eta\CAlg{A}$.
\end{cor}
\begin{proof}
By some straightforward arguments one verifies that
\begin{align}\label{ts1} 
&(ts-\lambda s-\mu t+\lambda\mu I)^{-1}=(s-\mu I)^{-1}(t-\lambda I)^{-1},\\&((ts-\lambda s-\mu t)^*+\overline{\lambda\mu} I)^{-1}=(t^*-\overline{\lambda} I)^{-1}(s^*-\overline{\lambda}I)^{-1}.\label{ts2} 
\end{align} 
Hence $-\lambda\mu\in\rho(ts-\lambda s-\mu t)$. From the only if part of Theorem \ref{Characterisation__affiliation_resolvent} it follows that the operators in (\ref{ts1}) and in (\ref{ts2}) belong to  $\Multiplier(\CAlg{A})$ and that they maps $\CAlg{A}$ densely into $\CAlg{A}$. Therefore, by the if part of Theorem \ref{Characterisation__affiliation_resolvent}, $(ts-\lambda s-\mu t)\eta\CAlg{A}$.
\end{proof}

\begin{prop}\label{Affiliation_of_sum__analyticity}
Suppose that $\lambda\in\rho(t)$, $s(t-\lambda I)^{-1}\in\Multiplier(\CAlg{A})$ and $\|s(t-\lambda I)^{-1}\|<1$. Then $(t+s)\eta\CAlg{A}$.
\end{prop}
\begin{proof}
By Theorem \ref{Characterisation__affiliation_resolvent}, $(t-\lambda I)^{-1}\in\Multiplier(\CAlg{A})$ and $(t-\lambda I)^{-1}\CAlg{A}$ and $(t^*-\ol{\lambda}I)^{-1}\CAlg{A}$ are dense in $\CAlg{A}$. By the  assumption we have $r:=s(t-\lambda I)^{-1}\in\Multiplier(\CAlg{A})$ and $\|r\|<1$. Therefore $(I+r)^{-1}$ is bounded and an element of\, $\Multiplier(\CAlg{A})$, since $I+r\in\Multiplier(\CAlg{A})$. Further, since $t-\lambda I$ and $I+r$ are bijective and $(t+s-\lambda I)\varphi=(I+r)(t-\lambda I)\varphi$\, for $\varphi\in\Def(t)\subseteq\Def(s)$, the map $t+s-\lambda I:\Def(t)\to\Hil$ is bijective. Hence $\lambda\in\rho(t+s)$ and $(t+s-\lambda I)^{-1}=(t-\lambda I)^{-1}(I+r)^{-1}\in\Multiplier(\CAlg{A})$. Because $I+r$ is an invertible element of $\Multiplier(\CAlg{A})$, the density of $(t-\lambda I)^{-1}\CAlg{A}$ implies the density of $(t+s-\lambda I)^{-1}\CAlg{A}$ in $\CAlg{A}$. Finally, $(t^*+s^*-\ol{\lambda}I)^{-1}=(I+r^*)^{-1}(t^*-\ol{\lambda} I)^{-1}$  maps $\CAlg{A}$ densely into $\CAlg{A}$, since $\|r^*\|<1$. Now applying again Theorem \ref{Characterisation__affiliation_resolvent} we obtain $(t+s)\eta\CAlg{A}$.
\end{proof}

\begin{cor}
Let $\CAlg{A}\subseteq\Be(\Hil)$ a $C^*$-algebra and suppose that\,  $t,s\eta\CAlg{A}$. If\,  $\lambda\in\rho(t)$, $0\in\rho(s)$, and $\|\lambda(t-\lambda I)^{-1}\|<1$,  then\, $ts\eta\CAlg{A}$.
\end{cor}
\begin{proof}
By Corollary \ref{Affiliation_of_product__general_resolvent} we have $0\in\rho(ts-\lambda s)$ and $(ts-\lambda s)\eta\CAlg{A}$. Since $t\eta\CAlg{A}$ and $\lambda\in\rho(t)$, it follows from Theorem \ref{Characterisation__affiliation_resolvent} that $(t-\lambda)^{-1}\in\Multiplier(\CAlg{A})$. Therefore,
$$\lambda s(ts-\lambda s)^{-1}=\lambda s((t-\lambda I)s)^{-1}=\lambda(t-\lambda I)^{-1}\in\Multiplier(\CAlg{A}).$$ 
Hence, since $\|\lambda(t-\lambda)^{-1}\|<1$ by assumption, Proposition \ref{Affiliation_of_sum__analyticity}  applies to the operators $\tilde{t}:=ts-\lambda s$ and $\tilde{s}:=\lambda s$ and implies that\, $\tilde{t}+\tilde{s} =ts \eta\CAlg{A}$.
\end{proof}

\begin{cor}\label{Resolvent_equation__more_general}
Let $\CAlg{A}\subseteq\Be(\Hil)$ a $C^*$-algebra. Suppose that $t\eta\CAlg{A}$ and $\lambda,\mu\in\rho(t)$. For an operator $s$ on $\Hil$ we have $s(t-\lambda I)^{-1}\in\Multiplier(\CAlg{A})$ if and only if $s(t-\mu I)^{-1}\in\Multiplier(\CAlg{A})$.
\end{cor}
\begin{proof}
Since $(t-\lambda I)^{-1}\in\Multiplier(\CAlg{A})$ and $(t-\mu I)^{-1}\in\Multiplier(\CAlg{A})$ by Theorem \ref{Characterisation__affiliation_resolvent}, the assertion follows from the identity 
$$s(t-\lambda I)^{-1}-s(t-\mu I)^{-1}=(\lambda-\mu )s(t-\mu I)^{-1}(t-\lambda I)^{-1}=(\lambda-\mu)s(t-\lambda I)^{-1}(t-\mu I)^{-1}.$$
\end{proof}

\begin{cor}
Let $\CAlg{A}\subseteq\Be(\Hil)$ a $C^*$-algebra. Suppose that  $t\eta\CAlg{A}$ is a self-adjoint operator and $s$ is a symmetric $t$-bounded operator  on $\Hil$ with $t$-bound less than $1$. If\, $s(t-\lambda I)^{-1}\in\Multiplier(\CAlg{A})$ for some $\lambda\in\rho(t)$, then $(t+s)\eta\CAlg{A}$  and $t+s$ is self-adjoint.
\end{cor}
\begin{proof}
The proof can be given by repeating the standard proof of the Kato-Rellich theorem and using Lemma \ref{Affiliation_of_sum__analyticity} and Corollary \ref{Resolvent_equation__more_general}.
\end{proof}

It is natural to ask whether or not the  second density assumption in Theorem \ref{Characterisation__affiliation_resolvent} can be omitted, that is, when does the density of $(t-\lambda I)^{-1}\CAlg{A}$ for $(t-\lambda I)^{-1}\in\Multiplier(\CAlg{A})$ imply the density of $(t^*-\overline{\lambda} I)^{-1}\CAlg{A}$ in $\CAlg{A}$?  A counterexample is provided by Example \ref{counterdensity} below. 
The next proposition shows that the answer is affirmative if the distance of $(t-\lambda I)^{-1}$ to the set $\CAlg{A}^{-1}$ of invertible elements of $\CAlg{A}$ is zero.

\begin{prop} Let $t$ be a densely defined closed operator and  $\CAlg{A}$ a  $C^*$-algebra  acting on a Hilbert space $\Hil$. Suppose that $t$ has a bounded inverse $t^{-1}$ contained in $\Multiplier(\CAlg{A})$. Assume that\,  ${\rm dist} (t^{-1}, \CAlg{A}^{-1})=0$. If\, $t^{-1}\CAlg{A}$ dense in $\CAlg{A}$, so is  $(t^*)^{-1}\CAlg{A}$.
\end{prop}    
\begin{proof} Set $x:=t^{-1}$. Then $(t^*)^{-1}=(t^{-1})^*=x^*$. Assume to the contrary that $x^*\CAlg{A}=(t^*)^{-1}\CAlg{A}$ is not dense in $\CAlg{A}$. Then, as in the  proof of Lemma \ref{auxlem2},  there is a state $\omega$ on $\CAlg{A}$ that annihilates $(x^*\CAlg{A})^*$. Arguing as in line (\ref{yvarphi}) it follows that $\pi_\omega(x)\varphi_\omega=0$. 

Let $x=v|x|$ be the polar decomposition of the operator $x$. Since $x\in\Multiplier(\CAlg{A})$, we have $|x|=(x^*x)^{1/2}\in\Multiplier(\CAlg{A})$. For $\varepsilon>0$ let $f_\varepsilon$ denote a continuous function on $\R$ such that $f_\varepsilon (\tau) =0$ on $[0,\frac{\varepsilon}{2}]$, $|f_\varepsilon(\tau) |\leq \varepsilon$ on $[\frac{\varepsilon}{2},\varepsilon]$ and $f_\varepsilon (\tau) =u$ on $[\varepsilon,+\infty)$. 
By \cite{pedersen} Theorem 6.1, applied to the multiplier algebra $\Multiplier(\CAlg{A})$, there exists a unitary operator $u_\varepsilon \in\Multiplier(\CAlg{A})$ such that
$$
vf_\varepsilon(|x|)= u_\varepsilon f_\varepsilon(|x|)\in\Multiplier(\CAlg{A}).
$$
Clearly, $|x|=\lim_{\varepsilon\to +0} f_\varepsilon(|x|)$ in $\Multiplier(\CAlg{A})$. Therefore, $$0=\pi_\omega(x)\varphi_\omega =\pi_\omega(v|x|)\varphi_\omega=\lim_{\varepsilon\to +0} \pi_\omega(vf_\varepsilon(|x|))\varphi_\omega=\lim_{\varepsilon\to +0} \pi_\omega(u_\varepsilon)\pi_\omega(f_\varepsilon(|x|))\varphi_\omega,$$
so that $0=\lim_{\varepsilon\to +0} \pi_\omega(f_\varepsilon(|x|))\varphi_\omega= \pi_\omega(|x|)\varphi_\omega$. For $a\in \CAlg{A}$ we now obtain
\begin{align*}
0 &=\langle  \pi_\omega (|x|)^2\varphi_\omega, \pi_\omega (a)\varphi_\omega\rangle=
 \langle \pi_\omega (x^*x)\varphi_\omega\,\pi_\omega(a)\varphi_\omega\rangle \\
&=\langle \varphi_\omega, \pi_\omega(x^*xa)\varphi_\omega\rangle=\omega((x^*xa)^*).
\end{align*}
This implies that $x^*x\CAlg{A}$ is not dense in $\CAlg{A}$. Hence $x\CAlg{A}$ is not dense by Lemma \ref{auxlem2} which is the desired contradiction.
\end{proof}

\begin{exa}\label{counterdensity}
Let $\Hil$ be the Hilbert space $l_2(\N^2)$ and let $\CAlg{A}$ be the $C^*$-algebra
\begin{align}\label{algexample}
 \CAlg{A} =\begin{pmatrix} \K(\Hil)& \K(\Hil)\\ \K(\Hil) &\Be(\Hil) \end{pmatrix}.
\end{align}
The multiplier algebra of $\CAlg{A}$ is 
\begin{align}\label{maexample}
 \Multiplier(\CAlg{A}) =\begin{pmatrix} \Be(\Hil)& \K(\Hil)\\ \K(\Hil) &\Be(\Hil) \end{pmatrix}.
\end{align}
Let $\{e_{kl}\}_{k,l\in \N_0}$ be the standard orthonormal basis of $\Hil$. Let $s\in \Be(\Hil)$ be the shift operator given by $se_{kl}=e_{k+1,l}$ and let $P_0$ be the orthogonal projection  onto $\Null(s^*)$. Clearly, $\{e_{0,l}\}_{l\in \N_0}$ is an orthonormal basis of $P_0\Hil$. Further, let $\{\lambda_{kl}\}_{k,l\in \N_0}$ be a double sequence of positive numbers such that $\lim_{k,l\to \infty} \lambda_{kl}=0$. Define a self-adjoint compact operator on $ \Hil$  by $re_{kl}:=\lambda_{kl}e_{kl}$, $k,l\in \N_0$.

Let $x\in\Be(\Hil \oplus \Hil)$ defined by the operator matrix
\begin{align}\label{tmatrix}
  x: =\begin{pmatrix} s& r\\ 0& s^* \end{pmatrix}.
\end{align}
Since $\lambda_{kl}>0$ for all $k,l$, the compression  $P_0r\upharpoonright P_0\Hil$ of $r$ to $P_0\Hil$ has trivial kernel and  dense range. Using this fact it is easily seen that $\Null(x)=\{0\}$ and $\Range(x)$ is dense in $\Hil\oplus\Hil$. Hence $t:=x^{-1}$ is a densely defined closed operator on the Hilbert space $\Hil\oplus \Hil$. By (\ref{maexample}), we have $t^{-1}=x\in\Multiplier(\CAlg{A})$. 
\smallskip\\
{\bf Statement:}  $t^{-1}\CAlg{A}=x\CAlg{A}$ is dense in $\CAlg{A}$, while $(t^*)^{-1}\CAlg{A}=x^*\CAlg{A}$ is not dense in $\CAlg{A}$.
\begin{proof}
Let $y$ be an element of $\CAlg{A}$. Then $y$ given by an operator matrix  
\begin{align*}
  y: =\begin{pmatrix} a& b\\ c& d \end{pmatrix},
\end{align*}  
$a,b,c\in \K(\Hil)$ and $d\in \Be(\Hil)$, and  have
\begin{align}\label{xyproduct}
  xy =\begin{pmatrix} s a+rc& s b+rd\\ s^* c& s^* d\end{pmatrix}.
\end{align}
Since $\K(\Hil)=s^*s\K(\Hil)\subseteq s^*\K(\Hil)\subseteq \K(\Hil)$, we have  $s^*\K(\Hil)=\K(\Hil)$. Similarly, $s^*\Be(\Hil)=\Be(\Hil)$. Since the range of $r$ contains all rank one operators $e_{kl}\otimes e_{nm}$, $s\K(\Hil)+r\K(\Hil)$ is dense in $\K(\Hil)$. Therefore, by (\ref{algexample}) and (\ref{xyproduct}),  $x\CAlg{A}$ is dense in $\CAlg{A}$.

Next we prove the second  assertion. First we note that $$P_0(r\K(\Hil)+s\Be(\Hil))=P_0(r\K(\Hil))\subseteq P_0\K(\Hil).$$ 
This implies that $r\K(\Hil)+s\Be(\Hil)$ is not dense in $\Be(\Hil)$. 
Therefore, since
\begin{align*}
  x^*y =\begin{pmatrix} s^*a& s^* b\\ ra +s c& rb+s d\end{pmatrix},
\end{align*}
it follows from (\ref{algexample}) that the set $x^*\CAlg{A}$ is not dense in $\CAlg{A}$.
\end{proof}
\end{exa}

  \section{Examples}\label{examples}

  \subsection{Matrices of commutative $C^*$-algebras and its multipliers}\label{matrices}

In this subsection  we use matrices over commutative $C^*$-algebras to construct  simple examples of operators that help to delimit the general theory.

Let $\CAlg{A}$ be a $C^*$-algebra. If $A_{ij}\subseteq\CAlg{A}$ for $i,j\in\{1,2\}$ set
\begin{align*}
\left(\begin{array}{cc}
A_{11} & A_{22} \\ 
A_{21} & A_{22}
\end{array} \right) := \left\{ \left(\begin{array}{cc}
a_{11} & a_{12} \\ 
a_{21} & a_{22}
\end{array}\right) | a_{ij}\in A_{ij} \text{ for } i,j\in\{1,2\} \right\}.
\end{align*}

Let $X$ be a locally compact non-compact Hausdorff space and set
\begin{align*}
\CAlg{A}_0 &:= \left(\begin{array}{cc}
C_0(X) & C_0(X) \\ 
C_0(X) & C_0(X)
\end{array} \right), \quad \CAlg{A} := \left(\begin{array}{cc}
C_0(X) & C_0(X) \\ 
C_0(X) & C_0(X)^\sim
\end{array} \right),
\end{align*}
where $C_0(X)^\sim:=C_0(X) +\C\cdot 1$. 
A straightforward computation shows that the left multiplier algebra $\LeftMultiplier(\CAlg{A})$ and the multiplier algebra $\Multiplier(\CAlg{A})$ are given by
\begin{align*}
\LeftMultiplier(\CAlg{A}) &= \left(\begin{array}{cc}
C_b(X) & C_0(X) \\ 
C_b(X) & C_0(X)^\sim
\end{array} \right), \quad \Multiplier(\CAlg{A}) = \left(\begin{array}{cc}
C_b(X) & C_0(X) \\ 
C_0(X) & C_0(X)^\sim
\end{array} \right).
\end{align*}
From this we can read off that all elements of the form
\begin{align*}
\left(\begin{array}{cc}
* & * \\ 
f & *
\end{array} \right)\in\LeftMultiplier(\CAlg{A}) \quad \text{with} \quad f\in C_b(X)\setminus C_0(X)
\end{align*}
act as operators $t$ on $\CAlg{A}$ defined on the whole space such that the adjoints are  not defined on the whole space. From now on let $X=\R$.

\begin{exa}
Let $f=1$ and set the other matrix entries zero. Then  $t$ acts as
\begin{align*}
\Def(t) &= \CAlg{A}, \quad t = \left(\begin{array}{cc}
0 & 0 \\ 
1 & 0
\end{array} \right), \quad \rm{and} \quad \Def(t^*) = \CAlg{A}_0, \quad t^* = \left(\begin{array}{cc}
0 & 1 \\ 
0 & 0
\end{array} \right).
\end{align*}
Hence $t^*$ is essentially defined. Further,  it is easily checked that $t^{**}=t$ and
\begin{align*}
\Def(1+t^*t) &= \CAlg{A},& 1+t^*t = \left(\begin{array}{cc}
2 & 0 \\ 
0 & 1
\end{array} \right), \quad \Def(1+tt^*) &= \CAlg{A}_0,& 1+tt^* = \left(\begin{array}{cc}
1 & 0 \\ 
0 & 2
\end{array} \right).
\end{align*}
From these formuals we read off that\, $\Range(1+t^*t)=\CAlg{A}$ and\, $\Range(1+tt^*)=\CAlg{A}_0\subsetneq\CAlg{A}$. Hence $a_t$ is adjointable, while $a_{t^*}$ is not. In fact, $b_t$ is also not adjointable, since
\begin{align*}
\Def(b_t) = \CAlg{A}, \quad b_t = \left(\begin{array}{cc}
0 & 0 \\ 
1/2 & 0
\end{array} \right) \notin \Multiplier(\CAlg{A}).
\end{align*}
\end{exa}
In the following slighty more sophisticated example   $a_t$ and $b_t$ are both adjointable, while  $a_{t^*}$  is not adjointable.

\begin{exa}\label{atstar_not_adjointable}
Let $f,g\in C(\R)$ be functions given by $f(x):=x\sqrt{1+\sin^2(x)}$ and $g(x):=x\sqrt{1+\cos^2(x)}$. Then $|f(x)|^2+|g(x)|^2=3x^2$. Define $t:\CAlg{A}\to\CAlg{A}$ by
\begin{align*}
\Def(t) &:= \bigg\{\left(\begin{array}{cc}
a & b \\ 
c & d
\end{array} \right)\in\CAlg{A} |\, fc,fd,gc\in C_0(\R),gd\in C_0(\R)^\sim\bigg\}, \quad t = \left(\begin{array}{cc}
0 & f \\ 
0 & g
\end{array} \right).
\end{align*}
For the adjoint $t^*$ we obtain
\begin{align*}
\Def(t^*) &:= \bigg\{\left(\begin{array}{cc}
a & b \\ 
c & d
\end{array} \right)\in\CAlg{A} | \, \ol{f}a+\ol{g}c\in C_0(\R),\ol{f}b+\ol{g}d\in C_0(\R)^\sim\bigg\}, \quad t = \left(\begin{array}{cc}
0 & 0 \\ 
\ol{f} & \ol{g}
\end{array} \right).
\end{align*}
It is now easily verified that $1+t^*t$ is surjective and
\begin{align*}
a_t &= \left(\begin{array}{cc}
1 & 0 \\ 
0 & \frac{1}{1+|f|^2+|g|^2}
\end{array} \right)\in\CAlg{A}, \quad b_t = \left(\begin{array}{cc}
0 & \frac{f}{1+|f|^2+|g|^2} \\ 
0 & \frac{g}{1+|f|^2+|g|^2}
\end{array} \right)\in\CAlg{A}.
\end{align*}
The operator $a_{t^*}$ is computed as 
\begin{align*}
\Def(a_{t^*}) &= \CAlg{A}_0, \quad a_{t^*} = \frac{1}{1+|f|^2+|g|^2}\left(\begin{array}{cc}
1+|g|^2 & -\ol{f}g \\ 
-\ol{g}f & 1+|f|^2
\end{array} \right).
\end{align*}
That is, $a_t\in \CAlg{A}$ and $b_t\in\CAlg{A}$ are adjointable, but  $a_{t^*}\notin \Multiplier(\CAlg{A})$  is not adjointable.
\end{exa}

\begin{exa}
Now we consider the operator $t$ given by
\begin{align*}
\Def(t) &= \CAlg{A}, \quad t = \left(\begin{array}{cc}
0 & 0 \\ 
1 & 1
\end{array} \right), \quad \rm{and} \quad \Def(t^*) = \CAlg{A}_0, \quad t^* = \left(\begin{array}{cc}
0 & 1 \\ 
0 & 1
\end{array} \right).
\end{align*}
Then one easily proves that $\Range(t)\nsubseteq\Def(t^*)$ and $\Def(t^*t)=\CAlg{A}_0\subsetneq\CAlg{A}=\Def(t)$. In particular, $\Def(t^*t)$ is not dense in the domain $\Def(t)$!
\end{exa}

  \subsection{A fraction algebra related to the Weyl algebra}\label{fractionWeyl}

Let  $P=-{\ii} \frac{d}{dt}$ and $Q=t$ be the momentum and  position operators acting as self-adjoint operators on the Hilbert space  $L^2(\R)$. Fix $\alpha,\beta \in\R\backslash \{0\}$  and define bounded operators $x$ and $y$  by
$$
x:=(Q-\alpha {\ii} I)^{-1}\quad {\rm  and}\quad y:=(P-\beta {\ii} I)^{-1}.$$ 
It is not difficult to verify that these operators satisfy the commutation relations 
\begin{align}
x-x^*=2\alpha {\ii} x^*x=2\alpha {\ii} xx^*,& ~y-y^*=2\beta {\ii} y^*y=2\beta {\ii} yy^*,\label{xyrel1}\\
xy-yx= -{\ii} xy^2x=-{\ii} yx^2y,& ~xy^*-y^*x= -{\ii} x(y^*)^2x=-{\ii} y^*x^2y^*\label{xyrel2}.
\end{align}

Let $\cX$ be the unital $*$-subalgebra of ${\bf B}(L^2(\R))$ generated by  $x$ and $y$. 
Since the operators $x, x^*, y, y^*$ are bijections of the Schwartz space $\cS(\R)$, so are their inverses.   Hence $\tau_x(\cdot):= x\cdot x^{-1}$ and $\tau_y(\cdot):=y\cdot y^{-1}$ are automorphisms of the algebra $L(\cS(\R))$ of linear operators on the Schwartz space $\cS(\R)$. From the relations (\ref{xyrel1}) and (\ref{xyrel2}) we conclude that $\tau_x$ and $\tau_y$ leave the algebra $\cX$ invariant, so they are algebra automorphisms of $\cX$. Hence
\begin{align}\label{defio}
\cJ_0:=yx \cX=xy\cX=\cX yx =\cX xy
\end{align}
and $\cJ_0$ is a two-sided  $*$-ideal of the $*$-algebra $\cX$. 

Let $\cF_x$ be the unital $*$-subalgebra of $\cX$ generated by $x$, that is, $\cF_x$ is the commutative $*$-algebra of polynomials $f(x,x^*)$ in $x$ and $x^*$ with complex coefficients. Likewise, $\cF_y$ denotes the unital $*$-subalgebra of $\cX$ generated by $y$. Note that $\cF_x \cap\cF_y=\C\cdot 1$. From the relations it follows easily that $\cX$ is the direct sum of vector spaces $\cF_x+\cF_y$ and $\cJ_0$. Hence each element $a\in \cX$ can be written as 
\begin{align}\label{arepfgb1}
a=f_1(x,x^*)+g_1(y,y^*)+yxb_1=f_2(x,x^*)+g_2(y,y^*)+xyb_2,
\end{align} 
where $f_1,g_1, f_2, g_2$ are  polynomials with $g_1(0,0)=f_2(0,0)=0$ and $b_1,b_2\in \cX$. Moreover,  these triples  $\{f_1,g_1,b_1\}$  and $\{f_2,g_2,b_2\}$ are uniquely determined  by  $a$. 

Let $\varepsilon >0$ and $\lambda \in \R$. We denote by  $\chi_\varepsilon$ the characteristic function of the interval $[0,\varepsilon]$.
Put~ $\omega_{\varepsilon,\lambda}(t):=\frac{1}{\sqrt{\varepsilon}}\chi_\varepsilon (t-\lambda).$
\begin{lem} For polynomials $f\in \C[x,x^*]$ and $g\in \C[y,y^*]$, where  $g(0,0)=0$, and $b\in \cX$ we have
\begin{align}
\lim_{\varepsilon\to +0}~ \langle f(x,x^*)~ \omega_{\varepsilon,\lambda}, \omega_{\varepsilon,\lambda}\rangle &=f((\lambda-\alpha {\ii})^{-1},(\lambda+ \alpha {\ii} )^{-1}),\label{xlimit}\\
\lim_{\varepsilon\to +0}~ \langle g(y,y^*)~ \omega_{\varepsilon,\lambda}, \omega_{\varepsilon,\lambda}\rangle&=0,\label{ylimit}\\
 \lim_{\varepsilon\to +0}~  \langle yxb~ \omega_{\varepsilon,\lambda}, \omega_{\varepsilon,\lambda}\rangle&=0.\label{xyblimit}
\end{align}
\end{lem}
\begin{proof}
Suppose that $\varphi\in C^1(\R)$. Then we have
$$
\langle \varphi \omega_{\varepsilon,\lambda}, \omega_{\varepsilon,\lambda}\rangle = \varphi(\lambda)+\int_\lambda^{\lambda+\varepsilon}  \varepsilon^{-1}
(\varphi(t) -\varphi (\lambda))dt \to \varphi(\lambda)$$
which in turn implies (\ref{xlimit}).

Next we prove (\ref{ylimit}). The crucial step for  this is to show that 
\begin{align}\label{ylimit1}
\lim_{\varepsilon\to +0} y \omega_{\varepsilon,\lambda}=\lim_{\varepsilon\to +0} y^* \omega_{\varepsilon,\lambda}=0
\end{align} in $L^2(\R)$. Without loss of generality we can assume that  $\beta <0$. 
Since $\beta <0$, for the resolvent of $P=-{\ii} \frac{d}{dt}$ we obtain
$$
(y\varphi)(t)=((P-\beta {\ii} I)^{-1}\varphi)(t)=-{\ii} \int_t^\infty e^{\beta (s-t)} \varphi(s) ds, \quad \varphi \in L^2(\R).
$$
Hence we compute $(y \omega_{\varepsilon,\lambda})(t)=0$ for $t\geq \lambda+\varepsilon$, 
\begin{align*}
(y \omega_{\varepsilon,\lambda})(t)= \frac{-{\ii}}{\sqrt{\varepsilon}} e^{-\beta t}\int_t^{\lambda+\varepsilon}e^{\beta s} ~ dt =
\frac{-{\ii}}{\beta \sqrt{\varepsilon}}(e^{\beta(\lambda+\varepsilon-t)}-1)
\end{align*}
for $ \lambda\leq t\leq \lambda+\varepsilon$, and
\begin{align*}(y \omega_{\varepsilon,\lambda})(t) &= \frac{-{\ii}}{\sqrt{\varepsilon}} e^{-\beta t}\int_\lambda^{\lambda+\varepsilon}e^{\beta s} ~ dt =\frac{-{\ii}}{\beta \sqrt{\varepsilon}}  e^{-\beta t}(e^{\beta(\lambda +\varepsilon)}-e^{\beta(\lambda -\varepsilon)})
\end{align*}
 for $ t\leq \lambda.$ From these formulas we easily derive that\, $\lim_{\varepsilon\to +0} y \omega_{\varepsilon,\lambda}=0$. Replacing $\beta$ by $-\beta$ a similar reasoning yields $\lim_{\varepsilon\to +0} y^* \omega_{\varepsilon,\lambda}=0$. Since the operator $y$ is bounded and $g(0,0)=0$, 
(\ref{ylimit1}) implies (\ref{ylimit}).
Since $\|\omega_{\varepsilon,\lambda}\|=1$, it follows from (\ref{ylimit1}) that
 $$|\langle yxb~ \omega_{\varepsilon,\lambda}, \omega_{\varepsilon,\lambda}\rangle| =|\langle xb~ \omega_{\varepsilon,\lambda},y^* \omega_{\varepsilon,\lambda}\rangle|\leq \|xb\| ~\|y^*\omega_{\varepsilon,\lambda}\|\to 0$$
which proves (\ref{xyblimit}).
\end{proof}

Next we define two circles $K_\alpha$ and $K_\beta$ intersecting in the origin  by
$$K_\alpha:=\{(z,0)\in  \C^2:z-\overline{z}=2\alpha{\ii} |z|^2\},\quad K_\beta:=\{(0,w)\in\C^2:w-\overline{w}=2\beta{\ii} |w|^2\}.$$

Let $\overline{\R}=\R \cup \{\infty\}$ be the one point compactification of the real line. The maps $\lambda \to (z_\lambda,0):=((\lambda-\alpha {\ii})^{-1},0)$ and $\lambda \to(0,w_\lambda):=(0,(\lambda-\beta {\ii})^{-1})$, where $z_\infty :=0$ and $w_\infty :=0$, are bijection of $\overline{\R}$ onto $K_\alpha$ and $K_\beta$, respectively. 
Let $z \in K_\alpha$ and $w\in K_\beta$. For $a$ as in (\ref{arepfgb1}) we define
$$\pi_{x,z} (a)=f_1(z,\overline{z}),\quad \pi_{y,w} (a)=g_2(w,\overline{w}).$$

\begin{lem} For all $z \in K_\alpha$ and $w\in K_\beta$,
$\pi_{x,z}$ and $ \pi_{y,w}$ are one dimensional $*$-representations of the $*$-algebra $\cX$ such that 
\begin{align}\label{conituitypi}|\pi_{x,z} (a)|\leq \|a\|\quad{\rm and}\quad  |\pi_{y,w}(a)|\leq \|a\|\quad {\rm for}\quad a\in \cX.
\end{align} 
Moreover, for $f\in \cF_x$ and $g\in \cF_y$ we have 
\begin{align}\label{normest}
\|f\|=\sup_{z\in K_\alpha}\, |f(z,\overline{z})|\quad {\rm and}\quad \|g\|=\sup_{w\in K_\beta}\, |g(w,\overline{w})|.
\end{align}
\end{lem}
\begin{proof}
A simple computation based on the relations (\ref{xyrel1}) and (\ref{xyrel2})  shows that $\pi_{x,z}$ and $ \pi_{y,w}$ are well-defined $*$-homomorphisms of $\cX$.  

Let $\lambda \in \R$. 
From the formulas (\ref{xlimit}), (\ref{ylimit}) and (\ref{xyblimit}) we infer that
$$ 
\pi_{x,z_\lambda}(a)=f_1((\lambda-\alpha {\ii})^{-1},(\lambda+ \alpha {\ii} )^{-1})=\lim_{\varepsilon\to +0}~ \langle a \omega_{\varepsilon,\lambda}, \omega_{\varepsilon,\lambda}\rangle. $$
Therefore, since $\|\omega_{\varepsilon,\lambda}\|=1$, we obtain $|\pi_{x,z_\lambda}(a)|\leq \|a\|$. Passing to the limit $|\lambda|\to \infty$, we get
$|\pi_{x,z_0}(a)|\leq \|a\|$. This proves the first inequality of (\ref{conituitypi}) for all  $z\in K_\alpha$.
The inequality $|\pi_{y,w}(a)|\leq \|a\|$ follows in a similar manner by interchanging the role of $x$ and $y$ and using the counterparts of formulas (\ref{xlimit}), (\ref{ylimit}) and (\ref{xyblimit}).

Finally, we prove (\ref{normest}). One inequality of the equality (\ref{normest})  was shown by (\ref{conituitypi}). The reversed inequality follows at once from the fact that the spectra of the self-adjoint operators $P$ and $Q$ are equal to $\R$.
\end{proof}

It was noted in \cite{schm2010} (and is easily verified by using  (\ref{xyrel1})), the $*$-algebra $\cX$ is algebraically bounded, that is, given $a\in\cX$ there exist $\gamma_a>0$ and finitely many elements $a_i\in \cX$ such that $\gamma_a=a^*a+\sum_i a_i^*a_i$. Hence 
$$
\|a\|_{un}:= \sup_\pi \, \|\pi(a)\|,\quad a\in \cX,
$$
defines a $C^*$-norm on $\cX$, where the supremum is taken over all $*$-representation of $\cX$, and the completion $\cX_{un}$ of $(\cX;\|\cdot\|_{un})$ is called the universal $C^*$-algebra of $\cX$. (The supremum is finite, since $\|\pi(a)\|\leq \gamma^{1/2}_a$ for all $a\in \cX$.) By decomposition theory it suffices to take all {\it{irreducible}} $*$-representations $\pi$. As proved in \cite{schm2010}, the irreducible $*$-representation of $\cX$ are the one-dimensional representations $\pi_{x,z}$ and $ \pi_{y,w}$,where $z\in\cX_x$ and $w\in K_\beta$, and the identity representation $\pi_0$ acting on the Hilbert space $L^2(\R)$. From (\ref{conituitypi})  it follows that the universal $C^*$-norm $\|\cdot \|_{un}$ coincides with the operator norm $\|\cdot\|$ on $L^2(\R)$. Hence the universal $C^*$-algebra $\cX_{un}$ is just the closure $\overline{\cX}$ of $\cX$ in ${\bf B}(L^2(\R)).$ By (\ref{normest}) the closures of $\cF_x$ and $\cF_y$ in ${\bf B}(L^2(\R))$ are the commutative $C^*$-subalgebras $C(K_\alpha)$ and $C(K_\beta)$, respectively, and the closure $\cJ$ of $\cJ_0$ in  ${\bf B}(L^2(\R))$  is a two-sided $*$-ideal of the $C^*$-algebra $\cX_{un}=\overline{\cX}.$ 
\begin{lem}\label{essentialio}
$\cJ_0$ is a two-sided essential ideal of $\cX_{un}$.
\end{lem}
\begin{proof}
Let $a\in \cX_{un}$ be such that $a\cJ_0=\{0\}.$ Then $axy=0$ in $L^2(\R)$. Since $x$ and $y$ are bijection, this implies $a=0$.
\end{proof}
The operator $x^*y^*yx$ is an integral operator on $L^2(\R)$ with kernel $$K(t,s):(2|\beta|)^{-1}(t+\alpha {\ii})^{-1}(s-\alpha {\ii})^{-1} e^{-|\beta||t-s|}.$$ Since  $K\in L^2(\R^2)$, the operator $x^*y^*yx=|yx|^2$ is compact,  so are $|yx|$ and hence $yx$. Hence $\cJ_0\subseteq {\bf K}(L^2(\R))$ by (\ref{defio}) and therefore $\cJ= {\bf K}(L^2(\R))$. 

Now we define two operators, denoted  by $Q$ and $P$, on the $C^*$-algebra $\cX_{un}$ by 
$$Q:=\alpha {\ii} I+x^{-1},\, \Def{(Q)}:=x\cX_{un}\quad {\rm and}\quad P:=\beta {\ii} +y^{-1}, \, \Def{(P)}:=y\cX_{un},$$
that is, $Q(xa)=\alpha {\ii} xa+a$~ and~ $P(ya)=\beta {\ii}ya +a$, where  $a\in\cX_{un}$.
\begin{thm}
$Q$ and $P$ are graph regular self-adjoint operators on the $C^*$-algebra $\cX_{un}=\overline{\cX}.$ 
\end{thm}
\begin{proof}
We carry out  the proof for $Q$; a similar reasoning yields the assertions for $P$. Since $x\cX_{un}$ and $x^*\cX_{un}$ contain the essential ideal $\cJ_0$ (by Lemma \ref{essentialio}), $x\cX_{un}$ and $x^*\cX_{un}$ are essential  in the $C^*$-algebra $\cX_{un}.$ Therefore, by Theorem \ref{Inverse_Multiplier__GraphRegular}, $x^{-1}$ and $(x^*)^{-1}$ are graph regular operators on $\cX_{un}$, so $Q$ and $P$ are graph regular by Proposition Corollary \ref{GraphRegular_Addition_Composition}.

Further, $(x^{-1})^*=(x^*)^{-1}$ and hence $Q^*=-\alpha {\ii} I+(x^*)^{-1}$  by Theorem \ref{Inverse_Multiplier__GraphRegular} and Proposition \ref{Adjoint__Addition_Compositon} (2).
From the first relation of (\ref{xyrel1})
it follows that $-\alpha {\ii} I+(x^*)^{-1}=\alpha {\ii} I+x^{-1}$. Hence $Q=Q^*$, that is, $Q$ is self-adjoint.
\end{proof}

The operators $Q$ and $P$ are not regular on $\cX_{un}$, since neither $x\cX_{un}$ nor $y\cX_{un}$ is dense  in $\cX_{un}$. Note that the corresponding restrictions of $Q$ and $P$ are affiliated with the essential ideal $\cJ= {\mathcal K}(L^2(\R))$ of the $C^*$-algebra $\cX_{un}$.

  \subsection{Unbounded Toeplitz operators}\label{unboundedToeplitz}

Let $ L^2(\Circle)$ be the Hilbert space  of square integrable functions on the unit circle $\Circle$ with  scalar product  
\begin{align*}
\SP{f}{g} &:= \int_0^1\ol{f(e^{2\pi \ii t})}g(e^{2\pi \ii t})\, dt \quad f,g\in L^2(\Circle),
\end{align*}
and let $P$ denote the projection of $ L^2(\Circle)$ on the closed subspace $H^2(\Circle)$ generated by $\{z^n:=e^{2\pi \ii tn} | n\in \N_0\}$. For $\phi\in L^\infty(\Circle)$ the Toeplitz operator $T_\phi$ is the bounded operator on the Hilbert space $H^2(\Circle)$ is defined by $T_\phi f:=P\phi f$,  $f\in H^2(\Circle)$. 

The $C^*$-algebra generated by the unilateral shift  $S:=T_z$ is the Toeplitz algebra 
\begin{align*}
\cT := \{T_\phi | \phi\in C(\Circle)\}\dotplus\Komp(H^2(\Circle)).
\end{align*}

Our aim is to construct a class of examples of graph regular (unbounded) Toeplitz operators on the $C^*$-algebra $\cT$. Let $p,q\in\C[z]$\,  be relatively prime polynomials such that $q$ has no zeros in the open unit disc $\Disc$. Then the Toeplitz operator with rational symbol $p/q$ is defined by
\begin{align*}
\Def(T_{p/q}) &:= \{f\in\Hil^2(\Circle) | \frac{p}{q}f\in\Hil^2(\Circle)\}, \quad T_{p/q}f := \frac{p}{q}f \quad (f\in\Def(T_{p/q})),
\end{align*}
Since $T_{p/q}$ is  a multiplication operator, $T_{p/q}$ is a closed densely defined operator  on the Hilbert space $\Hil^2(\Circle)$.

\begin{thm}\label{Toeplitz_AssociatedAffiliated}
Suppose that $p,q$ are relatively prime polynomials such that $q$ has no zero in the open unit disc. Then the Toeplitz operator $T_{p/q}$ is associated with the Toeplitz algebra $\cT$. Further, $T_{p/q}$ is affiliated with the Toeplitz algebra if and only if in addition $q$ has  no zero on the unit circle.
\end{thm}
\begin{proof}
Since  $q$ has no zero in $\Disc$, 
$q$ is an outer function (see e.g. \cite{RR85}).

Now we use an argument from \cite[Section 3]{Sar2}. Since $p$ and $q$ are relatively prime, we  have $|p|^2+|q|^2>0$ on the closed unit disc $\ol{\Disc}$. Therefore, by the Riezs-Fej\'er Theorem \cite{RR85}, there exists a polynomial $r\in \C[z]$ such that $r$ has no zero in $\ol{\Disc}$ and $|p|^2+|q|^2=|r|^2$ on $\Circle$. Let $f:=q/r$ and $g:=p/r$. Then $f$ and $g$ are continuous and in the unit ball of $H^\infty(\Circle)$, $f$ is outer, $|f|^2+|g|^2=1$ on $\Circle$. Upon multiplying $r$ by some constant of modulus one we can assume  that $f(0)>0$.

From \cite[Proposition 5.3]{Sar1} it follows that $\Def(T_{p/q})=fH^2(\Circle)$ and $T_{p/q}=T_gT_f^{-1}$. Moreover, $T_{p/q}^*=(T_f^{-1})^*T_g^*=T_{\ol{f}}^{-1}T_{\ol{g}}$. Using these facts we compute
\begin{align*}
1+T_{p/q}^*T_{p/q} &= 1+T_{\ol{f}}^{-1}T_{\ol{g}}T_gT_f = T_{\ol{f}}^{-1}(T_{\ol{f}}T_f+T_{\ol{g}}T_g)T_f^{-1} = T_{\ol{f}}^{-1}(T_{|f|^2}+T_{|g|^2})T_f^{-1}\\
&= T_{\ol{f}}^{-1}T_f^{-1} = (T_fT_{\ol{f}})^{-1},\\
1+T_{p/q}T_{p/q}^* &= 1+T_gT_f^{-1}T_{\ol{f}}^{-1}T_{\ol{g}} = 1+T_g(T_{\ol{f}}T_f)^{-1}T_{\ol{g}} = 1+T_g(1-T_{\ol{g}}T_g)^{-1}T_{\ol{g}}\\
&= 1+(1-T_gT_{\ol{g}})^{-1}T_gT_{\ol{g}} = (1-T_gT_{\ol{g}})^{-1}.
\end{align*}
Hence\, $a_{T_{p/q}}=T_fT_{\ol{f}}$\, and\, $a_{T_{p/q}^*}=I-T_gT_{\ol{g}}$\, are in $\mathcal{T}$. Further,
\begin{align*}
b_{T_{p/q}} = T_{p/q}A_{T_{p/q}} = T_gT_f^{-1}T_fT_{\ol{f}} = T_gT_{\ol{f}} \in \mathcal{T}.
\end{align*}
Since\, $a_{T_{p/q}}, a_{T_{p/q}^*}, b_{T_{p/q}} \in \cT$,  $T_{p/q}$ is associated with the $C^*$-algebra $\mathcal{T}$.

Suppose now $q$ has a zero at some $\lambda\in\Circle$. Then $a$ has a zero at $\lambda$ as well. For $z\in\Circle$ let $\omega_z$ be the character on $\mathcal{T}$ given by
\begin{align}\label{characterz}
\omega_z(T_\phi+ K)=\phi(z) \quad  (\phi\in C(\Circle), K\in\Komp(H^2(\Circle)).
\end{align}
If $T_\phi+K\in\mathcal{T}$, then $T_fT_{\ol{f}}(T_\phi+K)=T_{|f|^2\phi}+\tilde{K}$ for some $\tilde{K}\in\Komp(\Hil^2(\Circle))$. Hence
\begin{align*}
\omega_\lambda(a_{T_{p/q}}(T_\phi+K)) &= \omega_\lambda(T_{|f|^2\phi}+\tilde{K}) = |f(\lambda)|^2\phi(\lambda) = 0.
\end{align*}
Therefore, $a_{T_{p/q}}\mathcal{T}$ is not dense in $\mathcal{T}$ and hence $T_{p/q}$ is not affiliated with $\mathcal{T}$. 

On the other hand, if $q$ has no zero on $\Circle$, then $p/q\in C(\Circle)$ and  hence $T_{p/q}\in\mathcal{T}$, so  in particular, $T_{p/q}$ is affiliated with $\cT$.
\end{proof}
The simplest interesting example is the following.
\begin{exa}
Set $p(z)=1$ and $q(z)=1-z$, so  that $p/q=1/(1-z)$. Then, by Theorem \ref{Toeplitz_AssociatedAffiliated}, $T_{1/(1-z)}$ associated with $\cT$, but $T_{1/(1-z)}$ is  not affiliated with $\mathcal{T}$. In fact, $T_{1/(1-z)}=(I-S)^{-1}$.
\end{exa}

  \subsection{Heisenberg group}\label{Heisenberg}

Let $H$ be the 3-dimensional Heisenberg group, that is, $H$ is the Lie group whose differential manifold is the vector space  $\R^3$ and whose multiplication is given by
$$
(x_1,x_2,x_3)(x^\prime_1,x^\prime_2,x^\prime_3):=(x_1+x^\prime_1,x_2+x^\prime_2,x_3+x^\prime_3+\frac{1}{2}(x_1x^\prime_2- x^\prime_1x_2)).
$$
The $C^*$-algebra $C^*(H)$ of the Lie group $H$ was described in \cite{ludwigtour}. We briefly repeat this result.
First we recall that $C^*(H)$ is defined as the  completion of $L^1(H)$ with respect to the norm
$$
\|f\|= \sup \, \{ \|\pi^U
(f)\|: U~ {\rm unitary~ representation~ of}~~ H\}.
$$
where $\pi^U$ is the $*$-representation of $L^1(H)$ associated with $U$, that is,
$$
\pi^U(f):=\int_{\R^3} U(x_1,x_2,x_3) f(x_1,x_2,x_3)\, dx_1dx_2dx_3, \quad f\in L^1(H).
$$
The irreducible unitary  representations of  $H$ consist of a  series $U_\lambda$,   $\lambda\in  \R^\times$, of infinite dimensional representations  acting on $L^2(\R)$ and of a series $U_{a}$, $a\in \R^2$, of one dimensional representations. For $(x_1,x_2,x_3)\in H$, these  representations act as
\begin{align*}
(U_\lambda(x_1,x_2,x_3)\xi)(s)&=e^{-2{\pi \ii \lambda} ( x_3 +\frac{1}{2} x_1x_2 + sx_2)} \xi(s-x_1), \quad \xi\in L^2(\R),\, \lambda\in  \R^\times,\\
U_a(x_1,x_2,x_3)&=e^{-2{\pi \ii } (a_1 x_1 +a_2x_2)}, \quad a=(a_1,a_2)\in \R^2.
\end{align*}
The Lie algebra of $H$   has a  basis $\{X,Y,Z\}$ with commutation relations\, $$[X,Y]=Z,\quad [X,Z]=[Y,Z]=0$$
and we have  $dU_\lambda({\ii} Z)= 2\pi \lambda I$ and $dU_a({\ii} Z)=0$. 

Now let $\CAlg{F}$ be the $C^*$-algebra  of all operator fields $F=(F(\lambda);\lambda \in \R)$ satifying the following conditions:\\
$(i)$ $F(\lambda)$ is a compact operator on $L^2(\R)$ for each $\lambda \in \R^\times$,\\
$(ii)$ $F(0)\in C_0(\R^2)$,\\
$(iii)$ $\R^\times \ni \lambda \to F(\lambda)\in \Be(L^2(\R)$ is norm continuous,\\
$(iv)$ $\lim_{\lambda \to \infty}~ \|F(\lambda)\|=0$.\\
Let $\eta$ be a fixed  function of the Schwartz space $\cS(\R)$ of norm one in $L^2(\R)$. For $\xi\in L^2(\R)$, let $P_\xi$ denote the projection on the one dimensional subspace $\C\cdot \xi$.  Then for $h\in C_0(\R^2)$ and $\lambda \in \R^\times:=\R\backslash \{0\}$, the operator $\nu_\lambda(h)$ is defined by 
\begin{align}\label{nudef}
\nu_\lambda(h):=\int_{\R^2}\, \hat{h}(x_1,x_2)P_{\eta(\lambda;x_1,x_2)} |\lambda|^{-1} dx_1 dx_2,
\end{align}
where  $\hat{h}$ denotes the Fourier transform of $h$ and
$$
\eta(\lambda;x_1,x_2)(s):=|\lambda|^{1/4} e^{2\pi i x_1 s} ~\eta\big( |\lambda|^{1/2}(s+x_2\lambda^{-1})
\big),\quad x_1,x_2,s\in \R.
$$
By Proposition 2.14 in \cite{ludwigtour}, we have
\begin{align}\label{limitnu0}
\lim_{\lambda\to 0}\, \|\nu_\lambda (h)\|=\hat{h}\|_\infty\quad{\rm for} \quad h\in C_0(\R^2).
\end{align}

Then, according to Theorem 2.16 in \cite{ludwigtour}, the $C^*$-algebra $C^*(H)$ is the $C^*$-subalgebra of $C^*(H)$ formed by all operator fields $F\in \CAlg{F}$ such that
\begin{align}\label{limitnu}
\lim_{\lambda \to 0}~ \|F(\lambda)-\nu_\lambda(F(0))\|=0,
\end{align}
where $\nu_\lambda:C_0(\R^2) \to \CAlg{F}$ is defined by (\ref{nudef}), and  for $c\in C^*(H)$, we have  $F(c)(\lambda)=\pi_{U_\lambda}$,  $\lambda \in \R^\times$, and $F(c)(0)(a)=\pi_{U_a}(c)$,  $a\in \R^2$.

On the other hand, it was proved in \cite{wor92} that the Lie algebra generators $X,Y,Z$ act as skew-adjoint regular operators on the $C^*$-algebra $C^*(H)$. 

We show that the range $\cR({\ii} Z)$ is essential in $C^*(H)$. Assume that $G(\lambda)\in  C^*(H)$ and $G(\lambda) \in \cR({\ii} Z)^\bot$. Since $dU_\lambda({\ii} Z)= 2\pi \lambda I$ for $\lambda \in \R^\times$,  $\cR({\ii} Z)$ contains all vector fields $F(\lambda)\in \cF$ of compact support contained in $\R^\times$. This implies that $G(\lambda)=0$ on  $\R^\times$. Therefore, $\lim_{\lambda\to 0}\, \nu_\lambda(G(0))=0$ by (\ref{limitnu}) and hence $\widehat{G(0)}=0$ by (\ref{limitnu0}), so  $G(0)\in C_0(\R^2)$ is zero. Thus $G=0$ in $C^*(H)$ which proves that  $\cR({\ii} Z)$ is essential. Further, ${\ii} Z$ is self-adjoint, so $({\ii}Z)^{-1}$ is by Proposition \ref{Graph_tadj}.

Since ${\ii} Z$ is graph regular, so is  $({\ii}Z)^{-1}$  by Proposition \ref{Inverse_Multiplier__GraphRegular}. Note that $({\ii}Z)^{-1}$  is  not regular, because $dU_a({\ii} Z)=0$ for $a\in \R^2$ and hence $({\ii}Z)^{-1}$ is not densely defined.   
\begin{thm}
$({\ii}Z)^{-1}$ is a graph regular self-adjoint operator on the $C^*$-algebra $C^*(H)$.
\end{thm}

\bibliographystyle{amsalpha}

\bigskip

\author{K. Schm\"udgen};
\address{Universit\"at Leipzig, Mathematisches Institut, Augustusplatz 10/11, D-04109 Leipzig, Germany};
\email{E-Mail: schmuedgen@math.uni-leipzig.de}
\smallskip

\author{R. Gebhardt};
\address{Max Planck Institute for Mathematics in the Sciences, Inselstra\ss e 22, D-04103 Leipzig, Germany};
\email{E-Mail: rene.gebhardt@mis.mpg.de}

\end{document}